\documentclass[12pt]{elsart}
\usepackage{amssymb}
\usepackage{color}
\usepackage{graphicx}\usepackage{subfigure}
\usepackage{amsmath}
\usepackage{amsfonts}

\usepackage[colorlinks=true]{hyperref}
\usepackage[title]{appendix}
\usepackage{latexsym}
\usepackage{psfrag}
\usepackage{multirow}

\allowdisplaybreaks

\usepackage{paralist}

\def\pt{\partial}\def\al{\alpha}\def\dfr#1#2{\displaystyle{\frac{#1}{#2}}}
\def\la{\lambda}

\renewcommand{\vec}[1]{\mbox{\boldmath \small $#1$}}

 \def\Ze{\mathbb{Z}}

 \def\bga{\begin{array}} \def\eda{\end{array}}
   
  \def\pt{\partial}

 \def\la{\lambda} 
  \def\al{\alpha}

\def\dfr#1#2{\displaystyle{\frac{#1}{#2}}} 

\def\Ronu#1{\uppercase\expandafter{\romannumeral#1}}

\renewcommand{\qed}{\hfill \nobreak \ifvmode \relax \else
      \ifdim\lastskip<1.5em \hskip-\lastskip
      \hskip1.5em plus0em minus0.5em \fi \nobreak
      \vrule height0.75em width0.5em depth0.25em\fi}

\newtheorem{example}{Example}[section]

\newtheorem{remark}{Remark}[section]
\numberwithin{equation}{section}
\numberwithin{figure}{section}
\numberwithin{table}{section}
\numberwithin{thm}{section}

\newenvironment{proof}[1][Proof]{\begin{trivlist}
\item[\hskip \labelsep {\bfseries #1}]}{\end{trivlist}}
\renewcommand{\qed}{\hfill \nobreak \ifvmode \relax \else
      \ifdim\lastskip<1.5em \hskip-\lastskip
      \hskip1.5em plus0em minus0.5em \fi \nobreak
      \vrule height0.75em width0.5em depth0.25em\fi}

\begin{document}
\begin{frontmatter}
\title{A direct Eulerian GRP scheme for spherically symmetric general relativistic hydrodynamics}

\author{Kailiang Wu},
\ead{wukl@pku.edu.cn}
%\address{School of Mathematical Sciences, Peking University, Beijing 100871, P.R. China}
\author[label2]{Huazhong Tang}
\ead{hztang@math.pku.edu.cn}
\thanks[label2]{Corresponding author. Tel:~+86-10-62757018;
Fax:~+86-10-62751801.}
\address{HEDPS, CAPT \& LMAM, School of Mathematical Sciences, Peking University, Beijing 100871, P.R. China}

\date{\today{}}
\maketitle

\begin{abstract}
The paper proposes a second-order accurate direct Eulerian generalized Riemann problem (GRP) scheme for the spherically
symmetric general relativistic hydrodynamical (RHD) equations and
a second-order accurate discretization for the  spherically
symmetric Einstein (SSE) equations.
The former is directly using the Riemann invariants and the Runkine-Hugoniot jump conditions to analytically resolve the left and right nonlinear waves of the local GRP in the Eulerian formulation together with the local change of the metrics to obtain the limiting values of  the time derivatives of the conservative variables along the cell interface and the numerical flux for the GRP scheme. While the latter utilizes the energy-momentum tensor obtained in the GRP solver to evaluate the fluid variables in the  SSE equations and keeps the continuity of the metrics at the cell interfaces.  Several numerical experiments show that the GRP scheme can achieve second-order accuracy and high resolution, and is effective for spherically
symmetric general RHD  problems.
\end{abstract}

\end{frontmatter}

%\thispagestyle{plain} \markboth{ } {}
%%%%%%%%%%%%%%%%%%%%%%%%%%%%%%%%%%%%%%%%%%%%%%%%%%%% 5

\section{Introduction}
\label{sec:intro}

Many fields such as high-energy astrophysics etc. may involve flows at speeds close to the speed of light or influenced by large gravitational potentials such that the relativistic effect should be taken into account.
Relativistic flows appear in numerous astrophysical phenomena, from stellar to galactic
scales, e.g. super-luminal jets, gamma-ray bursts, core collapse super-novae, coalescing
neutron stars, formation of black holes and so on.

The governing equations of the relativistic hydrodynamics (RHD) are highly nonlinear so that their analytical treatment is extremely difficult.
A primary and powerful approach to understand the physical mechanisms in RHDs is  numerical simulations. The pioneering numerical work may date back to  %the sixteenth century and May and White's
the finite difference code by May and White with the artificial viscosity technique for spherically symmetric general RHD equations in the Lagrangian coordinate \cite{May1966,May1967}. Wilson first attempted to numerically solve multi-dimensional RHD equations in the Eulerian coordinate  by using the finite difference method with the artificial viscosity technique \cite{Wilson:1972}, which was systematically introduced in \cite{Wilson2003}.
Since 1990s, the numerical study of the RHDs began to attract considerable attention, and various modern shock-capturing methods based on exact
or approximate Riemann solvers have been developed for the RHD equations, the
readers are referred to the early review articles \cite{MME:2003,Font2008} and more recent works on numerical methods for the RHD equations in \cite{WuTang2014,WuTang2015}.
Recently, second-order accurate direct Eulerian generalized Riemann problem
(GRP) schemes were developed for both 1D and 2D special RHD equations \cite{Yang-He-Tang_GRP-RHD1D,Yang-Tang_GRP-RHD2D} and the third-order accurate  extension to the 1D case was also presented in \cite{WuYangTang2014}.

The GRP scheme, as an analytic high-order accurate extension of the Godunov method, was originally
devised for non-relativistic compressible fluid dynamics \cite{Ben-Falcovitz:1984}, by utilizing a piecewise linear
function to approximate the ``initial'' data and then analytically resolving a local GRP at each
interface to yield numerical flux, see the comprehensive description in \cite{Ben-Falcovitz:book}.
There exist two versions of
the original GRP scheme: the Lagrangian and Eulerian. The Eulerian version
is always derived by using the Lagrangian framework with a transformation, which is quite
delicate, particularly for the sonic case and multi-dimensional application.
To avoid those difficulties,
second-order accurate direct
Eulerian GRP schemes were respectively developed for the shallow water equations \cite{Li-GXChen},
the Euler equations \cite{Ben-Li-Warnecke}, and a more general weakly coupled system \cite{Ben-Li:RI:GRP} by directly resolving the local GRPs in the Eulerian formulation via the Riemann invariants and Rankine-Hugoniot jump conditions.
A recent comparison of the GRP scheme with the
gas-kinetic scheme  showed the good performance of the GRP solver
for some inviscid flow simulations \cite{LiLiXu2011}.
Combined with the moving mesh method \cite{TTang:2003}, the adaptive
direct Eulerian GRP scheme was  developed in \cite{Han-Li-TangJCP} with improved resolution as well as
accuracy. The accuracy and performance of the adaptive GRP scheme were further studied in \cite{Han-Li-TangCiCP} in simulating 2D
complex wave configurations formulated with the 2D Riemann problems of non-relativistic Euler equations.
Recently, the adaptive GRP scheme was also extended to unstructured triangular meshes \cite{LiZhang2013}.
The third-order accurate extensions of the direct Eulerian GRP scheme were studied for 1D and 2D non-relativistic Euler equations  in \cite{WuYangTang2013} and the general 1D hyperbolic balance laws in \cite{QianLiWang2013}.

%The aim of the paper is to develop the direct Eulerian GRP scheme for the general RHD equations. It is nontrivial and much more technical than the special relativistic case.
The aim of the paper is to develop a second-order accurate direct Eulerian GRP scheme for spherically symmetric general RHD equations.
The traditional Godunov-type schemes based on exact or approximate Riemann solvers can be extended to the general RHD equations from the special RHD case
through a local change of coordinates in terms of that the spacetime metric is locally Minkowskian \cite{Pons1998}.
Similar idea can be found in developing the so-called locally
inertial Godunov method for spherically symmetric general RHD equations \cite{VoglerTemple2012}.
However, such approach cannot be used to
 develop the direct Eulerian GRP scheme for the general RHD equations,
because %the  GRP scheme for the general RHD equations
%cannot be directly obtained  by a local change of coordinates and
it is necessary to resolve the local GRP
together with the local change of the metrics taken into account.
Moreover, the metrics should be approximately obtained at the cell interface by   accurate scheme for the SSE equations to  keep the continuity of the approximate metric functions.
In short,  developing the GRP scheme for the general RHD equations
 is not trivial and  much more technical than the special relativistic case.
%
%{\color{red}The key step in developing the GRP scheme is to derive the approximate states in numerical fluxes
%by resolving the local GRP at each cell interface. To this end, the initial data will be reconstructed by using a piecewise linear reconstruction, and the  first-order time derivatives of fluid variables are gotten by directly and analytically
%resolving the local GRP in the Eulerian formulation via two main ingredients, i.e. the Riemann invariants and Rankine-Hugoniot jump conditions.
%Based on the approximate energy-momentum tensor given by the GRP solver, the metrics will be approximately obtained at the cell interface by a second-order accurate scheme for the spherically symmetric Einstein equations to  keep the continuity of the approximate metric functions.}
%, their discrete values will be defined at the cell interfaces.

The paper is organized as follows. Section \ref{sec:GoEq} introduces the
governing equations of general RHDs in spherically symmetric spacetime and corresponding Riemann invariants as well as their total differentials.
The second-order accurate direct Eulerian GRP scheme is developed in Section \ref{sec:scheme}.
The outline of the scheme is first given in Section \ref{sec:outline}. Then the local GRPs are analytically resolved in Section \ref{sec:resoluGRP},
where Sections \ref{sec:resoluRare} and \ref{sec:resoluShock} resolve the rarefaction and shock waves by using the Riemann invariants and the Rankine-Hugoniot jump conditions, respectively,  Section \ref{sec:PuPt} concludes the the limiting values of the time
derivatives of the conservative variables at the ``initial'' discontinuous point along the cell interface for both nonsonic and sonic cases,
and Section \ref{sec:acoustic}   discusses the acoustic case. Several numerical experiments are conducted in Section \ref{sec:experiments} to demonstrate the performance and accuracy of the proposed GRP scheme. {Section \ref{sec:conclude}} concludes the paper with several remarks.

%
%\section{Preliminaries and notations}
%\label{sec:PreNota}
%
%This section introduces the governing equations of the spherically symmetric general relativistic hydrodynamics (SSGRHD) and corresponding
%Riemann invariants as well as their total differentials.
%

\section{Governing equations}
\label{sec:GoEq}

The general RHD equations \cite{Font2008} consist of the local conservation laws of the current density $J^{\mu}$ %(the continuity equation)
and the stress-energy tensor $T^{\mu \nu}$  %(the Bianchi identities):
\begin{align}
&
\nabla _\mu  J^{\mu}  = 0, \label{eq:Jmu} \\
&
\nabla _\mu  T^{\mu \nu }  = 0, \label{eq:Tmunu}
\end{align}
where the indexes $\mu$ and $\nu$ run from 0 to 3, and $\nabla _\mu$ stands for the covariant derivative associated with the
four-dimensional spacetime metric $g_{\mu\nu}$, that is,
the proper spacetime distance between any two points in the
four-dimensional spacetime can be measured by the line element  $ds^2 = g_{\mu\nu} d x^{\mu} dx^{\nu}$.
The current density is given by $J^{\mu} = \rho_0 u ^{\mu}$,  where $ u ^{\mu}$ represents
the fluid four-velocity  and $\rho_0$ denotes the proper rest-mass density.
 %in a locally inertial reference frame.
 The stress-energy
tensor for an ideal fluid is defined by
\begin{align*}
 T^{\mu \nu }  = (\rho + p) u^{\mu} u^{\nu} + p g^{\mu \nu}, %\label{eq:TmunuDef}
\end{align*}
in which $\rho$ and $p$  denote the rest energy density (including rest-mass) in the fluid frame and
the pressure, respectively, and $g^{\mu\la}g_{\la\nu}=\delta^{\mu}_{\nu}$ with $\delta^\mu_\nu$ denoting the Kronecker symbol.
The rest energy density $\rho$
can be expressed in terms of the rest-mass density $\rho_0$ and the internal energy $e$ as $\rho=\rho_0 (c^2+e)$, where $c$ denotes the speed of light  in vacuum.

An additional equation for the thermodynamical variables, i.e. the so-called equation of state,
is needed to close the system \eqref{eq:Jmu}--\eqref{eq:Tmunu} for a fixed spacetime.
This paper focuses on the
equation of state describing barotropic fluids %[xxx]
\begin{equation}\label{eq:EOS}
p=p(\rho),
\end{equation}
where $p(\rho)$ is a function of $\rho$ and satisfies
\begin{equation}\label{eq:EOS-thz01}0<\frac{dp}{d\rho}=p'(\rho)<1.\end{equation}
It is worth noting that the equations \eqref{eq:Tmunu}--\eqref{eq:EinsteinEqs} form a close system if $g^{\mu\nu}$ is given.
In the general theory, the Einstein gravitational field equations relate the curvature of spacetime to the distribution of mass-energy in the following form
\begin{equation}\label{eq:EinsteinEqs}
R^{\mu\nu} - \frac{1}{2} g^{\mu\nu} R = \kappa T^{\mu\nu},%\quad \kappa = \frac{8 \pi G}{c^4},
\end{equation}
where $\kappa=\frac{8 \pi {\cal G} }{c^4}$ is Einstein coupling constant, $\cal G$ is Newton's gravitational constant,
and $R^{\mu\nu}$ and $R$ denote the Ricci  tensor and the  scalar curvature, respectively.
For the sake of  convenience, units in which the speed of light $c$ and Newton's gravitational constant $\cal G$ are equal to one will be used
throughout the paper.

The general RHD system in spherically symmetric spacetime is a simple
but  good ``approximate'' model in investigating several astrophysical phenomena, e.g. gamma-ray bursts, spherical accretion onto compact objects, and stellar
collapse, etc. Its numerical methods have also received lots of attentions, see e.g. \cite{May1966,Gourgoulhon1991,RomeroIbanez1996,Yamada1997,Liebendorfer2002,Liebendorfer2004,OConnor2010,Radice2011,Guzman2012,VoglerTemple2012,Park2013}.
The spherically symmetric gravitational metrics in standard Schwarzschild coordinates  are  given by   the line element \cite{VoglerTemple2012}
\begin{equation} \label{eq:spacetime}
ds^2 = -B(t,r)dt^2 + \frac{1}{A(t,r)} dr^2 + r^2 ( d \theta^2 + \sin^2 \theta d \phi^2 ).
\end{equation}
where $B(t,r)$ is called the {\em lapse function}, $(t,r)$ are temporal and radial coordinates, and $\vec x=(x^0,x^1,x^2,x^3)=(t,r,\theta,\phi)$ is the spacetime coordinate system.
This paper is only concerned with the numerical method for the system \eqref{eq:Tmunu}--\eqref{eq:EinsteinEqs} in spherically symmetric spacetime \eqref{eq:spacetime}. Assume that the spherically symmetric metrics are Lipschitz and the stress-energy tensor is bounded in sup-norm, then the system \eqref{eq:Tmunu}--\eqref{eq:EinsteinEqs} is weakly equivalent to the following system \cite{GroahTemple2007}
\begin{align}
  \label{eq:RHD_1}
& \frac{{\partial \vec U}}{{\partial t}} + \frac{ \partial \big( \sqrt{AB} \vec F( \vec U)\big) }{{\partial r}} = \vec S(r,A,B,\vec U),\\[2mm]
\label{eq:RHD_1A}
& \frac{{\partial M}}{{\partial r}} = \frac12 \kappa r^2 {\cal T}^{00}, \\[2mm]
\label{eq:RHD_1B}
& \frac{1}{B} \frac{{\partial B}}{{\partial r}} = \frac{{1 - A}}{{Ar}} + \frac{{\kappa r}}{A} {\cal T}^{11},
\end{align}
where
\begin{align*}
  %\label{eq:UFS}
& \vec U = ({\cal T}^{00},{\cal T}^{01} )^{\rm T},\quad \vec F =  ( {\cal T}^{01}, {\cal T}^{11} )^{\rm T},\\
& \vec S = - \sqrt{AB}\bigg(  \frac{2}{r} {\cal T}^{01},
\frac{2}{r} {\cal T}^{11} + \frac{1-A}{2Ar} ( {\cal T}^{00} - {\cal T}^{11} ) + \frac{\kappa r}{A} \big( {\cal T}^{00} {\cal T}^{11} - ({\cal T}^{01})^2 \big) - \frac{2p}{r}  \bigg)^{\rm T},
\end{align*}
and the mass function $M$ is related to $A$ by $A=1-2M/r$. Here
\[
{\cal T}^{00}  = \left( {\rho  + p} \right)W^2  - p,\quad {\cal T}^{01}  = \left( {\rho  + p} \right)W^2 v,\quad {\cal T}^{11}  = \left( {\rho v^2  + p} \right)W^2 ,
\]
are the stress-energy tensor in locally flat Minkowski spacetime,
%``Minkowski stress-energy tensor'' %stresses
related to $T^{\mu \nu}$ by
$$
{\cal T}^{00} = BT^{00},\quad {\cal T}^{01 } = \sqrt{\frac{B}{A}} T^{01},\quad {\cal T}^{11} = \frac{1}{A} T^{11},
$$
and $W=1/\sqrt{1-v^2}$ is the Lorentz factor with the velocity
$$
v:=\frac{1}{\sqrt{AB}} \frac{u^1}{u^0}.
$$

Eq. \eqref{eq:RHD_1A} may be replaced with
\begin{align}
  \label{eq:RHD_2}
 \frac{{\partial M}}{{\partial t}} = - \frac12 \kappa r^2\sqrt{AB} {\cal T}^{01},
\end{align}
to derive another equivalent system \eqref{eq:RHD_1}, \eqref{eq:RHD_1B}, and \eqref{eq:RHD_2}.
%, see \cite{GroahTemple2007}.

The eigenvalues of the Jacobian matrix $\pt (\sqrt{AB} \vec F) /\pt \vec U$ of \eqref{eq:RHD_1} with \eqref{eq:EOS} are
\[
\lambda _ -   = \sqrt {AB} \bigg( \frac{{v - c_s }}{{1 - vc_s }}\bigg),
\quad \lambda _ +   = \sqrt {AB} \bigg( \frac{{v + c_s }}{{1 + vc_s }} \bigg),
\]
where  $c_s = \sqrt{p'(\rho)}$ denotes the local sound speed. Corresponding right eigenvectors $\vec R_{\pm}$, may be given as follows
\[
\vec R_{-}  = \left( \begin{array}{c}
 1 - vc_s  \\
 v - c_s  \\
 \end{array} \right),
 \quad \vec R_{+}  = \left( \begin{array}{c}
 1 + vc_s  \\
 v + c_s  \\
 \end{array} \right),
\]
and the inverse of the matrix $\vec R:=(\vec R_{-} ,\vec R_{+} )$ is
\[
\vec R^{ - 1}  =
\frac{{W^2 }}{{2c_s }}\left( {\begin{array}{*{20}c}
   {v + c_s }~ & ~~{ - (1 + vc_s )}  \\
   {c_s  - v}~ & ~~{1 - vc_s }  \\
\end{array}} \right).
\]
The condition \eqref{eq:EOS-thz01} implies that $\lambda_- < \lambda_+$. Thus the system \eqref{eq:RHD_1} is strictly hyperbolic.
Moreover, %it was showed in \cite{GroahTemple2007} that
both characteristic fields related to $\lambda_{\pm}$ are genuinely nonlinear if and only if the function $p(\rho)$ further satisfies \cite{GroahTemple2007}
\begin{equation}\label{eq:genuinely nonlinear}
{p''(\rho )} >  - \frac{{2p'(\rho )(1 - p'(\rho ))}}{{\rho  + p(\rho )}},
\end{equation}
which does always hold for
\begin{equation}\label{eq:EOS-thz02}
p(\rho)=\sigma^2 \rho,\quad \sigma \in (0,1).
\end{equation}
%Our GRP scheme will be derived under the conditions \eqref{eq:EOS-thz01} and \eqref{eq:genuinely nonlinear}.
The Riemann invariants $\psi _ \pm$ associated with the characteristic field $\lambda _ \pm$ can be obtained as follows \cite{GroahTemple2007}
\begin{equation}\label{eq:RIN}
\psi _ \pm   = \frac{1}{2}\ln \left( {\frac{{1 + v}}{{1 - v}}} \right) \mp \int_{}^\rho  {\frac{{\sqrt {p'(\omega )} }}{{\omega  + p(\omega )}}} d\omega,
\end{equation}
which will play a pivotal role in resolving the centered rarefaction waves
in the direct Eulerian GRP scheme for the RHD equations \eqref{eq:RHD_1}.

%\subsection{Riemann invariants}
%\label{sec:RIN}

In the smooth region, by using \eqref{eq:RHD_1A} and \eqref{eq:RHD_1B},
the RHD equations \eqref{eq:RHD_1} can  be reformed
in  the primitive variable vector $\vec V:=(\rho ,v)^{\rm T}$  as follows
\begin{equation}  \label{eq:RHD_pri}
\frac{{\partial \vec V}}{{\partial t}} + \vec J \frac{ \partial \vec V }{{\partial r}} = \vec H,\\[2mm]
\end{equation}
where
\begin{align*}
&\vec J = \frac{\sqrt{AB}}{{1 - v^2 c_s^2 }}\left( {\begin{array}{*{20}c}
   {v(1 - c_s^2 )} & {\rho  + p}  \\
   {\frac{{(1 - v^2 )^2 c_s^2 }}{{\rho  + p}}} & {v(1 - c_s^2 )}  \\
\end{array}} \right),\\[3mm]
&
\vec H =(H_1,H_2)^{\rm T}=  - \frac{{\sqrt {AB} }}{{r(1 - v^2 c_s^2 )}}\left( \begin{array}{c}
 2v(\rho  + p)\left( {1 - \frac{{\kappa r^2 (\rho  + p)}}{4A}} \right) \\
  (1 - v^2 )\left( { - 2v^2c_s^2  + \frac{{(1 - A)(1 - v^2 c_s^2 )}}{{2A }} + \frac{{\kappa  r^2 (p + \rho v^2 c_s^2 )}}{2A}} \right) \\
 \end{array} \right).
\end{align*}
By using \eqref{eq:RHD_pri}, one can derive the following differential relations of the Riemann invariants %$\psi _ \mp$
\begin{align} %\nonumber
 \frac{{D_ \pm  \psi _ \mp  }}{{Dt}} &= \frac{1}{{1 - v^2 }}\frac{{D_ \pm  v}}{{Dt}} \pm \frac{{c_s }}{{\rho  + p}}\frac{{D_ \pm  \rho }}{{Dt}}
 \label{eq:Diff_RIN}
  = \frac{1}{{1 - v^2 }}H_2  \pm \frac{{c_s }}{{\rho  + p}}H_1  = :s_ \mp ,
\end{align}
where %the differential operators
$$
\frac{{D_ \pm    }}{{Dt}} := \frac{\pt }{\pt t} + \lambda_{\pm } \frac{\pt }{\pt r},
$$
denote the total derivative operators along the characteristic curves  $\frac{d r }{d t}=\lambda_{\pm}$.

\section{Numerical scheme}
\label{sec:scheme}

\subsection{The outline of the GRP scheme}
\label{sec:outline}

This subsection gives the outline of the GRP scheme.
For the sake of simplicity,
 the equally spaced grid points $\left\{r_{j+\frac12} = \left(j+\frac12\right)\Delta r\right\}$   are used for
the spatial domain $\Omega$ and the cell is denoted by $I_{j} = [r_{j-\frac12},r_{j+\frac12}]$, $j \in \Ze^+$.
The time domain $[0, T]$ is also divided into the (non-uniform) grid $\{t_0=0, t_{n+1}=t_n+\Delta t_{n}, n\geq 0\}$
with the time step size $\Delta t_{n}$ determined by
$$
\Delta t_{n}= C_{cfl}\frac{\Delta r}{\max\limits_{j}\left\{|\lambda_-({\vec U}_{j}^n,A_{j}^n,B_{j}^n)|,
|\lambda_+({\vec U}_{j}^n,A_{j}^n,B_{j}^n)|\right\}},
$$
where ${\vec U}_{j}^n$, $A_{j}^n$, and $B_{j}^n$ approximate the values of $\vec U(t,r)$, $A(t,r)$ and $B(t,r)$
at the point $(t_n,r_{j})$, respectively, and $C_{cfl}$ is the CFL number.

Assume that the ``initial'' data at time $t=t_n$ are piecewise linear functions as follows
\begin{equation}\label{eq:InitPLF}
\left\{ \begin{array}{l}
 \vec U_h (t_n,r ) = \vec U_j^n  + \vec  \sigma _j^n (r - r_j )=:\vec U^n_{j}(r), \\[3mm]
 \left( \begin{array}{c}
 A_h (t_n,r ) \\
 B_h (t_n,r ) \\
 \end{array} \right) = \dfr{{r_{j + \frac{1}{2}}  - r}}{{\Delta r}}\left( \begin{array}{c}
 A_{j - \frac{1}{2}}^n  \\
 B_{j - \frac{1}{2}}^n  \\
 \end{array} \right) + \dfr{r - r_{j - \frac{1}{2}} }{\Delta r}  \left( \begin{array}{c}
 A_{j + \frac{1}{2}}^n  \\
 B_{j + \frac{1}{2}}^n  \\
 \end{array} \right),
 \end{array} \right.
\end{equation}
for $r \in I_j$, where $A_h (t_n,r )$ and $B_h (t_n,r )$ are continuous at cell interfaces
$r_{j\pm \frac12}$.

\noindent
{\tt Step I}. Evaluate the point values $\vec U_{j+\frac12}^{n+\frac12}$ approximating $\vec U(t_{n+\frac12},r_{j+\frac12})$ by
\begin{equation}\label{eq:TalyorGRP}
\vec U_{j + \frac{1}{2}}^{n + \frac{1}{2}}  = \vec U_{j + \frac{1}{2}}^{\mbox{\tiny RP},n}  +
\frac{{\Delta t_n }}{2} \left( {\frac{{\partial \vec U}}{{\partial t}}} \right)_{j + \frac{1}{2}}^{\mbox{\tiny GRP},n} ,
\end{equation}
where  $\vec U_{j + \frac{1}{2}}^{\mbox{\tiny RP},n}$ is the values at $r=r_{j+\frac12}$ of the solutions to the
following Riemann problem of the homogeneous hyperbolic conservation laws
\begin{equation*}
  \label{eq:RP}
\begin{cases}
\dfr{{\partial \vec U}}{{\partial t}} + \sqrt {A_{j + \frac{1}{2}}^n B_{j + \frac{1}{2}}^n } \dfr{{\partial \vec F(\vec U)}}{{\partial r}} = 0,    \qquad r>0,~t>t_n,\\[3mm]
 \vec  U(t_n,r)=
     \begin{cases}
     \vec U_{j+\frac12,L}^n := \vec U_h (t_n,r_{j+\frac12} -0 ), & r<r_{j+\frac{1}{2}},\\
     \vec U_{j+\frac12,R}^n := \vec U_h (t_n,r_{j+\frac12} +0 ), & r>r_{j+\frac{1}{2}},
     \end{cases}
\end{cases}
\end{equation*}
and $\left( \partial \vec U / \partial t \right)_{j + \frac{1}{2}}^{\mbox{\tiny GRP},n}$ is analytically derived by a second order accurate resolution of the local generalized
Riemann problem (GRP) at each point $(t_n,r_{j+\frac12})$, i.e.
\begin{equation}
  \label{eq:GRP}
\begin{cases}
\mbox{Eqs. \eqref{eq:RHD_1}},    \qquad r>0,~t>t_n,\\[2mm]
 \vec  U(t_n,r)=
     \begin{cases}
     \vec U^n_{j}(r), & r<r_{j+\frac{1}{2}},\\
     \vec U^n_{j+1}(r), & r>r_{j+\frac{1}{2}}.
     \end{cases}
\end{cases}
\end{equation}
The calculation of $\left( \partial \vec U / \partial t \right)_{j + \frac{1}{2}}^{\mbox{\tiny GRP},n}$ is one of the key elements in the GRP scheme and will be given in Section \ref{sec:resoluGRP}.

{\tt Step II}. Calculate the point values $A_{j+\frac12}^{n+\frac12}$ and $B_{j+\frac12}^{n+\frac12}$,
which  are approximation of $A(t_{n+\frac12},r_{j+\frac12})$ and $B(t_{n+\frac12},r_{j+\frac12})$, respectively, by
\begin{align*}
&
M_{j + \frac{1}{2}}^{n + \frac{1}{2}}  = M_{j + \frac{1}{2}}^n  + \frac{{\Delta t_n }}{2}\left( {\frac{{\partial M}}{{\partial t}}} \right)_{j + \frac{1}{2}}^n  = M_{j + \frac{1}{2}}^n  - \frac{{\Delta t_n }}{4}\kappa r_{j + \frac{1}{2}}^2 \sqrt {A_{j + \frac{1}{2}}^n B_{j + \frac{1}{2}}^n } {\cal T}^{01} \left( {\vec U_{j + \frac{1}{2}}^{\mbox{\tiny RP},n} } \right), \\
&
A_{j + \frac{1}{2}}^{n + \frac{1}{2}}  = 1 - 2M_{j + \frac{1}{2}}^{n + \frac{1}{2}} /r_{j + \frac{1}{2}} ,\\
&
\ln B_{j + \frac{1}{2}}^{n + \frac{1}{2}}  = \ln B_{j - \frac{1}{2}}^{n + \frac{1}{2}}  + \frac{{\Delta r}}{2}\left( {\frac{{1 - A_{j - \frac{1}{2}}^{n + \frac{1}{2}} }}{{A_{j - \frac{1}{2}}^{n + \frac{1}{2}} r_{j - \frac{1}{2}} }} + \frac{{\kappa r_{j - \frac{1}{2}} }}{{A_{j - \frac{1}{2}}^{n + \frac{1}{2}} }}{\cal T}^{11} \left( {\vec U_{j - \frac{1}{2}}^{n + \frac{1}{2}} } \right) + \frac{{1 - A_{j + \frac{1}{2}}^{n + \frac{1}{2}} }}{{A_{j + \frac{1}{2}}^{n + \frac{1}{2}} r_{j + \frac{1}{2}} }} + \frac{{\kappa r_{j + \frac{1}{2}} }}{{A_{j + \frac{1}{2}}^{n + \frac{1}{2}} }}{\cal T}^{11} \left( {\vec U_{j + \frac{1}{2}}^{n + \frac{1}{2}} } \right)} \right).
\end{align*}

{\tt Step III}. Approximately evolve the solution vector $\vec U$ at time $t_{n+1}$ of \eqref{eq:RHD_1} by a second-order accurate Godunov-type scheme
\begin{align}\nonumber
\vec U_j^{n + 1}  = &\vec U_j^n  - \dfr{{\Delta t_n }}{{\Delta r}}
\left( {\sqrt {A_{j + \frac{1}{2}}^{n + \frac{1}{2}} B_{j + \frac{1}{2}}^{n + \frac{1}{2}} } \vec F\left( {\vec U_{j + \frac{1}{2}}^{n + \frac{1}{2}} } \right)
    - \sqrt {A_{j - \frac{1}{2}}^{n + \frac{1}{2}} B_{j - \frac{1}{2}}^{n + \frac{1}{2}} } \vec F\left( {\vec U_{j - \frac{1}{2}}^{n + \frac{1}{2}} } \right)} \right) \\[2mm]
    \label{eq:evolve}
&
+ \dfr{{\Delta t_n }}{2}
\bigg( {\vec S \left( {r_{j - \frac{1}{2}} ,A_{j - \frac{1}{2}}^{n + \frac{1}{2}} ,B_{j - \frac{1}{2}}^{n + \frac{1}{2}} ,\vec U_{j - \frac{1}{2}}^{n + \frac{1}{2}} } \right)
+ \vec S\left( {r_{j + \frac{1}{2}} ,A_{j + \frac{1}{2}}^{n + \frac{1}{2}} ,B_{j + \frac{1}{2}}^{n + \frac{1}{2}} ,\vec U_{j + \frac{1}{2}}^{n + \frac{1}{2}} } \right)} \bigg).
\end{align}

{\tt Step V}. Calculate   $A_{j+\frac12}^{n+1}$ and $B_{j+\frac12}^{n+1}$ by
%which respectively approximates $A(t_{n+1},r_{j+\frac12})$ and $B(t_{n+1},r_{j+\frac12})$,
%by using
\begin{align*}
&
M_{j + \frac{1}{2}}^{n + 1}  = M_{j - \frac{1}{2}}^{n + 1}  + \frac{{\Delta r}}{2}\kappa r_j^2 {\cal T}^{00} \left( {\vec U_j^{n + 1} } \right), \quad A_{j + \frac{1}{2}}^{n + 1}  = 1 - 2M_{j + \frac{1}{2}}^{n + 1} /r_{j + \frac{1}{2}} ,\\[2mm]
&
\ln B_{j + \frac{1}{2}}^{n + 1}  = \ln B_{j - \frac{1}{2}}^{n +1}  + {\Delta r}
\left( \frac{1 - A_j^{n + 1} }{A_j^{n + 1} r_j }
 + \frac{\kappa r_j }{A_j^{n + 1}}
 {\cal T}^{11} \left( \vec U_j^{n + 1}  \right)   \right), \quad A_j^{n+1} := \frac12 \big( A_{j-\frac12}^{n+1} + A_{j+\frac12}^{n+1} \big).
\end{align*}

{\tt Step IV}.  Update the slope $\vec \sigma_j^{n+1}$ component-wisely in the local characteristic variables by
%Define
%In order to suppress possible numerical oscillations near discontinuities, we apply to $\vec %\sigma _j^{n + 1, - } $
%a slope limiter through the local characteristic variables as follows
\begin{equation}\label{eq:limiter}
\vec \sigma _j^{n + 1}  = \vec R_j {\rm minmod} \left(   \frac{\theta}{\Delta r} \vec R_j^{-1} \left(\vec U_j^{n + 1}  - \vec U_{j - 1}^{n + 1} \right),
\vec R_j^{-1} \vec \sigma _j^{n + 1, - } , \frac{\theta}{\Delta r} \vec R_j^{-1} \left(\vec U_{j+1}^{n + 1}  - \vec U_{j }^{n + 1} \right) \right),
\end{equation}
where $\vec R_j := \vec R(\vec U_j^{n+1})$,  the parameter $\theta \in [1,2)$, and
\[
\vec \sigma _j^{n + 1, - }  = \frac{1}{{\Delta r}}\left( {\vec U_{j + \frac{1}{2}}^{n + 1, - }  - \vec U_{j - \frac{1}{2}}^{n + 1, - } } \right),\quad
\vec U_{j + \frac{1}{2}}^{n + 1, - }  = \vec U_{j + \frac{1}{2}}^{\mbox{\tiny RP},n}  +
\Delta t_n \left( {\frac{{\partial \vec U}}{{\partial t}}} \right)_{j + \frac{1}{2}}^{\mbox{\tiny GRP},n}.
\]

The paper does not pay much attention to  the treatment of singularity in the source $\vec S$ of \eqref{eq:RHD_1} and the imposition of boundary conditions at the symmetric center $r=0$ for the GRP scheme, the readers are referred to \cite{LiLiuSun2009} for the details.

\subsection{Resolution of generalized Riemann problem}
\label{sec:resoluGRP}

This subsection resolves the   GRP \eqref{eq:GRP} in order to get $\left( \partial \vec U / \partial t \right)_{j + \frac{1}{2}}^{\mbox{\tiny GRP},n}$ in \eqref{eq:TalyorGRP}.
For the sake of convenience, the subscript $j$ and the superscript $n$ will be ignored
and the local GRP \eqref{eq:GRP} is transformed with a linear coordinate transformation to the ``non-local'' GRP for \eqref{eq:RHD_1} with the initial data
\begin{equation}
  \label{eq:InitialData}
 % \begin{cases}
  %  \mbox{~\eqref{eq:1D}},\quad t>0,\\
    \vec U(0,r) =
    \begin{cases}
      \vec U_L + (r-r_0)\vec U'_L,\ \ & r < r_0,\\[3mm]
      \vec U_R + (r-r_0)\vec U'_R, & r>r_0,
    \end{cases}
%  \end{cases}
\end{equation}
where $\vec U_L, \vec U_R, \vec U'_L$ and $\vec U'_R$
are corresponding constant vectors.
The notations $\vec U_{j+1/2}^{\mbox{\tiny RP},n}$ and $\left(\vec U_{t}\right)_{j+1/2}^{\mbox{\tiny GRP},n}$ will also be simply replaced with $\vec U_{*}$ and
$\left(\vec U_{t}\right)_{*}$, respectively, which
also denote the limiting states at $r=r_0$, as $t\rightarrow 0^+$.

Since both $A$ and $B$ are locally Lipschitz continuous, the initial structure of the solution $\vec U^{\mbox{\tiny GRP}}(t,r)$
to the GRP for \eqref{eq:RHD_1} with \eqref{eq:InitialData} may be determined by the solution $\vec U^{\mbox{\tiny RP}}(t,r)=\vec \omega((r-r_0)/t;
\vec U_L, \vec U_R)$ of the associated (classical) Riemann problem (RP) \cite{Ben-Falcovitz:1984,LiYu1985}
\begin{equation}
  \label{eq:InitialData-RP}
\begin{cases}
\dfr{{\partial \vec U}}{{\partial t}} + \sqrt {A_* B_* } \dfr{{\partial \vec F(\vec U)}}{{\partial r}} = 0,    \qquad t>0,\\[3mm]
    \vec U(0,r) =
    \begin{cases}
      \vec U_L,\ \ & r < r_0,\\[1mm]
      \vec U_R, & r>r_0,
    \end{cases}
\end{cases}
\end{equation}
and
\begin{equation*}
  \lim_{t \rightarrow 0^+} \vec U^{\mbox{\tiny GRP}}(t,t\la+r_0)
  =\vec\omega(\la;\vec U_L,\vec U_R),\quad r-r_0=t\la.
\end{equation*}
The local wave configuration around the singularity point
$(t,r)=(0,r_0)$ of the GRP for \eqref{eq:RHD_1} with \eqref{eq:InitialData} depends on the values
of four constant vectors and
consists of two nonlinear waves, each of which may be rarefaction  or shock wave.
Fig.~\ref{fig:wave-pattern-a} shows the schematic description of a local wave configuration:
a rarefaction wave moving to the left and a shock  to the right.
Fig.~\ref{fig:wave-pattern-b} displays corresponding local wave configuration of the RP \eqref{eq:InitialData-RP}.
In those schematic descriptions,  $\vec U_*$  denotes  the limiting state at $r=r_0$, as $t\rightarrow 0^+$,
 and $\alpha$ and $\beta$ are characteristic coordinate within the rarefaction wave and will be introduced in Section \ref{sec:resoluRare}.

\begin{figure}[!htbp]
  \centering
    \includegraphics[width=0.6\textwidth]{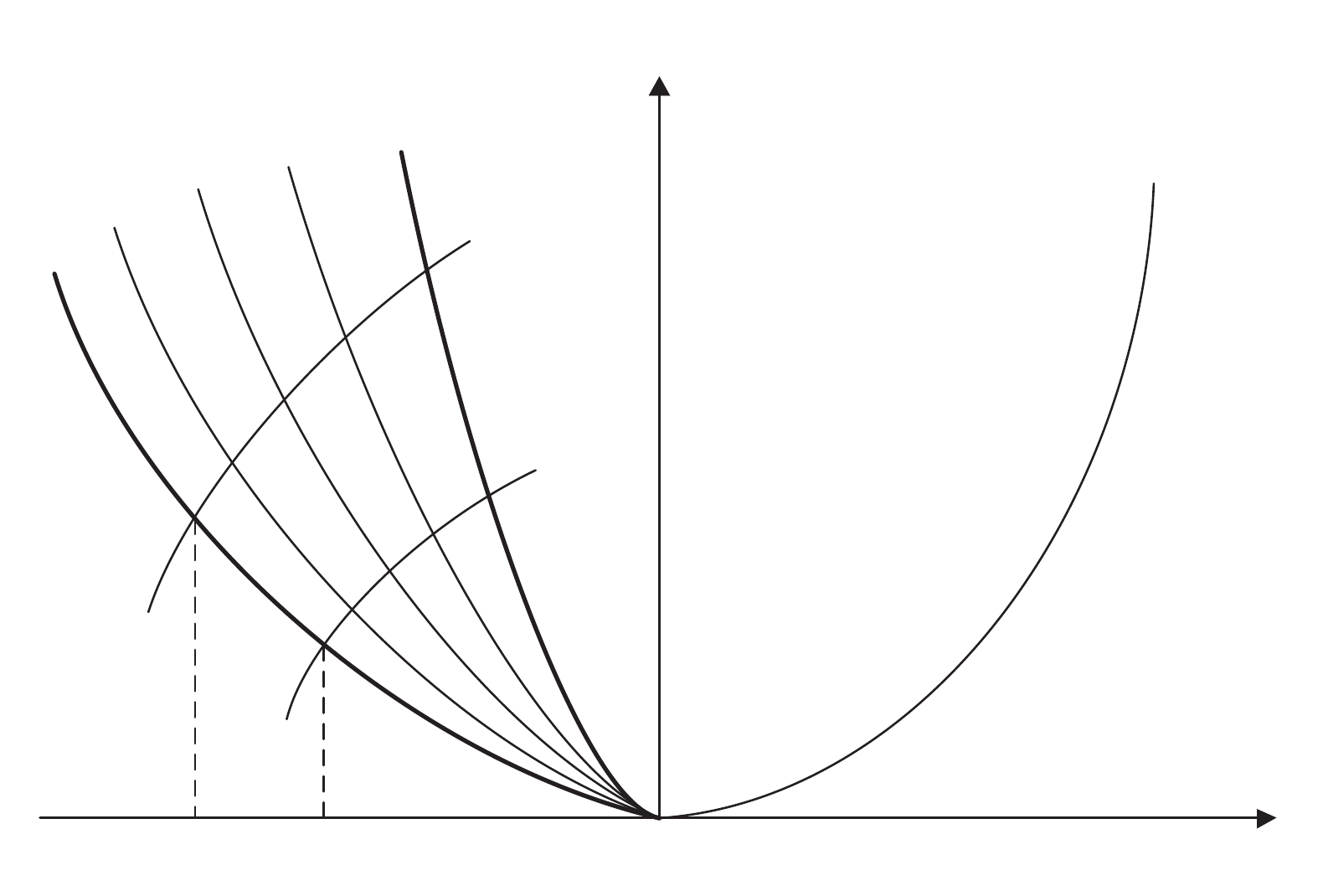}
    \put(-288,133){\small $\beta=\beta_L$} \put(-208,158){\small $\beta=\beta_*$}
    \put(-138,173){\small $t$}
    \put(-266,155){\small $\mbox{rarefaction}$}
    \put(-32,120){\small $\mbox{shock}$}
    \put(-248,4){\small $\bar{\bar \al}+r_0$}
    \put(-212,5){\small $\bar \al+r_0$}
    \put(-175,133){\small $\al = \bar{\bar \al}$}
    \put(-168,90){\small $\al = {\bar \al}$}
    \put(-275,40){\small $\vec U_L(t,r)$}
    \put(-148,5){\small $r=r_0$}
    \put(-15,5){\small $r$}
    \put(-116,90){\small $\vec U_*(t,r)$}
    \put(-60,40){\small $\vec U_R(t,r)$}
  \caption{The schematic description of a local wave configuration for
    the GRP for \eqref{eq:RHD_1} with \eqref{eq:InitialData} with $0\leq t\ll 1$.}
  \label{fig:wave-pattern-a}
\end{figure}

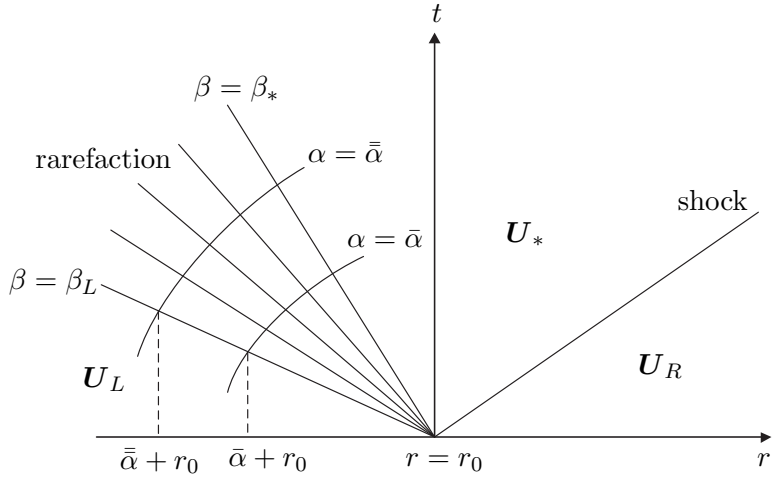
\begin{figure}[!htbp]
  \centering
    \includegraphics[width=0.6\textwidth]{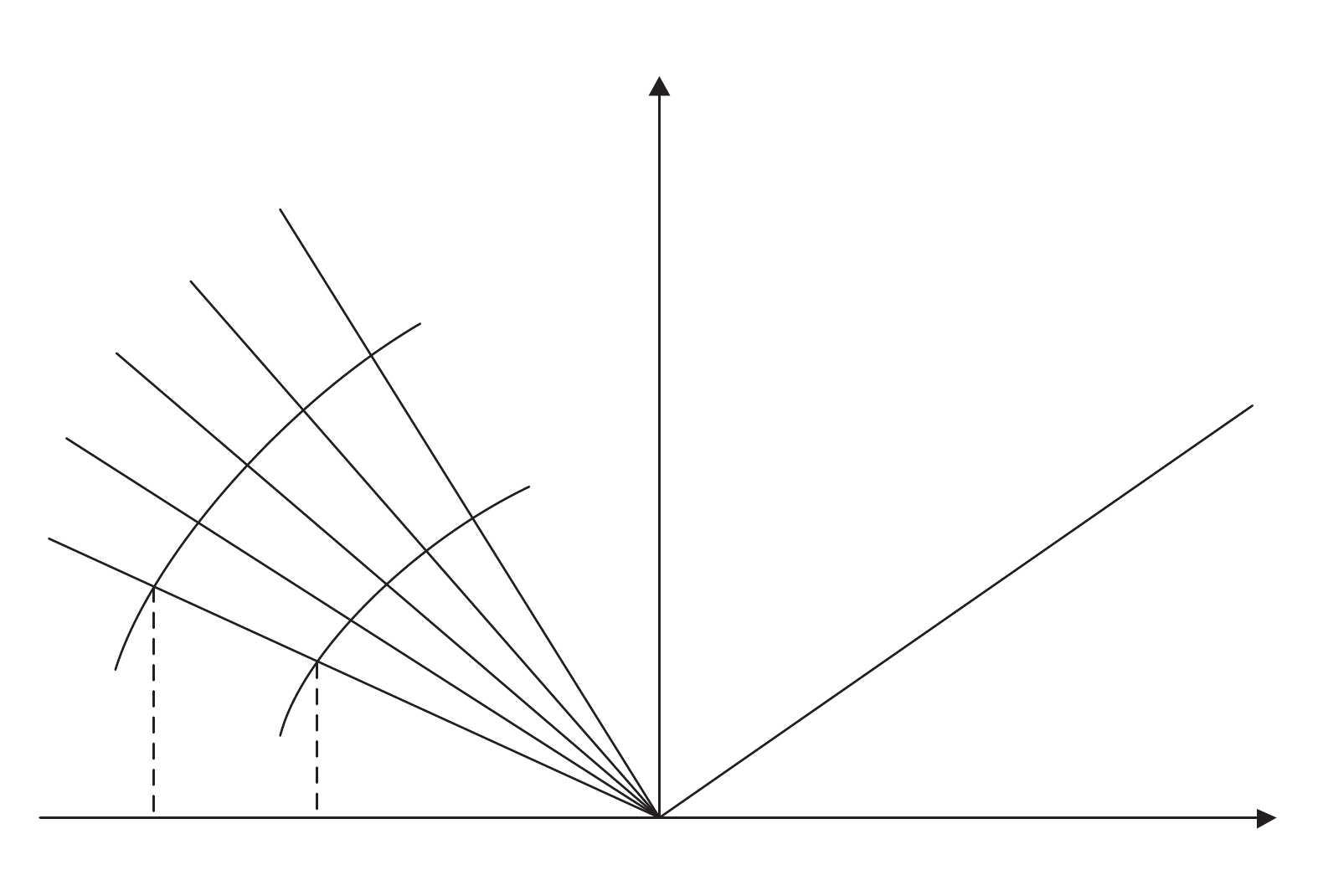}
    \put(-298,72){\small $\beta=\beta_L$} \put(-228,145){\small $\beta=\beta_*$}
    \put(-138,173){\small $t$}
    \put(-288,118){\small $\mbox{rarefaction}$}
    \put(-45,102){\small $\mbox{shock}$}
    \put(-256,4){\small $\bar{\bar \al}+r_0$}
    \put(-215,5){\small $\bar \al+r_0$}
    \put(-185,120){\small $\al = \bar{\bar \al}$}
    \put(-170,88){\small $\al = {\bar \al}$}
    \put(-270,35){\small $\vec U_L$}
    \put(-148,5){\small $r=r_0$}
    \put(-15,5){\small $r$}
    \put(-110,90){\small $\vec U_*$}
    \put(-60,40){\small $\vec U_R$}
  \caption{The schematic description of a local wave configuration for
    the associated (classical) Riemann problem \eqref{eq:InitialData-RP}.}
  \label{fig:wave-pattern-b}
\end{figure}

Although there are other local wave configurations, we will restrict our discussion to the local wave configuration shown in
Figs.~\ref{fig:wave-pattern-a} and \ref{fig:wave-pattern-b}.
%, where a rarefaction wave moves to the left and a shock moves to the right.
Other local wave
configurations can be dealt with similarly and are considered in the code.
The solutions to the GRP inside the left, intermediate and right subregions are denoted by $\vec U_L(t,r),~\vec U_*(t,r)$, and $\vec U_R(t,r)$, respectively.
For any variable $V$, which may be  $\vec U$ or the derivatives $\vec U_t$ or $\vec U_r$ etc.,
the symbols $V_L$ and $V_R$  are used to denote the limiting values of $V$ as $t \to 0^+$ in the left and right subregions adjacent to $r$-axis, respectively,
and $V_*$ is used to denote the limiting values of $V$ as $t \to 0^+$ in the intermediate subregions. The main task of the direct Eulerian GRP scheme is to form a linear
algebraic system
\begin{equation}
  \label{eq:AlgebricEqns}
  \begin{cases}
    a_L \left(\dfr{\pt \rho}{\pt t}\right)_* + b_L \left(\dfr{\pt v}{\pt t}\right)_* = d_L,\\[4mm]
    a_R \left(\dfr{\pt \rho}{\pt t}\right)_* + b_R \left(\dfr{\pt v}{\pt t}\right)_* = d_R,
  \end{cases}
\end{equation}
by  resolving the left wave and the right wave as shown in
Figure \ref{fig:wave-pattern-a}, respectively.  Solving this system gives the values
of the derivatives  $(\pt \rho / \pt t)_*$ and $(\pt v / \pt t)_*$, and closes the
calculation in \eqref{eq:TalyorGRP}.

\subsubsection{Resolution of the rarefaction wave}
\label{sec:resoluRare}

This section resolves the left rarefaction wave shown in Figure \ref{fig:wave-pattern-a} for the
GRP \eqref{eq:RHD_1} and \eqref{eq:InitialData},  and gets the first equation in \eqref{eq:AlgebricEqns}.
%Before that, we introduce some notations similar to those in \cite{Ben-Li-Warnecke}.

The relation \eqref{eq:Diff_RIN} for the Riemann invariant $\psi _ -$ will be used to resolve the left rarefaction wave by
tracking the directional derivatives $\frac{D_ -  \psi _ -  }{Dt}$ in the rarefaction fan. For this purpose, a
local coordinate transformation is first introduced  within the rarefaction wave, i.e. the characteristic coordinates, similar to those in \cite{Ben-Li-Warnecke,Li-GXChen,Yang-He-Tang_GRP-RHD1D}.
The region of the left rarefaction wave can
be described by the set ${\cal R}:=\big\{\big(\alpha(t,r),
\beta(t,r)\big)| \beta \in [\beta_L,\beta_*], -\infty < \al \leq
0\big\}$, where $\beta_L =\la_-(\vec U_L)$ and $\beta_* =\la_-(\vec
U_*)$, and $\beta=\beta(t,r)$ and $\al=\al(t,r)$ are the integral
curves of the following equations
\begin{equation}
  \label{eq:IntegralCurves}
  \dfr{d r}{d t} =  \la_-,\quad
  \dfr{d r}{d t} =  \la_+,
\end{equation}
respectively.  Here $\beta$ and $\al$ have been denoted as follows:
$\beta$ is the initial value of the slope $\lambda_-$ at the
singularity point $(t,r) = (0,r_0)$, and $\al + r_0$ for the transversal
characteristic curves is the $r$-coordinate of the intersection point
with the leading $\beta$-curve, see Figure
\ref{fig:wave-pattern-a}. In this case, due to the local
transformation between $(t,r)$ and $(\alpha,\beta)$, all physical
quantities can be considered as functions of the local coordinates
$(\alpha,\beta)$ and the limiting states at $(t,r)=(0,0+)$ may be
represented as $\vec U_*= \vec U(0,\beta_*)$, etc. On the other hand,
the coordinates $(t,r)$ inside the left rarefaction fan shown in Figure \ref{fig:wave-pattern-a}
can be expressed in terms of $\al$ and $\beta$ as follows
$$
t=t(\al,\beta),\quad r=r(\al,\beta).
$$
Using the equations in \eqref{eq:IntegralCurves} gives
\begin{equation}
  \label{eq:pt2t-ab}
  \dfr{\pt r}{\pt \al} = \la_- \dfr{\pt t}{\pt \al}, \quad
  \dfr{\pt r}{\pt \beta} = \la_+\dfr{\pt t}{\pt
    \beta},\end{equation}
which respectively imply
\begin{equation}
  \label{eq:pt2t-ab_DT}
  \dfr{\pt }{\pt \al} =  \dfr{\pt t}{\pt \al} \dfr{D_-}{Dt} , \quad
    \dfr{\pt }{\pt \beta} = \dfr{\pt t}{\pt \beta} \dfr{D_{+}}{Dt}.
  \end{equation}
Differentiating the first equation in \eqref{eq:pt2t-ab} with respect to $\beta$ and the second with respect to $\alpha$ may further give
\begin{equation}
  \label{eq:pt2t}
  \dfr{\pt^2 t}{\pt \al \pt \beta} = \dfr{1}{\la_+ - \la_-}
  \left(
    \dfr{\pt\la_-}{\pt \beta}
    \dfr{\pt t}{\pt \al} - \dfr{\pt \la_+}{\pt \al}
    \dfr{\pt t}{\pt \beta}
  \right).
\end{equation}
At $\al =0$, the definition of the characteristic coordinates yields
\begin{equation}
  \label{eq:Ch-Der}
  \dfr{\pt \lambda_- }{\pt \beta} (0,\beta) = 1,\quad
  \dfr{\pt t}{\pt \beta}(0,\beta) = 0,\quad  \forall \beta \in [\beta_L ,  \beta_*].
\end{equation}
With the above two relations, setting $\al =0$ in \eqref{eq:pt2t} gives
\begin{equation}
    \label{eq:pt2t:al0}
    \dfr{\pt^2 t}{\pt \al \pt \beta}(0,\beta) = \dfr{1}{\lambda_+(0,\beta) - \lambda_-(0,\beta)}
    \dfr{\pt t}{\pt \al} (0,\beta).
\end{equation}

Our main result is given in following theorem.
\begin{thm}\label{thm:001}
  The limiting values $\left(\frac{\pt \rho}{\pt t}\right)_*$ and
  $\left(\frac{\pt v}{\pt t}\right)_*$ satisfy
  \begin{equation}
    \label{eq:Rare}
    a_L \left(\dfr{\pt \rho}{\pt t}\right)_* + b_L \left(\dfr{\pt v}{\pt t}\right)_* = d_L,
  \end{equation}
  where
  \begin{equation}
    \label{eq:aLbL}
    a_L=\left( \frac{c_s}{\rho +p}\right)_*,\quad
    b_L  = \left( \frac{1}{1 - v^2}\right)_*,
  \end{equation}
  and
 \begin{align}\nonumber
d_L = & \left( {\frac{{\lambda _ +  }}{{\lambda _ +   - \lambda _ -  }}} \right)_*  \left( \dfr{D_ -  \psi _ -  }{Dt} \right)_L \exp \left( -\int_{\beta _L }^{\beta_*}  {\dfr{d\hat\beta }{\lambda _ + (0,\hat\beta)  - \lambda _ - (0,\hat\beta) }}  \right) - \left( {\frac{{\lambda _ -  s_ -  }}{{\lambda _ +   - \lambda _ -  }}} \right)_*
\\[3mm]
\label{eq:dL}
&+ \left( {\frac{{\lambda _ +  }}{{\lambda _ +   - \lambda _ -  }}} \right)_*
\int_{\beta _L }^{\beta_*}   \left( \frac{s_ -  }{\lambda _ +   - \lambda _ -  } \right)
 (0,\hat\beta )\exp \left( -\int_{\hat\beta }^{\beta_*}   \frac{d\omega }{\lambda _ +  (0,\omega ) - \lambda _ -  (0,\omega )} \right)d\hat\beta,
\end{align}
   is a function of the initial data $\vec U_L, \vec U'_L$,
  and the limiting values $\vec U_*$ or $\vec U^{\mbox{\tiny RP}}$ of
  the (classical) Riemann problem solution $\vec \omega((r-r_0)/t,\vec
  U_L,\vec U_R)$.
\end{thm}
\begin{proof}
Using \eqref{eq:RIN} and \eqref{eq:Diff_RIN} gives
\begin{align*}
 \frac{{\lambda _ +  }}{{\lambda _ +   - \lambda _ -  }}\frac{{D_ -  \psi _ -  }}{{Dt}}
 &= \frac{{\lambda _ +  }}{{\lambda _ +   - \lambda _ -  }}\left( {\frac{{\partial \psi _ -  }}{{\partial t}} + \lambda _ -  \frac{{\partial \psi _ -  }}{{\partial r}}} \right) \\
 &= \frac{{\partial \psi _ -  }}{{\partial t}} + \frac{{\lambda _ -  }}{{\lambda _ +   - \lambda _ -  }}\left( {\frac{{\partial \psi _ -  }}{{\partial t}} + \lambda _ +  \frac{{\partial \psi _ -  }}{{\partial r}}} \right) \\
 & = \frac{1}{{1 - v^2 }}\frac{{\partial v}}{{\partial t}} + \frac{{c_s }}{{\rho  + p}}\frac{{\partial \rho }}{{\partial t}} + \frac{{\lambda _ -  s_ -  }}{{\lambda _ +   - \lambda _ -  }},
\end{align*}
which yields \eqref{eq:Rare} by transposition and setting $\al=0,~\beta=\beta_*$, with the coefficients $a_L,b_L$ given by \eqref{eq:aLbL} and
  \begin{equation}
    \label{eq:dL0}
    d_L  = \left( {\frac{{\lambda _ +  }}{{\lambda _ +   - \lambda _ -  }}} \right)_* \frac{{D_ -  \psi _ -  }}{{Dt}}(0,\beta _* )
    - \left( {\frac{{\lambda _ -  s_ -  }}{{\lambda _ +   - \lambda _ -  }}} \right)_*.
  \end{equation}
Thus the following task is to complete the calculation of $\frac{{D_ -  \psi _ -  }}{{Dt}}(0,\beta _* )$ in \eqref{eq:dL0}.
On the one hand, using the chain rule and \eqref{eq:pt2t-ab_DT} gives
\begin{equation}\label{eq:Ppsi_ab_1}
\frac{{\partial ^2 \psi _ -  }}{{\partial \alpha \partial \beta }} = \frac{\partial }{{\partial \beta }}\left( {\frac{{\partial t}}{{\partial \alpha }}\frac{{D_ -  \psi _ -  }}{{Dt}}} \right) = \frac{{\partial ^2 t}}{{\partial \beta \partial \alpha }}\frac{{D_ -  \psi _ -  }}{{Dt}} + \frac{{\partial t}}{{\partial \alpha }}\frac{\partial }{{\partial \beta }}\left( {\frac{{D_ -  \psi _ -  }}{{Dt}}} \right).
\end{equation}
On the other hand, utilizing \eqref{eq:pt2t-ab_DT} and \eqref{eq:Diff_RIN}, one has
\begin{equation}\label{eq:Ppsi_ab_2}
\frac{{\partial ^2 \psi _ -  }}{{\partial \alpha \partial \beta }} = \frac{\partial }{{\partial \alpha }}\left( {\frac{{\partial t}}{{\partial \beta }}\frac{{D_ +  \psi _ -  }}{{Dt}}} \right) = \frac{\partial }{{\partial \alpha }}\left( {\frac{{\partial t}}{{\partial \beta }}s_ -  } \right) = s_ -  \frac{{\partial ^2 t}}{{\partial \alpha \partial \beta }} + \frac{{\partial t}}{{\partial \beta }}\frac{{\partial s_ -  }}{{\partial \alpha }}.
\end{equation}
Combing \eqref{eq:Ppsi_ab_1} and \eqref{eq:Ppsi_ab_2}, and then setting $\al =0$ by making use of \eqref{eq:Ch-Der} may give
\[
\frac{{\partial t}}{{\partial \alpha }}(0,\beta )\frac{d }{{d \beta }}\left[ {\frac{{D_ -  \psi _ -  }}{{Dt}}(0,\beta )} \right] = \left( {s_ -  (0,\beta ) - \frac{{D_ -  \psi _ -  }}{{Dt}}(0,\beta )} \right)\frac{{\partial ^2 t}}{{\partial \beta \partial \alpha }}(0,\beta ),
\]
which further gives a ordinary differential equation at $\al=0$ for $\frac{D_ -  \psi _ -  }{Dt}(0,\beta ) $
\begin{equation}\label{eq:ODEpsi}
\dfr{d}{{d\beta }}\left[ {\frac{{D_ -  \psi _ -  }}{{Dt}}(0,\beta )} \right] =  - \dfr{1}{{\lambda _ +   - \lambda _ -  }}\dfr{{D_ -  \psi _ -  }}{{Dt}}(0,\beta ) + \dfr{{s_ -  }}{{\lambda _ +   - \lambda _ -  }},\quad \beta \in [\beta_L,\beta_*],
\end{equation}
by noting \eqref{eq:pt2t:al0}. Hence $\frac{D_ -  \psi _ -  }{Dt}(0,\beta )$ can be formulated by integrating \eqref{eq:ODEpsi} as
\begin{align}\nonumber
\dfr{D_ -  \psi _ -  }{Dt}(0,\beta )
= & \left( \dfr{D_ -  \psi _ -  }{Dt} \right)_L \exp \left( -\int_{\beta _L }^\beta  {\dfr{d\hat\beta }{\lambda _ + (0,\hat\beta)  - \lambda _ - (0,\hat\beta) }}  \right)\\[3mm]
\label{eq:Dpsi_beta}
&+
\int_{\beta _L }^\beta  \left( \frac{s_ -  }{\lambda _ +   - \lambda _ -  } \right)
 (0,\hat\beta )\exp \left( -\int_{\hat\beta }^\beta  \frac{d\omega }{\lambda _ +  (0,\omega ) - \lambda _ -  (0,\omega )} \right)d\hat\beta,
\end{align}
for all $\beta \in [\beta_L,\beta_*]$.
Setting $\beta=\beta_*$ in \eqref{eq:Dpsi_beta} and substituting it into \eqref{eq:dL0} may give the expression of $d_L$ in \eqref{eq:dL} and complete the proof. \qed
\end{proof}

\begin{remark}
If~$p(\rho)=\sigma^2 \rho$, one has $\lambda_-(0,\omega)=\sqrt{A_*B_*} \frac{ v(0,\omega) - \sigma }{1-v(0,\omega) \sigma}=\omega$, and
$$
v(0,\omega) = \frac{\sigma + \varpi } {1+\sigma \varpi },\quad \varpi = \frac{\omega}{\sqrt{A_*B_*}},
$$
For this case,  the integral $\int_{\hat\beta }^\beta  \frac{d\omega }{\lambda _ +  (0,\omega ) - \lambda _ -  (0,\omega )} $ in \eqref{eq:Dpsi_beta} can be expressed as
\begin{align*}
\int_{\hat\beta }^\beta  \frac{d\omega }{\lambda _ +  (0,\omega ) - \lambda _ -  (0,\omega )}
& = \frac{1}{2\sigma} \int_{\hat\beta/\sqrt{A_*B_*} }^{\beta/\sqrt{A_*B_*} }  \frac{ \sigma^2+1+2\sigma \varpi   }{1-\varpi ^2}  d \varpi
\\[2mm]
& = \frac{1}{4\sigma} \Big[ (\sigma-1)^2 \ln(1+\varpi) - (\sigma +1)^2 \ln(1-\varpi)  \Big]_{\varpi=\hat \beta/\sqrt{A_*B_*}}^{\varpi=\beta/\sqrt{A_*B_*} } .
\end{align*}
\end{remark}

\begin{remark} \label{rem:rightRare}
If the right rarefaction wave associated with the eigenvalue $\lambda_+$ appears in the GRP,
then the above derivation can be used by the ``reflective symmetry'' transformation
\begin{equation}\label{eq:reflect}
 \rho (r,t) = \tilde \rho ( - r,t), \ \
 v(r,t) =  - \tilde v( - r,t), \ \
 p(r,t) = \tilde p( - r,t),
\end{equation}
where $(\rho ,v,p)^T$ and $(\tilde \rho ,\tilde v,\tilde p)^T$ denote the primitive variables before and after the reflective transformation, respectively.
Specially, the ``reflective symmetry'' transformation is first used to
transfer the ``real'' right rarefaction wave into a  ``virtual''  left rarefaction wave,
 Theorem  \ref{thm:001} is then directly applied to the ``virtual''  left rarefaction wave, and finally using inverse transformation gives the linear equation of $\big(\frac{\pt \rho}{\pt t}\big)_*$ and $\big(\frac{\pt v}{\pt t}\big)_*$ for the right rarefaction wave.

\end{remark}

\subsubsection{Resolution of the shock wave}
\label{sec:resoluShock}
This section resolves  the right shock wave for the
GRP \eqref{eq:RHD_1} and \eqref{eq:InitialData}  in Figure \ref{fig:wave-pattern-a}
and  gives the second equation in
\eqref{eq:AlgebricEqns} through differentiating the shock relation along the shock trajectory.

Let $r=r_s(t)$ be the shock trajectory which is associated with the
$\lambda_+$--characteristic field, and assume that it propagates with the
speed $s:= r'_s(t)>0$ to the right, see Figure
\ref{fig:wave-pattern-a}. Denote the left and right states of the
shock wave  by $\vec U(t)$ and $\overline{\vec U}(t)$,
respectively, i.e. $\vec U(t)=\vec U\big(t,r_s(t) - 0\big)$ and
$\overline{\vec U}(t)=\vec U\big(t,r_s(t)+0\big)$.
The Rankine-Hugoniot relation across the shock wave is
\begin{equation}\label{eq:RHrelation}
\left[ { \sqrt {AB} \vec F(\vec U)}
 \right] = s\left[{ \vec U}
 \right],
\end{equation}
or equivalently,
\begin{equation}\label{eq:RHrelation2}
\left[ {  \vec F(\vec U)}
 \right] = (s/\sqrt {AB}) \left[ { \vec U}
\right],
\end{equation}
where $\left[ {\cdot} \right]$ denotes the jump across the shock wave, and the continuity of $\sqrt {AB}$ has been used. Utilizing this relation gives
\begin{equation}
  \label{eq:u-Phi}
   \dfr{ v -\bar v}{1 - v \bar v}
  =\Phi(\rho,\bar \rho),
\end{equation}
where
\begin{equation}
  \label{eq:Phi}
  \Phi(\rho,\bar \rho):= \sqrt{ \dfr{ (p-\bar p)(\rho - \bar \rho)}{(\rho + \bar p)(\bar \rho + p)  }  },
\end{equation}
see \cite{Landau1987} for the detailed derivation.
Thus along the shock trajectory, one always has
\begin{equation}
  \label{eq:DShockRelation}
  \dfr{D_{s}}{D t}
  \Big(\dfr{v -\bar v}{1 - v \bar v}\Big) =
  \dfr{D_{s}}{D t} \Phi(\rho,\bar \rho),
\end{equation}
where $\frac{D_{s} }{D t} := \frac{\pt }{\pt t} + s\frac{\pt}{\pt x}$ denotes the directional derivative along the shock trajectory.

The main result in this subsection is given as follows.

\begin{thm}\label{thm:002}
The limiting values of $\left(\frac{\pt \rho}{\pt t}\right)_*$ and
  $\left(\frac{\pt v}{\pt t}\right)_*$ satisfy
  \begin{equation}
    \label{eq:shock}
    a_R \left(\dfr{\pt \rho}{\pt t}\right)_* + b_R \left(\dfr{\pt v}{\pt t}\right)_* = d_R,
  \end{equation}
  where the expressions of the coefficients $a_R,b_R$, and $d_R$ will be given in the proof.
\end{thm}

\begin{proof}
Utilizing \eqref{eq:RHD_pri}, one has
\begin{align}\nonumber
&
\left( {\begin{array}{*{20}c}
   {v^2  - c_s^2  - \frac{{sv(1 - c_s^2 )}}{{\sqrt {AB} }}} & {\frac{{s(\rho  + p)}}{{\sqrt {AB} }}}  \\
   {\frac{{s(1 - v^2 )^2 c_s^2 }}{{\sqrt {AB} (\rho  + p)}}} & {v^2  - c_s^2  - \frac{{sv(1 - c_s^2 )}}{{\sqrt {AB} }}}
\end{array}} \right)\frac{{\partial \vec V}}{{\partial t}} \\[2mm] \nonumber
&
 = \left( {v^2  - c_s^2 } \right)\frac{{\partial \vec V}}{{\partial t}} - \frac{s}{{\sqrt {AB} }}\left( {\begin{array}{*{20}c}
   {v(1 - c_s^2 )} & { - (\rho  + p)}  \\
   { - \frac{{(1 - v^2 )^2 c_s^2 }}{{(\rho  + p)}}} & {v(1 - c_s^2 )}
\end{array}} \right)\frac{{\partial \vec V}}{{\partial t}} \\[2mm] \nonumber
&
 = \left( {v^2  - c_s^2 } \right)\frac{{\partial \vec V}}{{\partial t}} - \frac{s}{{\sqrt {AB} }}\left( {\begin{array}{*{20}c}
   {v(1 - c_s^2 )} & { - (\rho  + p)}  \\
   { - \frac{{(1 - v^2 )^2 c_s^2 }}{{(\rho  + p)}}} & {v(1 - c_s^2 )}
\end{array}} \right)\left( {\vec H - \vec  J\frac{{\partial \vec V}}{{\partial r}}} \right) \\[2mm] \nonumber
&
 = \left( {v^2  - c_s^2 } \right)\left( {\frac{{\partial \vec V}}{{\partial t}} + s\frac{{\partial \vec V}}{{\partial r}}} \right) - \frac{s}{{\sqrt {AB} }}\left( \begin{array}{l}
 v(1 - c_s^2 )H_1  - (\rho  + p)H_2  \\
  - \frac{{(1 - v^2 )^2 c_s^2 }}{{(\rho  + p)}}H_1  + v(1 - c_s^2 )H_2
 \end{array} \right) \\ \nonumber
&
 = \left( {v^2  - c_s^2 } \right)\frac{{D_s \vec V}}{{Dt}} + \frac{s}{r}\left( \begin{array}{l}
 (\rho  + p)\left( {2v^2  - \frac{{(1 - A)(1 - v^2 )}}{{2A}} - \frac{{\kappa r^2 (p + \rho v^2 )}}{{2A}}} \right) \\
 v(v^2  - 1)\left( {2c_s^2  - \frac{{(1 - A)(1 - c_s^2 )}}{{2A}} - \frac{{\kappa r^2 (p + \rho c_s^2 )}}{{2A}}} \right) \\
 \end{array} \right)
\\[2mm] \label{eq:DsANDPt}
&
 =: \left( {v^2  - c_s^2 } \right)\frac{{D_s \vec V}}{{Dt}} + (\Pi _1, \Pi _2)^{\rm T},
\end{align}
with which a relation between $\frac{\pt} {\pt t}$ and $\frac{D_s} {Dt}$ is established. The following will expand \eqref{eq:DShockRelation} by the chain rule and use the relation \eqref{eq:DsANDPt} to transfer the derivatives along the shock trajectory to the time derivatives for the left state of the shock wave.  The left and right hand sides of \eqref{eq:DShockRelation} will be treated separately.

The left hand side of \eqref{eq:DShockRelation} can be rewritten as
\begin{align*}%\label{eq:LHS}
\dfr{D_{s}}{D t}
  \Big(\dfr{v -\bar v}{1 - v \bar v}\Big)= \dfr{1-\bar v^2}{(1-v \bar v)^2}\dfr{D_{s} v}{D t} -  \dfr{1- v^2}{(1-v \bar v)^2}\dfr{D_{s} \bar v}{D t} .
\end{align*}
If considering $p$ as a function of $\rho$, i.e. $p=p(\rho)$, then the  right hand side of \eqref{eq:DShockRelation} may be expanded as
\[
\frac{{D_s }}{{Dt}}\Phi (\rho ,\bar \rho ) = \Phi _\rho (\rho,\bar \rho)  \frac{{D_s \rho }}{{Dt}} + \Phi _{\bar \rho } (\rho,\bar \rho) \frac{{D_s \bar \rho }}{{Dt}},
\]
where
\begin{align*}
 \Phi _\rho  (\rho,\bar \rho)  &= \frac{{(\bar \rho  + \bar p)\left( {c_s^2 \Delta_{\rho,\bar \rho}^{-1} (\rho  + \bar p) + \Delta_{\rho,\bar \rho}  (\bar \rho  + p)} \right)}}{{
 2 \left( {(\rho  + \bar p)(\bar \rho  + p)} \right)^{\frac{3}{2}}  }}, \\
 \Phi _{\bar \rho } (\rho,\bar \rho) &= - \frac{{(\rho  + p)\left( {\bar c_s^2 \Delta_{\rho,\bar \rho}^{-1} (\bar \rho  + p) +  \Delta_{\rho,\bar \rho} (\rho  + \bar p)} \right)}}{{
 2 \left( {(\rho  + \bar p)(\bar \rho  + p)} \right)^{\frac{3}{2}} }},
\end{align*}
with $\Delta_{\rho,\bar \rho} := \sqrt{ (p-\bar p)/(\rho - \bar \rho)}$.
Therefore, \eqref{eq:DShockRelation} is equivalent to
\begin{equation}\label{eq:DsShockRelation2}
\Phi _\rho (\rho,\bar \rho) \frac{{D_s \rho }}{{Dt}} +
 \dfr{\bar v^2-1}{(1-v \bar v)^2}\dfr{D_{s} v}{D t}
=  \dfr{v^2-1}{(1-v \bar v)^2}\dfr{D_{s} \bar v}{D t} - \Phi _{\bar \rho } (\rho,\bar \rho) \frac{{D_s \bar \rho }}{{Dt}}.
\end{equation}
Multiplying both sides of \eqref{eq:DsANDPt} with the row  vector $\left(\Phi _\rho (\rho,\bar \rho) , \frac{\bar v^2-1}{(1-v \bar v)^2} \right),$
and then substituting \eqref{eq:DsShockRelation2} into it, we have
\begin{align}\nonumber
&
\bigg(\Phi _\rho (\rho,\bar \rho), \dfr{\bar v^2-1}{(1-v \bar v)^2} \bigg)
\left( {\begin{array}{*{20}c}
   {v^2  - c_s^2  - \frac{{sv(1 - c_s^2 )}}{{\sqrt {AB} }}} & {\frac{{s(\rho  + p)}}{{\sqrt {AB} }}}  \\
   {\frac{{s(1 - v^2 )^2 c_s^2 }}{{\sqrt {AB} (\rho  + p)}}} & {v^2  - c_s^2  - \frac{{sv(1 - c_s^2 )}}{{\sqrt {AB} }}}
\end{array}} \right)\frac{{\partial \vec V}}{{\partial t}} \\[2mm]
\nonumber
& = (v^2-c_s^2) \bigg( \Phi _\rho (\rho,\bar \rho) \frac{{D_s \rho }}{{Dt}} +
 \dfr{\bar v^2-1}{(1-v \bar v)^2}  \dfr{D_{s} v}{D t} \bigg) + \Phi _\rho (\rho,\bar \rho) \Pi _1 + \dfr{\bar v^2-1}{(1-v \bar v)^2} \Pi _2
\\[2mm]
\label{eq:prhopv}
& = (v^2-c_s^2)
\bigg(
\dfr{v^2-1}{(1-v \bar v)^2}\dfr{D_{s} \bar v}{D t} - \Phi _{\bar \rho } (\rho,\bar \rho) \frac{{D_s \bar \rho }}{{Dt}}
 \bigg) + \Phi _\rho (\rho,\bar \rho) \Pi _1 + \dfr{\bar v^2-1}{(1-v \bar v)^2} \Pi _2.
\end{align}
Setting $t \to 0^+$ in \eqref{eq:prhopv} gives \eqref{eq:shock} with
\begin{align} \label{eq:aR}
a_R &=
\Phi _\rho (\rho_*, \rho_R)  \left( {v^2  - c_s^2  - \frac{{sv(1 - c_s^2 )}}{{\sqrt {AB} }}} \right)_*
+ \dfr{v^2_R-1}{(1-v_*  v_R)^2}   \left ( {\frac{{s(1 - v^2 )^2 c_s^2 }}{{\sqrt {AB} (\rho  + p)}}}  \right)_*  ,
\\[3mm]
\label{eq:bR}
b_R &=
\Phi _\rho  (\rho_*, \rho_R)  {\frac{{s_*(\rho_*  + p_*)}}{{\sqrt {A_*B_*} }}}
+ \dfr{v^2_R-1}{(1-v_*  v_R)^2}   \left(   {v^2  - c_s^2  - \frac{{sv(1 - c_s^2 )}}{{\sqrt {AB} }}}  \right)_*,
\end{align}
and
\begin{align*}
d_R  = & \Phi _\rho (\rho_*,\rho_R) ( \Pi _1 )_* + \dfr{v^2_R-1}{(1-v_*  v_R)^2} (\Pi _2)_*
 \\[3mm]
 &
 +(v^2-c_s^2)_*
\left[
\dfr{ v^2_*-1}{(1-v_* v_R)^2} \left( \dfr{D_{s}  v}{D t} \bigg)_R- \Phi _{\bar \rho } (\rho_*,\rho_R)  \right(  \frac{{D_s  \rho }}{{Dt}} \bigg)_R
 \right],
\end{align*}
where $s_*$ denotes the initial value of the slope of the shock trajectory.
The proof is completed.
\qed
\end{proof}

\begin{remark}
% \label{rem:rightRare}
Similarly, if the left shock wave associated with the eigenvalue $\lambda_-$ appears in the GRP,
then Theorem  \ref{thm:002} may  be applied to the ``virtual'' right shock wave obtained by using the ``reflective symmetry'' transformation \eqref{eq:reflect},
and  corresponding inverse transformation is finally used to give the linear equation of the limiting values $\big(\frac{\pt \rho}{\pt t}\big)_*$ and $\big(\frac{\pt v}{\pt t}\big)_*$.
\end{remark}

\subsubsection{Time derivatives of solutions at singularity point}
\label{sec:PuPt}

This subsection is devoted to derive the time derivatives $ \left( {\pt \vec U} /{\pt t} \right)_*$
and complete the calculation in \eqref{eq:TalyorGRP}.

\subsubsection*{I. Nonsonic case}

The discussion in this part is  only restricted to the case of the local wave
configuration  in Fig. \ref{fig:wave-pattern-a}, in which
 $\lambda_-(\vec U_*)<0$, that is to say,   the $t$-axis  locates
at the right hand side of the rarefaction wave.
Other (possible) nonsonic local wave configurations can be similarly discussed
and should be implemented in the code,
while the case of the transonic rarefaction wave will be discussed later.

% Actually, solving the $2\times 2$ linear system formed by \eqref{eq:Rare} in Theorem \ref{thm:001}
%and \eqref{eq:shock} in Theorem \ref{thm:002} may give the time derivatives $(\pt \rho / \pt t)_*$ and $(\pt v / \pt t)_*$.
%Thus, it only needs  to show that the $2\times 2$ linear system has unique solution, which may be concluded from the following theorem.

\begin{thm}
  \label{thm:unique}
Under the assumption that $\lambda_-(\vec U_*)<0$ and $s_*>0$, the $2\times 2$ linear system formed by \eqref{eq:Rare} in Theorem \ref{thm:001} and \eqref{eq:shock} in Theorem \ref{thm:002}
has unique solution.
\end{thm}

\begin{proof}
It only needs to check
\[
\det \left( {\begin{array}{*{20}c}
   {a_L } & {b_L }  \\
   {a_R } & {b_R }
\end{array}} \right) = a_L b_R  - a_R b_L  \ne 0,
\]
which is sufficiently  ensured if  $a_L b_L a_R b_R<0$.
It is obvious that $a_Lb_L>0$ due to \eqref{eq:aLbL}.
Thus if $a_Rb_R<0$, then the proof may be completed.
Before that, we first prove the following inequality
\begin{equation}\label{eq:ieq-s}
\left( {v^2  - c_s^2  - \frac{{sv(1 - c_s^2 )}}{{\sqrt {AB} }}} \right)_* < 0.
\end{equation}
The Lax entropy inequality for the right shock wave
\begin{equation}\label{eq:EntropyCon}
\left( \sqrt {AB} \dfr{v + c_s }{1 + vc_s } \right)_*  > s_*  > \left( \sqrt {AB} \dfr{v + c_s }{1 + vc_s } \right)_R,
\end{equation}
 implies $(v+c_s)_*>0$ due to $s_*>0$.
 If $v_* \le 0$, with the left inequality in \eqref{eq:EntropyCon}, one has
\begin{align*}
 \left( {\frac{{sv(1 - c_s^2 )}}{{\sqrt {AB} }}} \right)_*   \ge \left( {\frac{{(v + c_s)v(1-c_s^2) }}{{1 + vc_s }}} \right)_*   = \left(v^2  - c_s^2   +  {\frac{{c_s (v + c_s )(1 - v^2 )}}{{1 + vc_s }}} \right)_*   > (v^2  - c_s^2 )_* ,
\end{align*}
where $(v+c_s)_*>0$ has been used in the last inequality.
If $v_*>0$, then one has
\[
\left( {v^2  - c_s^2  - \frac{{sv(1 - c_s^2 )}}{{\sqrt {AB} }}} \right)_*  < (v^2  - c_s^2 )_*  = \lambda _ -  (\vec U_* )\left( {v + c_s } \right)_* \left( {\frac{{1 - vc_s }}{{\sqrt {AB} }}} \right)_*  < 0,
\]
where $\lambda_-(\vec U_*)<0$ and $(v+c_s)_*>0$ have been used in the last inequality. Hence the inequality \eqref{eq:ieq-s} holds for any $v_*$. The expressions of
$a_R$ in \eqref{eq:aR} and $b_R$ in \eqref{eq:bR}, and $v_R<1$ and $\Phi _\rho (\rho_*, \rho_R)>0$ imply  that $a_R<0$ and $b_R>0$. Therefore $a_L b_L a_R b_R<0$ and the proof is completed. \qed
\end{proof}

%\begin{remark}
%\end{remark}

\subsubsection*{II. Sonic case}

Next, discuss the sonic case of that the $t$-axis is within the left rarefaction region in Fig. \ref{fig:wave-pattern-a}. In this case,
the $t$-axis is actually tangential to the $\lambda_-$-characteristic curve with zero initial slope, i.e. $\lambda_-(0,\beta_*)=0$ or $\beta_*=0$.
Then one has
$$
\left( \dfr{\pt \vec U}{\pt t} \right)_* = \dfr{D_- \vec U}{D t} (0,\beta_*),
$$
for $\beta_*=0$.
\begin{thm}
  \label{thm:sonic}
If the $t$-axis is (locally) located inside the left rarefaction wave,
then one has
\begin{align}\label{eq:sonic:rho}
\left( \dfr{\pt  \rho }{\pt t} \right)_* = \left( \dfr{\rho+p}{2c_s}\right)_* \left( \dfr{D_- \psi_-}{D t}(0,\beta_*) - ( s_ + )_* \right),
\\[3mm]
\label{eq:sonic:v}
\left( \dfr{\pt v}{\pt t} \right)_* =  \dfr{1-v_*^2}{2} \left( \dfr{D_- \psi_-}{D t}(0,\beta_*) + ( s_ + )_* \right),
\end{align}
with $\beta_*=0$, where $\frac{D_- \psi_-}{D t}(0,\beta_*)$ can be calculated by setting $\beta=\beta_*$ in \eqref{eq:Dpsi_beta}.
\end{thm}

\begin{proof}
From \eqref{eq:Diff_RIN},  for $\al=0,\beta_*=0$, one has
\begin{equation}\label{eq:sonic:puptEq1}
\left( \frac{1}{{1 - v^2 }}
 \frac{{\pt   v}}{{\pt t}} - \frac{{c_s }}{{\rho  + p}}\frac{{\pt  \rho }}{{\pt t}} \right)_* =
\left( \frac{1}{{1 - v^2 }}\frac{{D_ -  v}}{{Dt}} - \frac{{c_s }}{{\rho  + p}}\frac{{D_ -  \rho }}{{Dt}} \right)_*= ( s_ + )_*.
\end{equation}
On the other hand, the expression of $\psi_-$ in \eqref{eq:RIN} yields
\begin{equation}\label{eq:sonic:puptEq2}
\left( \frac{1}{{1 - v^2 }}\frac{{\pt   v}}{{\pt t}} + \frac{{c_s }}{{\rho  + p}}\frac{{\pt  \rho }}{{\pt t}} \right)_* =
\dfr{\pt \psi_-}{\pt t}(0,\beta_*) = \dfr{D_- \psi_-}{D t}(0,\beta_*) .
\end{equation}
Combining \eqref{eq:sonic:puptEq1} and \eqref{eq:sonic:puptEq2} may give \eqref{eq:sonic:rho} and \eqref{eq:sonic:v}, and complete the proof. \qed
\end{proof}

As soon as the limiting values of the time derivatives $(\pt \rho / \pt t)_*$ and $(\pt v / \pt t)_*$ are obtained, the limiting values of the time derivatives
for conservative variable vector $\vec U$ can be calculated by the chain rule as follows
\[
\left( \dfr{{\partial \vec U}}{{\partial t}} \right)_* = \left( \begin{array}{c}
 \left( {(1 + c_s^2 )W^2  - c_s^2 } \right)\frac{{\partial \rho }}{{\partial t}} + 2vW^4 (\rho  + p)\frac{{\partial v}}{{\partial t}} \\[1mm]
 (1 + c_s^2 )W^2 v\frac{{\partial \rho }}{{\partial t}} + W^4 (1 + v^2 )(\rho  + p)\frac{{\partial v}}{{\partial t}}
 \end{array} \right)_*.
\]

\subsubsection{Acoustic case}
\label{sec:acoustic}

This section  discusses  the acoustic case of the GRP \eqref{eq:RHD_1} and
\eqref{eq:InitialData}, i.e. $\vec U_L = \vec U_R$ and $\vec U_L' \neq
\vec U_R'$. In this case, $\vec U_L = \vec U_* = \vec U_R$ and only
linear waves emanate from the origin (0,0) so that the resolution of
the GRP becomes simpler than the general case discussed before, see
Theorem \ref{thm:acoustic}.
%For the sake of simplicity, in the following we only reserve the
%subscripts related to the state $\vec U_L$ or $\vec U_*$ or $\vec U_R$
%for the derivatives of all physical variables.
For the sake of simplicity,  the  subscripts $L$, $R$, and $*$
of the variables $\vec U$ or $\vec V$ etc. will be
omitted because $\vec U_L=\vec U_*=\vec U_R$.

\begin{thm}
\label{thm:acoustic}
If $\lambda_-< 0$ and $\lambda_+> 0$, then $(\pt \rho/ \pt t)_*$ and $(\pt v/\pt t)_*$ can be obtained by
\begin{align}
\label{eq:AcousticRho}
 \left( {\frac{{\partial \rho }}{{\partial t}}} \right)_*  &=  - \frac{1}{2}\left[ {\lambda _ +  \rho' _L  + \lambda _ -  \rho' _R  + \frac{{\rho  + p}}{{c_s (1 - v^2 )}}\left( {\lambda _ +  v'_L  - \lambda _ -  v'_R } \right)} \right] + H_1, \\[2mm]
 \label{eq:AcousticV}
 \left( {\frac{{\partial v }}{{\partial t}}} \right)_*  &=  - \frac{1}{2}\left[ {\lambda _ +  v' _L  + \lambda _ -  v' _R  + \frac{{c_s (1 - v^2)}}{{ \rho  + p}}\left( {\lambda _ +  \rho'_L  - \lambda _ -  \rho'_R } \right)} \right] + H_2.
\end{align}
\end{thm}

\begin{proof}
Since the solution $\rho(t,r)$ is continuous across the $\lambda_-$ characteristic curves, the directional derivative along the trajectory $r'(t)=\lambda_-$ of
the variable $\rho$ satisfies
\begin{align}
\nonumber
 \left( {\frac{D_- \rho }{D t}} \right)_*
 &= \left( {\frac{{\partial \rho }}{{\partial t}}} \right)_L  + \lambda _ -  \left( {\frac{{\partial \rho }}{{\partial x}}} \right)_L  \\
\nonumber
 & =  - \frac{{\sqrt {AB} }}{{1 - v^2 c_s^2 }}\left[ {v(1 - c_s^2 )\left( {\frac{{\partial \rho }}{{\partial x}}} \right)_L  + (\rho  + p)\left( {\frac{{\partial v}}{{\partial x}}} \right)_L } \right] + H_1  + \lambda _ -  \left( {\frac{{\partial \rho }}{{\partial x}}} \right)_L  \\
\label{eq:leftrho}
 & =  - \frac{{\sqrt {AB} }}{{1 - v^2 c_s^2 }}\left[ {c_s (1 - v^2 )\left( {\frac{{\partial \rho }}{{\partial x}}} \right)_L  + (\rho  + p)\left( {\frac{{\partial v}}{{\partial x}}} \right)_L } \right] + H_1 ,
\end{align}
here the first equation in \eqref{eq:RHD_pri} has been used. Similarly, the directional derivative along the trajectory $r'(t)=\lambda_+$ of the variable $\rho$ may be calculated by
\begin{align}
\nonumber
 \left( {\frac{D_+ \rho }{D t}} \right)_*
 &= \left( {\frac{{\partial \rho }}{{\partial t}}} \right)_R  + \lambda _ +  \left( {\frac{{\partial \rho }}{{\partial x}}} \right)_R  \\
\nonumber
 & =  - \frac{{\sqrt {AB} }}{{1 - v^2 c_s^2 }}\left[ {v(1 - c_s^2 )\left( {\frac{{\partial \rho }}{{\partial x}}} \right)_R  + (\rho  + p)\left( {\frac{{\partial v}}{{\partial x}}} \right)_R } \right] + H_1  + \lambda _ +  \left( {\frac{{\partial \rho }}{{\partial x}}} \right)_R  \\
\label{eq:rightrho}
 & =  - \frac{{\sqrt {AB} }}{{1 - v^2 c_s^2 }}\left[ {c_s ( v^2-1 )\left( {\frac{{\partial \rho }}{{\partial x}}} \right)_R  + (\rho  + p)\left( {\frac{{\partial v}}{{\partial x}}} \right)_R } \right] + H_1 .
\end{align}
Substituting  \eqref{eq:leftrho} and \eqref{eq:rightrho} into
$$
 \left( {\frac{{\partial \rho }}{{\partial t}}} \right)_*  = \frac{1}{{\lambda _ +   - \lambda _ -  }}\left( {\lambda _ +  \frac{{D_ -  \rho }}{{Dt}} - \lambda _ -  \frac{{D_ +  \rho }}{{Dt}}} \right)_*
$$
may yield \eqref{eq:AcousticRho}. Since the solution $v(t,r)$ is continuous across the $\lambda_-$ characteristic curves, it holds by the second equation in \eqref{eq:RHD_pri} that
\begin{align}
\nonumber
 \left( {\frac{D_- v }{D t}} \right)_*
 &= \left( {\frac{{\partial v }}{{\partial t}}} \right)_L  + \lambda _ -  \left( {\frac{{\partial v }}{{\partial x}}} \right)_L  \\
\nonumber
 & =  - \frac{{\sqrt {AB} }}{{1 - v^2 c_s^2 }}\left[ { \frac{(1 - v^2 )^2 c_s^2 }{\rho  + p} \left( {\frac{{\partial \rho }}{{\partial x}}} \right)_L  + v(1 - c_s^2 ) \left( {\frac{{\partial v}}{{\partial x}}} \right)_L } \right] + H_2  + \lambda _ -  \left( {\frac{{\partial v }}{{\partial x}}} \right)_L  \\
\label{eq:leftv}
 & =  - \frac{{\sqrt {AB} }}{{1 - v^2 c_s^2 }}\left[ {   \frac{(1 - v^2 )^2 c_s^2 }{\rho  + p}   \left( {\frac{{\partial \rho }}{{\partial x}}} \right)_L  + c_s (1 - v^2 )\left( {\frac{{\partial v}}{{\partial x}}} \right)_L } \right] + H_2 ,
\end{align}
Similarly, the directional derivative along the trajectory $r'(t)=\lambda_+$ of the variable $v$ is given by
\begin{align}
\nonumber
 \left( {\frac{D_+ v }{D t}} \right)_*
 &= \left( {\frac{{\partial v }}{{\partial t}}} \right)_R  + \lambda _ +  \left( {\frac{{\partial v }}{{\partial x}}} \right)_R  \\
\nonumber
 & =  - \frac{{\sqrt {AB} }}{{1 - v^2 c_s^2 }}\left[ { \frac{(1 - v^2 )^2 c_s^2 }{\rho  + p} \left( {\frac{{\partial \rho }}{{\partial x}}} \right)_R  + v(1 - c_s^2 ) \left( {\frac{{\partial v}}{{\partial x}}} \right)_R } \right] + H_2  + \lambda _ -  \left( {\frac{{\partial v }}{{\partial x}}} \right)_R  \\
\label{eq:rightv}
 & =  - \frac{{\sqrt {AB} }}{{1 - v^2 c_s^2 }}\left[ {   \frac{(1 - v^2 )^2 c_s^2 }{\rho  + p}   \left( {\frac{{\partial \rho }}{{\partial x}}} \right)_R  + c_s (v^2-1 )\left( {\frac{{\partial v}}{{\partial x}}} \right)_R } \right] + H_2 ,
\end{align}
Substituting  \eqref{eq:leftv} and \eqref{eq:rightv} into
$$
 \left( {\frac{{\partial v }}{{\partial t}}} \right)_*  = \frac{1}{{\lambda _ +   - \lambda _ -  }}\left( {\lambda _ +  \frac{{D_ -  v }}{{Dt}} - \lambda _ -  \frac{{D_ +  v }}{{Dt}}} \right)_*
$$
can give \eqref{eq:AcousticV}.
The proof is completed. \qed
\end{proof}

%%%%%%%%%%%%%%%%%%%%%Numerical Experiments

%%%%%%%%%%%%%%%%%%%%%%%%%%%%%%%%%%%%%%%%%%%%%%%%%%%%%%%%%%%%%%%%%%%%%%%%%%%%%%%%%%%
\section{Numerical experiments}
\label{sec:experiments}

This section will solve several initial-boundary-value problems of the spherically symmetric general RHD equations \eqref{eq:RHD_1}--\eqref{eq:RHD_1B}
to verify the accuracy and the capability in resolving discontinuity  of the GRP scheme presented in the last section, in comparison with
the Godunov scheme, given in Appendix \ref{sec:AppendixA}, which is   little different from but simpler than the one presented in \cite{VoglerTemple2012}.
Unless specifically stated, all computations will be restricted to the equation of state %\eqref{eq:EOS} with $p(\rho)=\sigma^2 \rho$
\eqref{eq:EOS-thz02}
and the CFL number $C_{cfl}=0.45$ (resp. 0.9) for the GRP (resp. Godunov) scheme, where $\sigma$ is a positive constant less than 1. Moreover,
the parameter $\theta$ in \eqref{eq:limiter} is taken as $1.9$.

\begin{example}[Accretion onto a Schwarzschild black hole]\label{example:accretion}\rm

This test simulates the stationary solution of
the spherical accretion onto a Schwarzschild black hole of unit mass (i.e. $M=1$), %\cite{Hawley1984}) J.M. Michel 1972
where  the Schwarzschild spacetime is considered  with the line element
\eqref{eq:spacetime} and
$$
A(t,r)=B(t,r)=1-\frac{2}{r}.
$$
%ds^2  =  - \left( 1-\frac{2}{r} \right) d t^2  + \frac{1}{ 1-\frac{2}{r} } d r^2 + r^2 \left( d\theta ^2  + \sin ^2 \theta d\phi ^2  \right).
%$$
The Einstein coupling constant $\kappa$ is taken as 0 in the spherically symmetric
general RHD equations \eqref{eq:RHD_1}--\eqref{eq:RHD_1B} and the derivation of the GRP scheme, %the radial geometric ``dust infall'' with
and the parameter $\sigma$ in \eqref{eq:EOS-thz02} %the equation of state \eqref{eq:EOS}
is taken as 0.1.
Thus the analytic steady state solution  is given by
\begin{equation}
\rho (r) = \frac{{{D_0}(1 - {v^2})}}{{ - v{r^2} A(t,r)}},\quad v(r) =  - \sqrt {\varpi (r)} ,
\label{eq:examp01}
\end{equation}
where $\varpi (r)$ solves the nonlinear equation
\[(1 - \varpi ){\varpi ^{\frac{{{\sigma ^2}}}{{1 - {\sigma ^2}}}}} =
%\left( {1 - \frac{2}{r}} \right)
A(t,r){\left( {\frac{2}{r}} \right)^{\frac{{4{\sigma ^2}}}{{1 - {\sigma ^2}}}}}.\]

  \begin{figure}[!htbp]
    \centering
        \includegraphics[height=0.38\textwidth]{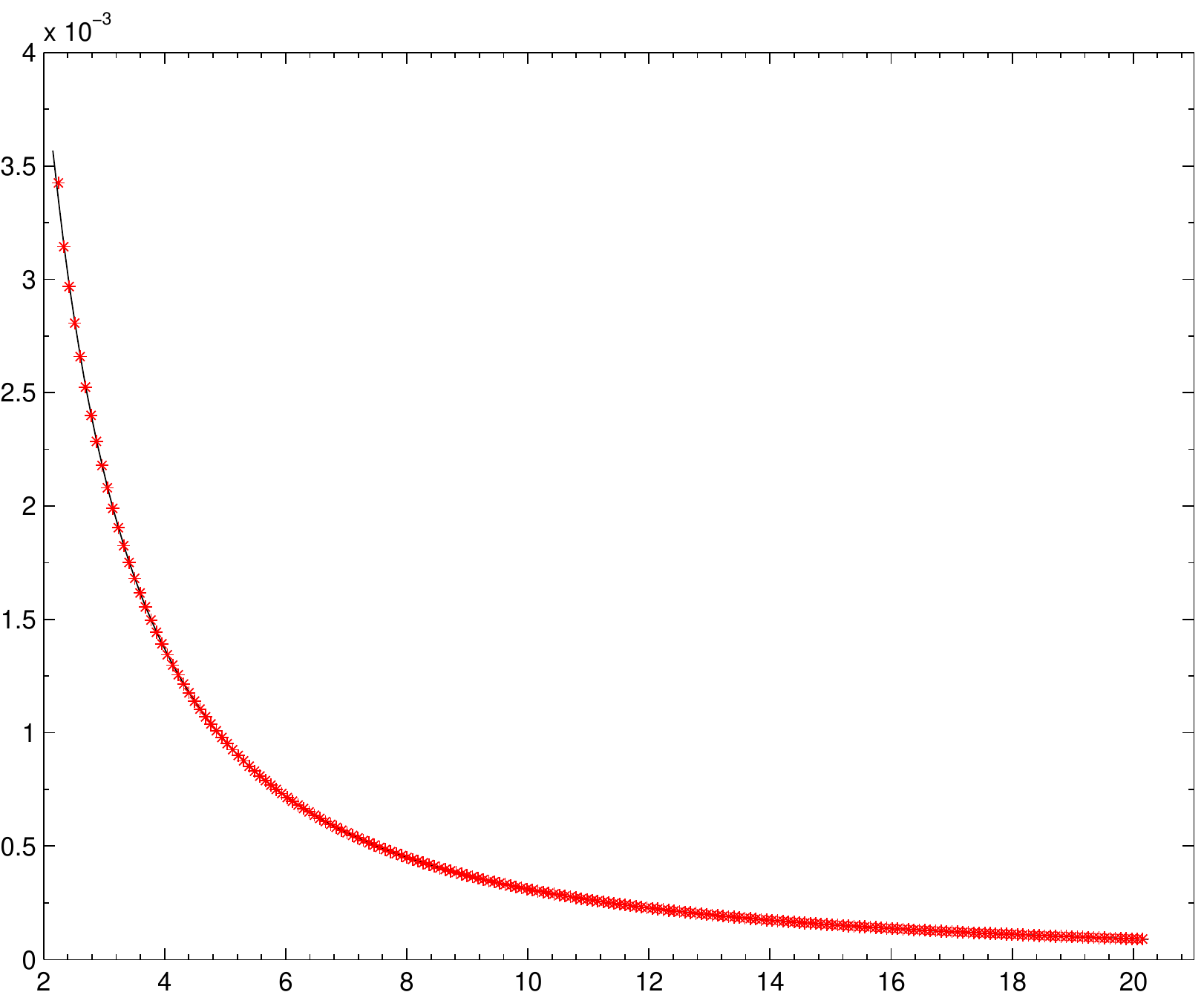}
        \includegraphics[height=0.38\textwidth]{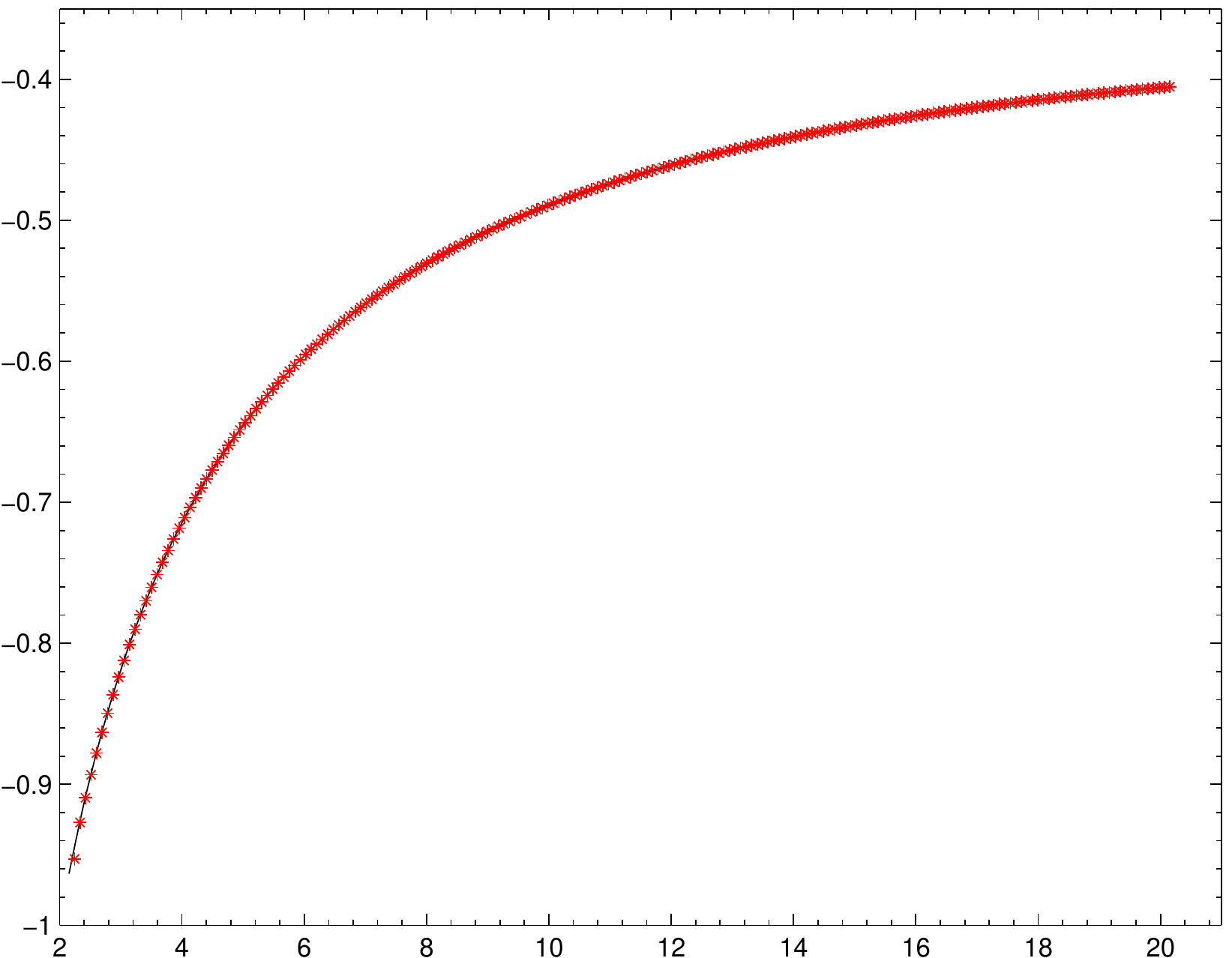}
    \caption{Example \ref{example:accretion}:
    The rest energy density $\rho$ (left) and velocity $v$ (right) at $t=160$.
    The numerical solutions obtained by the GRP scheme
are drawn in the symbol ``{\color{red}$\ast$}'',
while the solid lines stand for the exact steady solutions.}
    \label{fig:accretion}
  \end{figure}

    \begin{figure}[!htbp]
    \centering
        \includegraphics[height=0.4\textwidth]{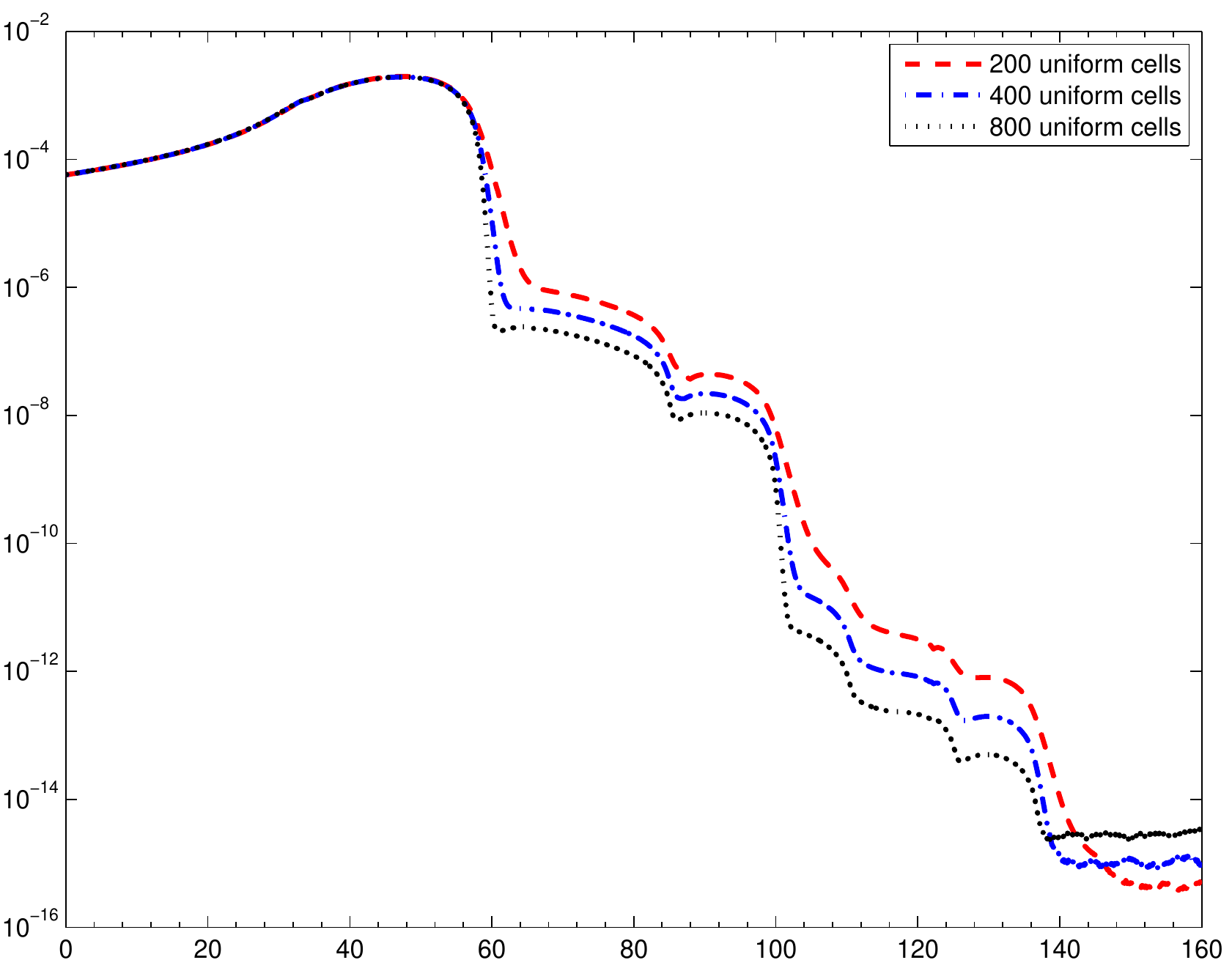}
    \caption{Example \ref{example:accretion}:
    Convergence history in the residuals with respect to the time $t$ on three uniform meshes. }
    \label{fig:accretion2}
  \end{figure}

  In the computations, the parameter $D_0$ in \eqref{eq:examp01} is taken as $1.6 \times 10^{-2}$,
  %The initial conditions are those of a vacuum, i.e.,
  the initial density $\rho$ is zero everywhere and taken as
  a small number e.g. $10^{-8}$,  and the velocity $v=0$, except on the
  outer boundary where a gas is  injected continuously with  the exact flow variables in \eqref{eq:examp01}.
  Fig. \ref{fig:accretion} shows the numerical results at $t=160$ given by the GRP scheme with 200 uniform cells in the computational domain $[2.2,20.2]$. Outflow boundary conditions have been specified at the inner boundary $r=2.2$.
  It can be seen that the velocity approaches the speed of light when the gas approaches to the black hole, while the proposed GRP scheme exhibits good robustness.
  Fig. \ref{fig:accretion2} displays the convergence history in the residuals with respect to the time $t$ on three meshes of 200, 400 and 800 uniform cells respectively.
  It can be seen that the correct steady solutions are obtained by the GRP scheme with the residuals  less than $10^{-14}$.

\end{example}

Before simulating the shock wave models, we first consider several continuous models in fullly general relativistic case, which are two transformations of the Friedmann-Robertson-Walker (FRW) metrics (denoted by
FRW-1 and FRW-2 respectively) and the Tolmann-Oppengeimer-Volkoff (TOV) metric presented in \cite{Vogler2010}. The exact solutions to those continuous models
are smooth and may be used to test the accuracy of the proposed GRP scheme.

\begin{example}[FRW-1 model]\label{example:FRW1}\rm
Consider the conformally flat FRW metric, where the distance measure by the line element
\begin{equation}\label{eq:FRW}
ds^2  =  - d\tilde t^2  + R^2 (\tilde t)\left[ {d\tilde r^2  + \tilde r^2 \left( {d\theta ^2  + \sin ^2 \theta d\phi ^2 } \right)} \right],
\end{equation}
where $\tilde t$ is the time since the big bang, and the cosmological scale function
is defined by $R (\tilde t)=\sqrt{\tilde t}$. Under the  coordinate transformation \cite{TempleSmoller2009}
\[
t = \tilde t + \frac{{\tilde r^2 }}{4},\quad r = \tilde r\sqrt {\tilde t},
\]
Eq. \eqref{eq:FRW} goes over to \eqref{eq:spacetime} %in standard Schwarzschild coordinates
%\begin{equation}\label{eq:FRW1}
%ds^2  =  - B(t,r) d t^2  + \frac{1}{A(t,r)} d r^2 + r^2 \left( d\theta ^2  + \sin ^2 \theta d\phi ^2  \right),
%\end{equation}
with the metric components
$$
A(t,r ) = 1 - v^2 ,\quad B(t,r ) = \frac{1}{{1 - v^2 }}.
$$
The exact fluid variables at $(t,r)$ are
\begin{equation}
\rho (t,r) = \frac{{16 v^2 }}{{ 3(1+\sigma^2)^2 \kappa r^2 }},\quad v(t,r) = \frac{{1 - \sqrt {1 - \xi ^2 } }}{\xi },
\label{eq:examp02}
\end{equation}
where $\kappa$ is Einstein's coupling constant, $\xi:=r/t$, and  $\sigma$  denotes
the parameter in \eqref{eq:EOS-thz02}.
This model is solved by using the proposed GRP scheme from $t=15$ to $16$
on several different uniform meshes for the spatial domain $[3,7]$.
The boundary conditions are specified at both ends by using the
exact solutions \eqref{eq:examp02} and $\sigma$  is taken as $1/\sqrt{3}$.
Table \ref{tab:FRW1} gives the numerical relative errors of $\rho,v,A$, and $B$ in $l^1$-norm  and  corresponding convergence rates. The results show that  the numerical convergence rates in $l^1$-norm  are almost {$(\Delta r)^2$}, in agreement with the theoretical.
%We also say that the convergence rate of the method is $h^2$.
%If the numerical method is of order 2,
\begin{table}[htbp]
  \centering
    \caption{\small Example \ref{example:FRW1}: Numerical errors in $l^1$-norm and corresponding convergence rates
      at $t=16$ for the GRP scheme. }
\begin{tabular}{|c||c|c||c|c||c|c||c|c|}
  \hline
\multirow{2}{8pt}{$N$}
 &\multicolumn{2}{c||}{$\rho$}&\multicolumn{2}{c||}{$v$} &\multicolumn{2}{c||}{$A$} &\multicolumn{2}{c|}{$B$}                    \\
 \cline{2-9}
 & error & order & error & order  & error &order &error & order \\
 \hline
25 &4.8775e-9& --     & 1.0383e-5   &--    &1.2692e-5& --    &9.2447e-6 &--\\
50&1.2695e-9& 1.94  & 2.7667e-6   & 1.91   &3.1843e-6& 1.99  &2.3011e-6 &2.01\\
100&3.2486e-10&1.97  &7.1233e-7    &1.96  &7.9744e-7& 2.00&5.7398e-7 &2.00\\
200&8.2267e-11&1.98  &1.8094e-7    &1.98 &1.9952e-7& 2.00&1.4334e-7 &2.00\\
400&2.0723e-11&1.99 &4.5522e-8  &1.99 &4.9895e-8& 2.00&3.5820e-8 &2.00\\
800&5.2016e-12&1.99&1.1409e-8  &2.00 &1.2476e-8& 2.00&8.9526e-9 &2.00\\
1600&1.3035e-12&2.00&2.8557e-9  &2.00 &3.1193e-9& 2.00&2.2379e-9 &2.00\\
\hline
\end{tabular}\label{tab:FRW1}
\end{table}

\end{example}

\begin{example}[FRW-2 model]\label{example:FRW2}\rm
The FRW metric  can also be transformed to the  standard Schwarzschild coordinates
under the coordinate transformation \cite{Vogler2010}
$$
t= \tilde r \sqrt{\tilde t}, \quad r = \frac{\Psi_0}{2} \sqrt{\frac{4 \tilde t^2+ \tilde t \tilde r^2}{\tilde t}},
$$
that is to say, the line element \eqref{eq:FRW} is
transformed to \eqref{eq:spacetime} with the metric components
\begin{equation}\label{eq:FRW2}
  A(t,r)=1-v^2,\quad B(t,r)=\frac{1}{\Psi (1-v^2) },
%ds^2  =  - \frac{1}{\Psi (1-v^2) } d t^2  + \frac{1}{1-v^2} d r^2 + r^2 \left( d\theta ^2  + %\sin ^2 \theta d\phi ^2  \right),
\end{equation}
where $\Psi(t,r) = \Psi_0 \sqrt{ \frac{\tilde t}{4 \tilde t^2+r^2} }$, and $\Psi_0$ is a positive constant and taken as 1 in the simulation.

The exact solutions in the fluid variables
corresponding to the above metric are
$$
\rho(t,r) = \frac{4}{3(1+\sigma^2)^2\kappa \tilde t^2}, \quad v(t,r) = \frac{r}{2 \tilde t},
$$
where
$$
\tilde t = \frac{ t^2+\sqrt{t^4-r^2\Psi_0^4} }{2 \Psi_0^2}.
$$
The spherically symmetric general RHD equations \eqref{eq:RHD_1}--\eqref{eq:RHD_1B} are solved by using the proposed GRP scheme from $t=15$ to $16$
on several different uniform meshes in the spatial interval $[3,7]$ with the boundary conditions specified by the
exact solutions. Table \ref{tab:FRW2} lists the numerical relative errors of $\rho,v,A$, and $B$ in $l^1$-norm  and  corresponding convergence rates.
 The results show that the numerical convergence rates of the GRP scheme
are almost {$(\Delta r)^2$}, which is the same as the theoretical, thus  the numerical method is of order 2.
%We also say that the convergence rate of the method is $h^2$.
%If the numerical method is of order 2,
\begin{table}[htbp]
  \centering
    \caption{\small Example \ref{example:FRW2}: Numerical errors in $l^1$-norm and corresponding convergence rates
      at $t=16$ for the GRP scheme.
  }
\begin{tabular}{|c||c|c||c|c||c|c||c|c|}
  \hline
\multirow{2}{8pt}{$N$}
 &\multicolumn{2}{c||}{$\rho$}&\multicolumn{2}{c||}{$v$} &\multicolumn{2}{c||}{$A$} &\multicolumn{2}{c|}{$B$}                    \\
 \cline{2-9}
 & error & order & error & order  & error &order &error & order \\
 \hline
25 &4.9541e-7& --     & 2.7875e-4   &--    &1.0705e-4& --    &4.9777e-5 &--\\
50&1.2027e-7& 2.04  & 6.7309e-5   & 2.05   &2.6922e-5& 1.99  &1.2251e-5 &2.02\\
100&2.9824e-8&2.01  &1.6539e-5    &2.02  &6.7379e-6& 2.00& 3.0235e-6 &2.02\\
200&7.4306e-9&2.00  &4.1037e-6    &2.01 &1.6857e-6& 2.00&7.5182e-7 &2.01\\
400&1.8551e-9&2.00 &1.0223e-6  &2.01 &4.2159e-7& 2.00& 1.8747e-7 &2.00\\
800&4.6353e-10&2.00&2.5514e-7  &2.00 &1.0541e-7& 2.00& 4.6800e-8 &2.00\\
1600&1.1584e-10&2.00&6.3738e-8  &2.00 &2.6357e-8& 2.00& 1.1694e-8 &2.00\\
\hline
\end{tabular}\label{tab:FRW2}
\end{table}

\end{example}

\begin{example}[TOV model]\label{example:TOV}\rm
The  general relativistic version of TOV model
describes the static singular isothermal spheres  \cite{SmollerTemple1993}.
Two components  of the TOV metrics
are
\begin{equation}\label{eq:TOV}
A(t,r)=1-8 \pi {\cal G} \gamma,
\quad
B(t,r)= B_0 r^{\frac{4 \sigma}{1+\sigma}},
%ds^2  =  - B_0 r^{\frac{4 \sigma}{1+\sigma}}  d t^2  + \frac{1}{1-8 \pi \cal G \gamma } d r^2 %+ r^2 \left( d\theta ^2  + \sin ^2 \theta d\phi ^2  \right).
\end{equation}
where the parameter $\gamma$ is  related to $\sigma$ in \eqref{eq:EOS-thz02} by
\begin{equation}\label{eq:DEFgamma}
\gamma = \frac{1}{2 \pi {\cal G}} \left( \dfr{\sigma^2}{1+6\sigma^2 +\sigma^4} \right).
\end{equation}
%Assuming $p = \sigma^2 \rho$, these are
%given by
The exact solutions in the fluid variables are given by
\begin{equation}\label{eq:TOV-thz01}
\rho (t,r) = \frac{\gamma} {r^2}, \quad v(t,r)=0.
\end{equation}
In our computations,  $B_0$ and $\sigma$ are taken as 1 and $1/\sqrt{3}$, respectively.

\begin{table}[htbp]
  \centering
    \caption{\small Example \ref{example:TOV}: Numerical errors in $l^1$-norm and corresponding convergence rates
      at $t=16$ for the GRP scheme.
  }
\begin{tabular}{|c||c|c||c|c||c|c||c|c|}
  \hline
\multirow{2}{8pt}{$N$}
 &\multicolumn{2}{c||}{$\rho$}&\multicolumn{2}{c||}{$v$} &\multicolumn{2}{c||}{$A$} &\multicolumn{2}{c|}{$B$}                    \\
 \cline{2-9}
 & error & order & error & order  & error &order &error & order \\
 \hline
25 &4.4342e-7& --     & 6.5575e-4   &--    &2.5962e-5& --    &1.1524e-3 &--\\
50&1.1136e-7& 1.99  & 1.6311e-4   & 2.01   &6.2838e-6& 2.05  &2.6952e-4 &2.10\\
100&2.7877e-8&2.00  &4.0716e-5    &2.00  &1.5688e-6& 2.00& 6.5011e-5 &2.05\\
200&6.9757e-9&2.00  &1.0174e-5    &2.00 &3.9366e-7& 2.00&1.5952e-5 &2.03\\
400&1.7449e-9&2.00 &2.5427e-6  &2.00 &9.8718e-8& 2.00& 3.9504e-6 &2.01\\
800&4.3635e-10&2.00&6.3548e-7  &2.00 &2.4722e-8& 2.00& 9.8298e-7 &2.01\\
1600&1.0911e-10&2.00&1.5886e-7  &2.00 & 6.1869e-9& 2.00& 2.4515e-7 &2.00\\
\hline
\end{tabular}\label{tab:TOV}
\end{table}

Numerical experiments are conducted  by the proposed GRP scheme from the time $t=15$ to $16$
on the different uniform meshes in the interval $[3,7]$, where the boundary conditions are specified by using the
exact solutions \eqref{eq:TOV-thz01}.
Table \ref{tab:TOV} shows the numerical relative errors of $\rho,v,A$, and $B$ in $l^1$-norm  and   corresponding convergence rates. We see that the numerical convergence rates of the  proposed GRP scheme are almost {$(\Delta r)^2$}.
\end{example}

\begin{example}[Shock wave model]\label{example:FRW1-TOV}\rm
This test will simulate the general relativistic shock waves. The setup of
the problem is the same as that in \cite{VoglerTemple2012}.
The initial conditions at $t=t_0$ are as follows
\begin{equation}\label{eq:FRW1-TOV-rhov}
\big(\rho ({t_0},r), v({t_0},r)\big) =
\begin{cases}
\big(\frac{3  v^2 }{\kappa r^2 }, \frac{{1 - \sqrt {1 - {\xi ^2}} }}{\xi }\big), & r<r_0,\\
\big(\frac{\gamma}{r^2},0\big), & r>r_0,
\end{cases}
\end{equation}
and
\begin{align}\label{eq:FRW1-TOV-AB}
\begin{aligned}
A({t_0},r) =&
\begin{cases}
1-v^2, & r<r_0,\\
1-8\pi {\cal G} \gamma, & r>r_0,
\end{cases}\\
B ({t_0},r) =&
\begin{cases}
\frac{1}{1-v^2}, & r<r_0,\\
B_0 r^{ \frac{4\sigma^2}{1+\sigma^2} }, & r>r_0,
\end{cases}
\end{aligned}
\end{align}
which  match the FRW-1 and  TOV metrics if the left limiting value of the initial fluid velocity at the discontinuity $r=r_0$
 from the FRW-1 side is given by
$$ v_0 := v(t_0,r_0-0)=\sqrt{8\pi {\cal G}\gamma},$$
with the parameter $\gamma$ defined in \eqref{eq:DEFgamma},
and the starting time $t_0$ and the parameter $B_0$ are taken as
$$ t_0 = \frac{ r_0 (1+v_0^2) }{2 v_0}, \quad B_0 = r_0^{ -{4\sigma^2}/{1+\sigma^2} }/(1-v_0^2),$$
respectively.
In numerical computations,
the computational domain is chosen as $[3,7]$, the boundary
conditions are specified at $r=3$ and 7 by using the exact FRW-1 and TOV solutions
\eqref{eq:examp02} and \eqref{eq:TOV-thz01}, respectively,
and the initial discontinuity is located at $r_0=5$.

  \begin{figure}[!htbp]
    \centering
        \includegraphics[height=0.38\textwidth]{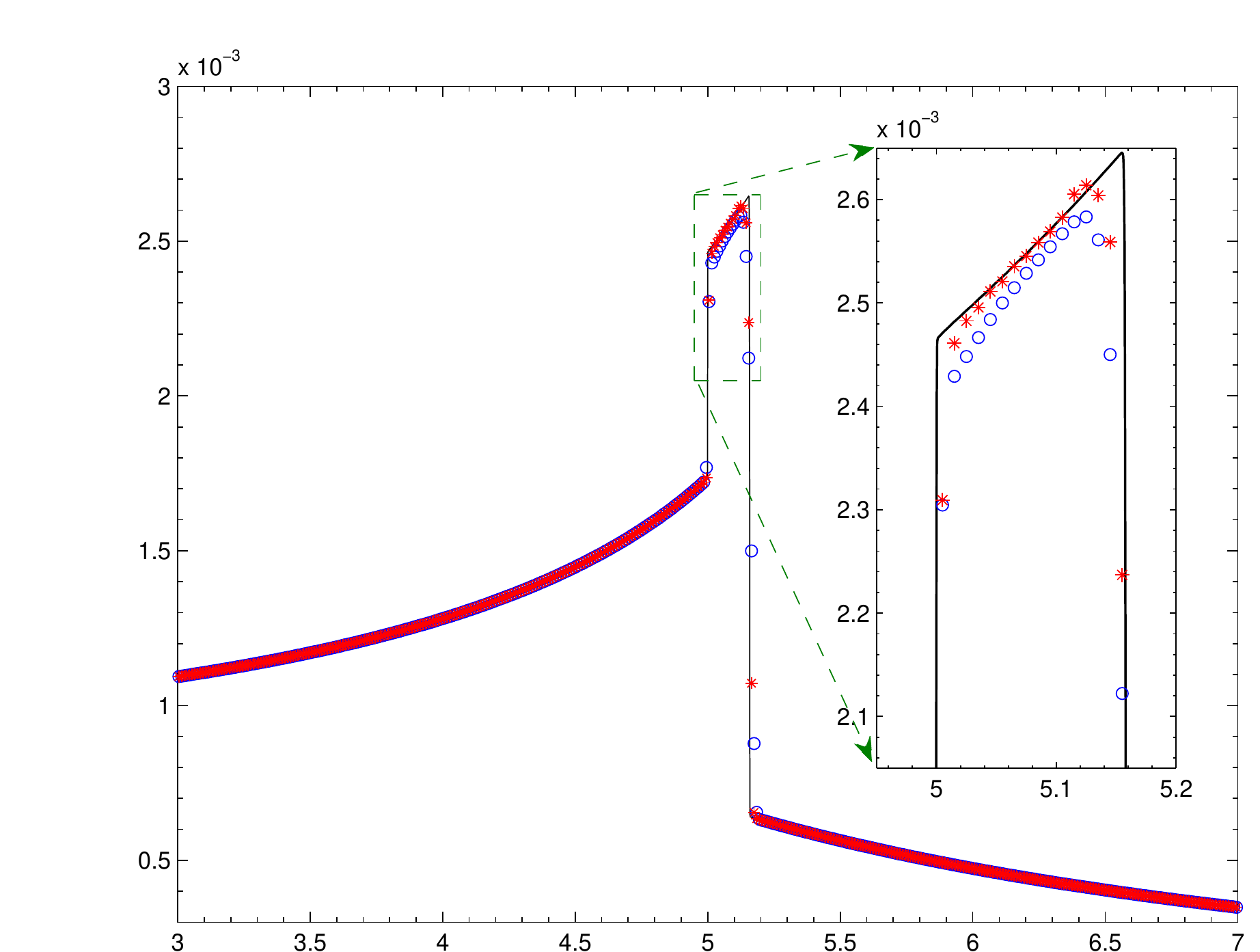}
        \includegraphics[height=0.38\textwidth]{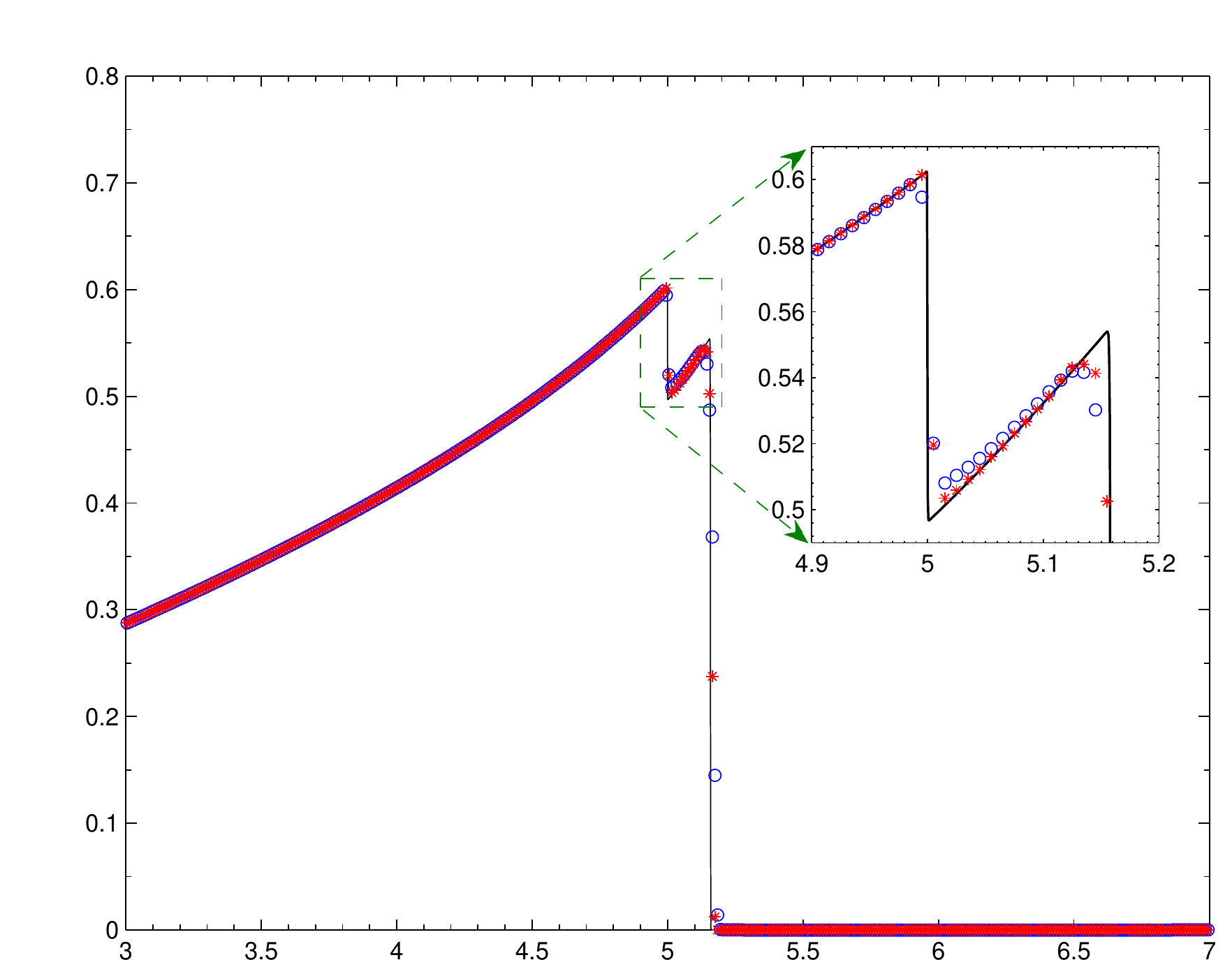}
    \caption{Example \ref{example:FRW1-TOV}:
    The rest energy density $\rho$ (left) and velocity $v$ (right) at $t=t_0+0.2$.
    The numerical solutions obtained by the   Godunov and
     GRP schemes
are drawn in the symbols  ``{\color{blue}$\circ$}'' and ``{\color{red}$\ast$}'', respectively,
while the solid lines stand for the reference solutions given by the Godunov scheme with a fine mesh of 10000 uniform cells.}
    \label{fig:shockwave_rhov0.2}
  \end{figure}
  \begin{figure}[!htbp]
    \centering
        \includegraphics[height=0.38\textwidth]{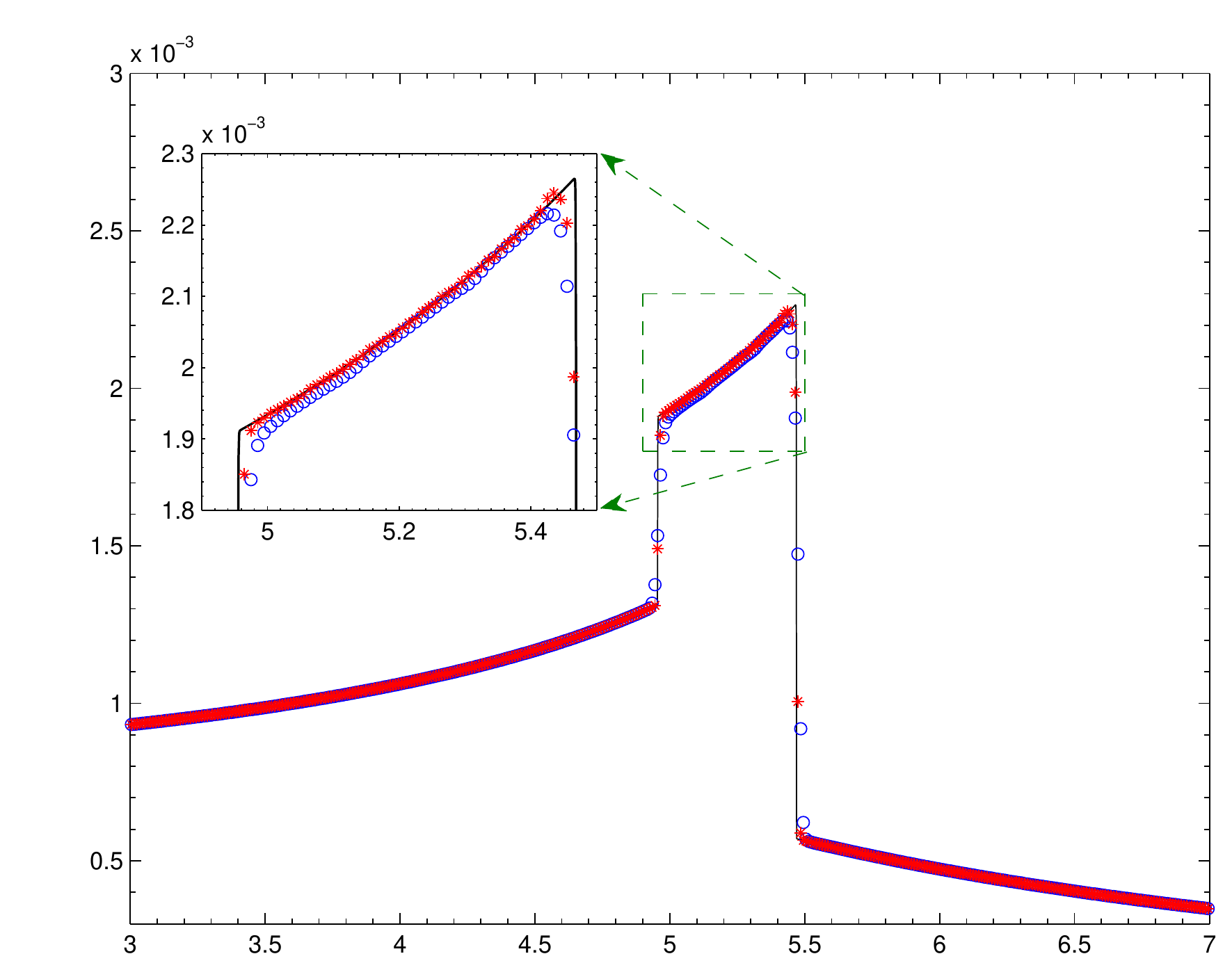}
        \includegraphics[height=0.38\textwidth]{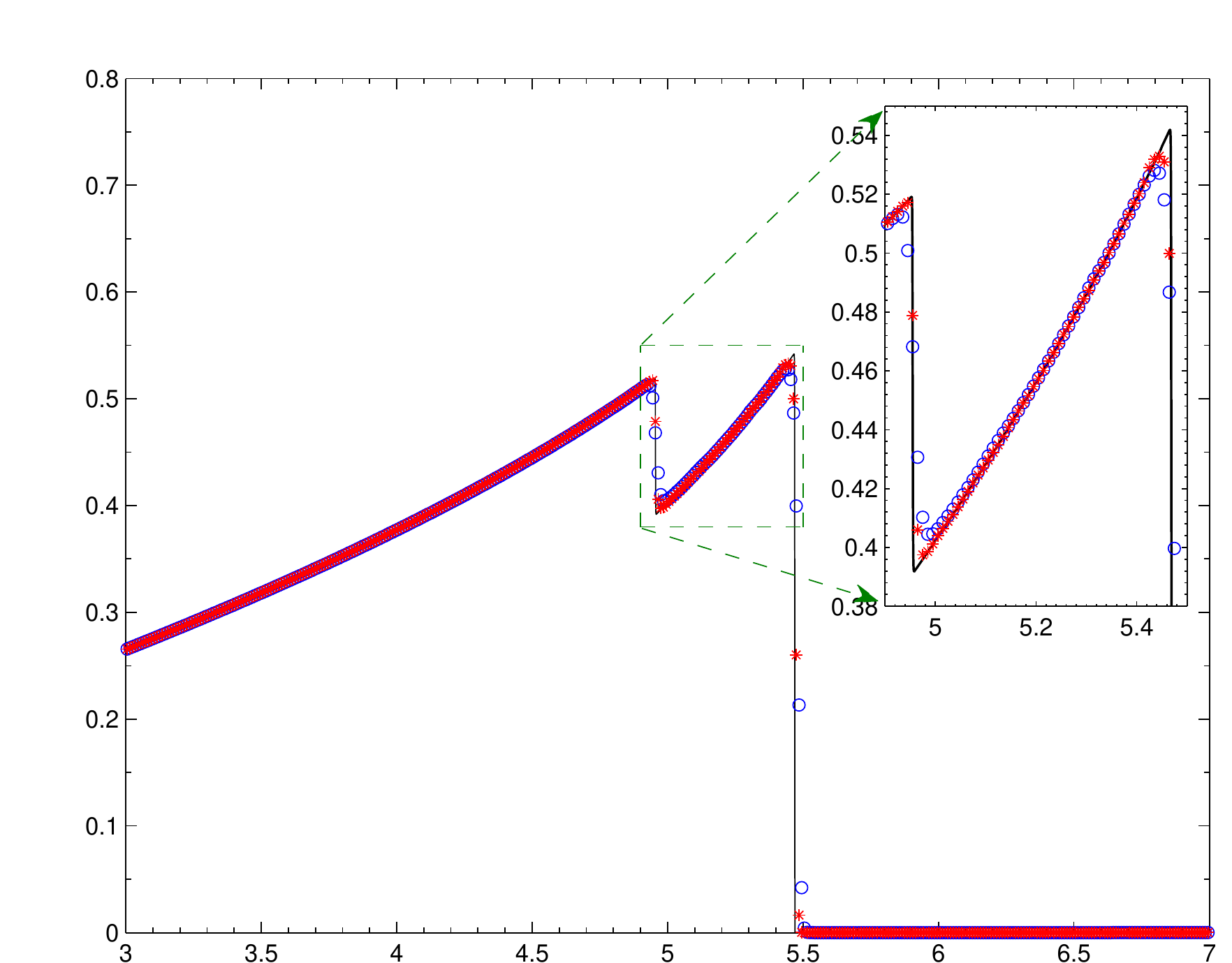}
    \caption{Same as Fig. \ref{fig:shockwave_rhov0.2}, except for the output time
    $t=t_0+0.6$.}
    \label{fig:shockwave_rhov0.6}
  \end{figure}

  \begin{figure}[!htbp]
    \centering
        \includegraphics[height=0.37\textwidth]{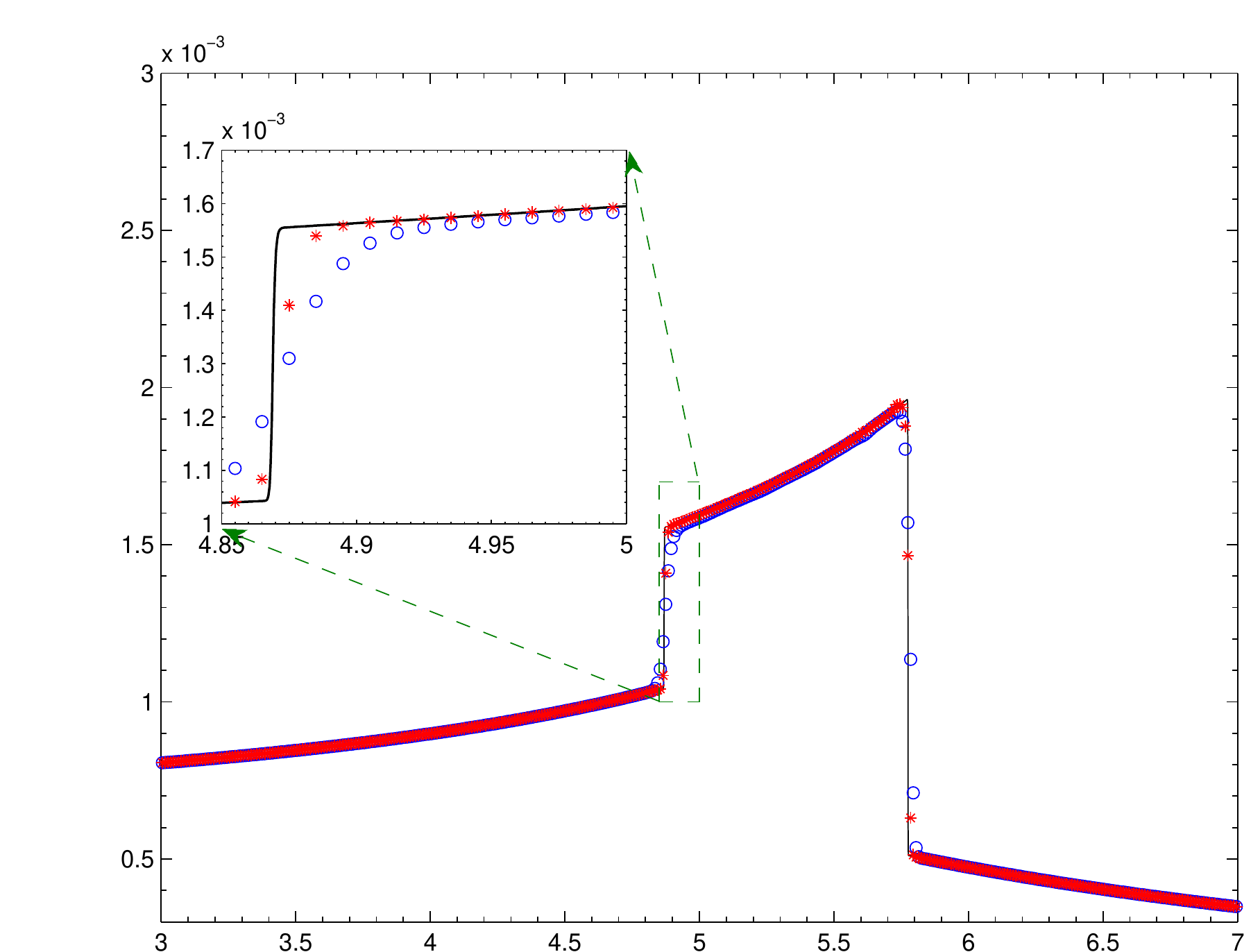}
        \includegraphics[height=0.37\textwidth]{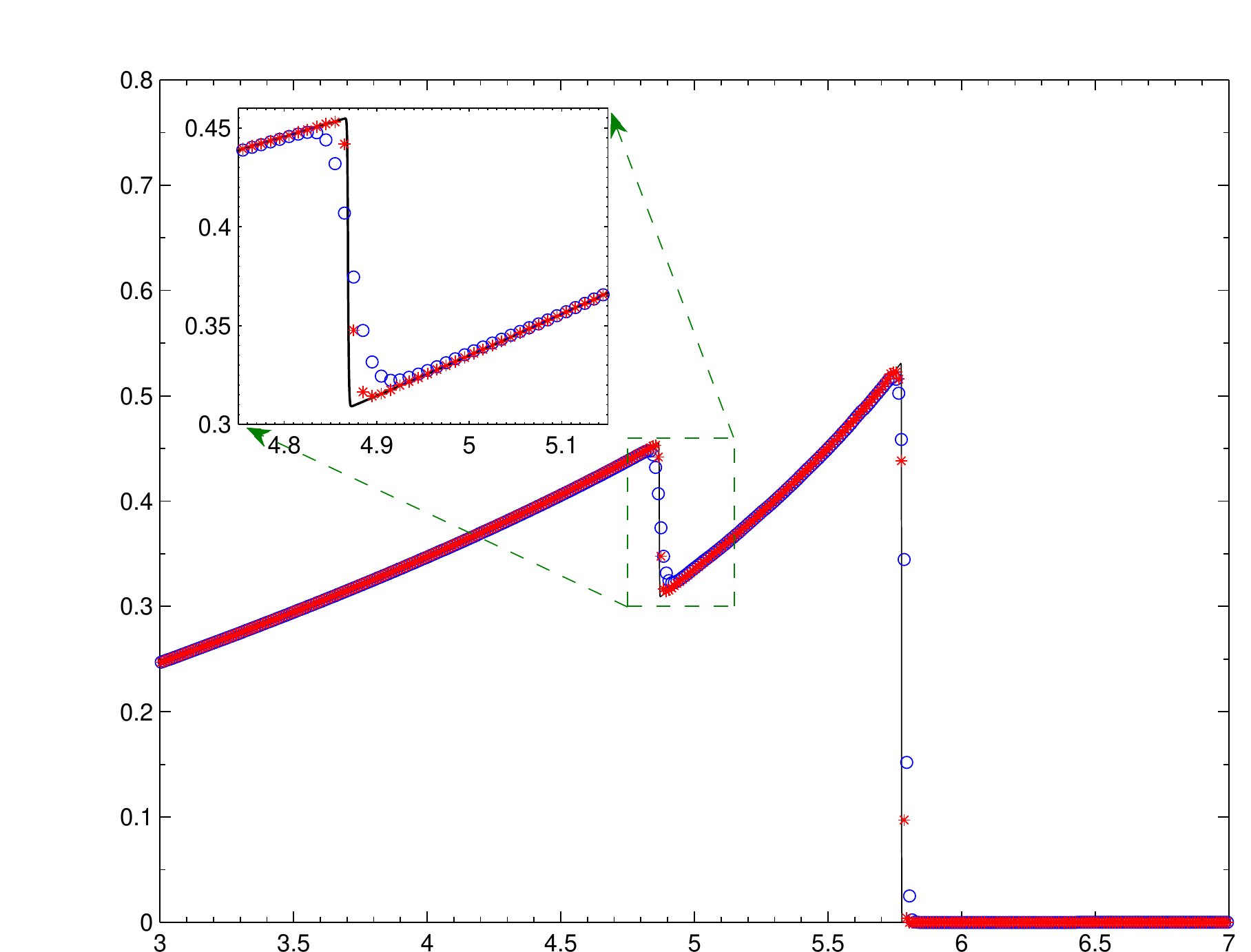}
    \caption{Same as Fig. \ref{fig:shockwave_rhov0.2}, except for $t=t_0+1$.}
    \label{fig:shockwave_rhov1.0}
  \end{figure}

  \begin{figure}[!htbp]
    \centering
        \includegraphics[height=0.37\textwidth]{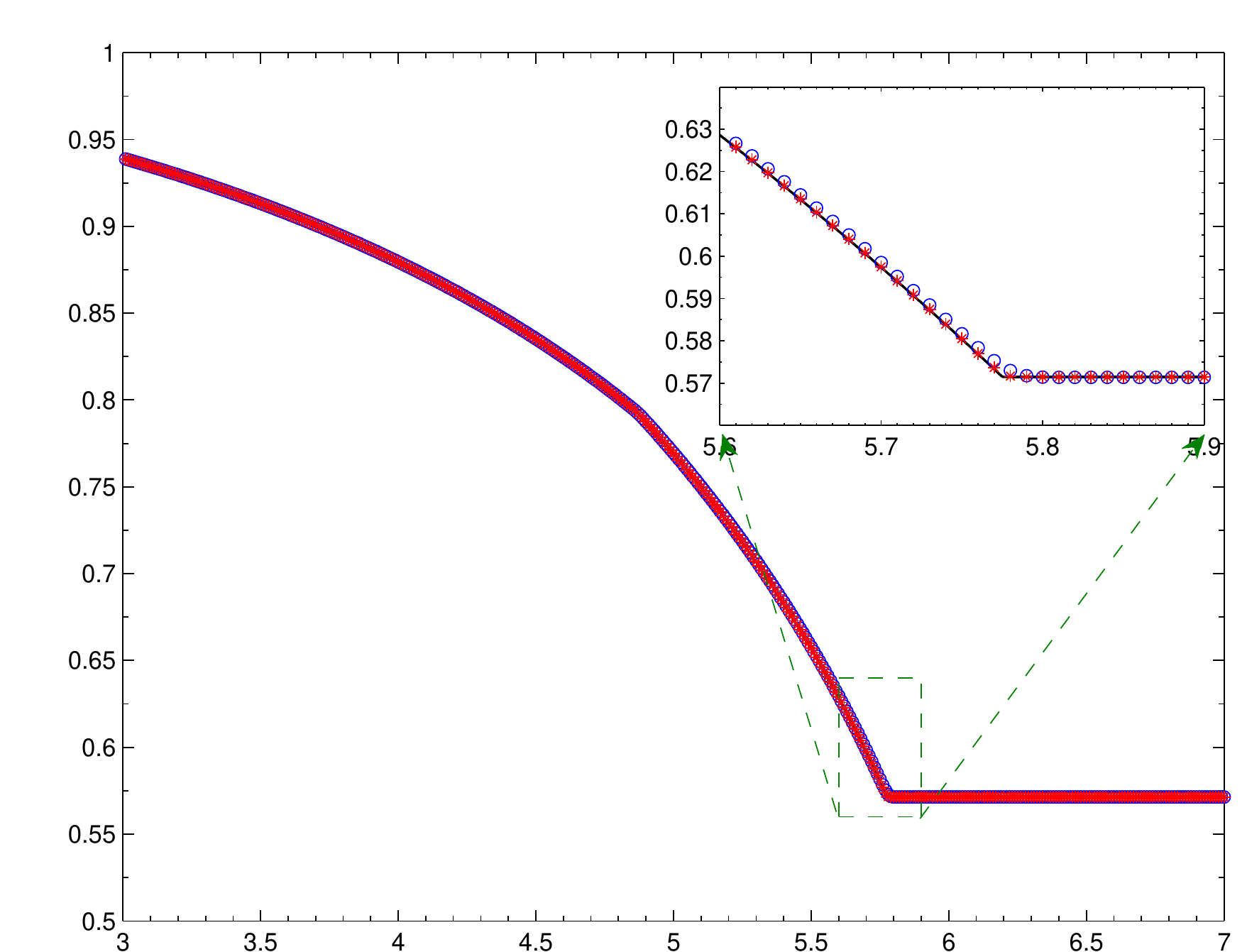}
        \includegraphics[height=0.38\textwidth]{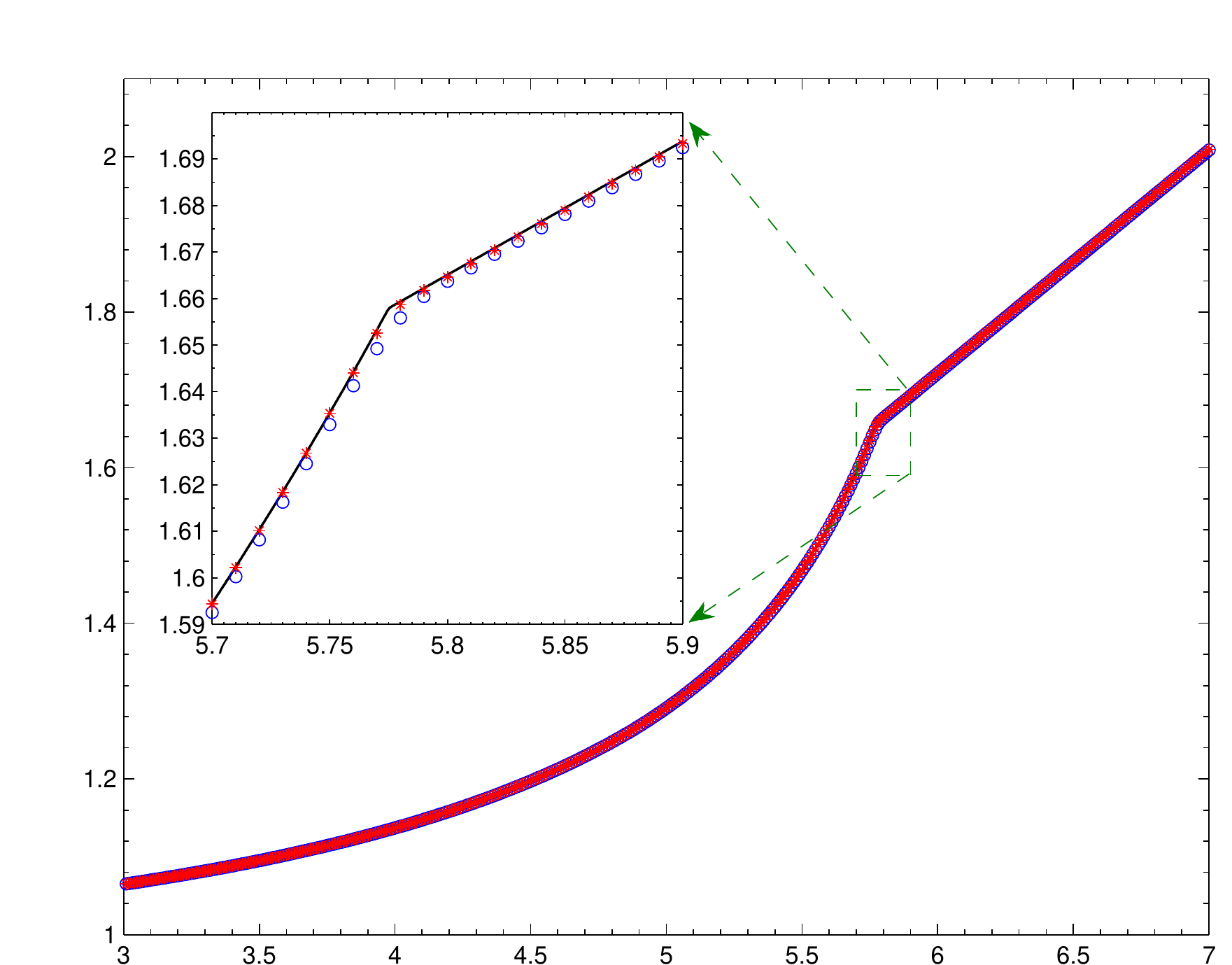}
    \caption{Same as Fig. \ref{fig:shockwave_rhov0.2}, except for the metric functions $A$ (left) and $B$ (right) at $t=t_0+1$.}
    \label{fig:shockwave_AB}
  \end{figure}

  \begin{figure}[!htbp]
    \centering
        \includegraphics[height=0.38\textwidth]{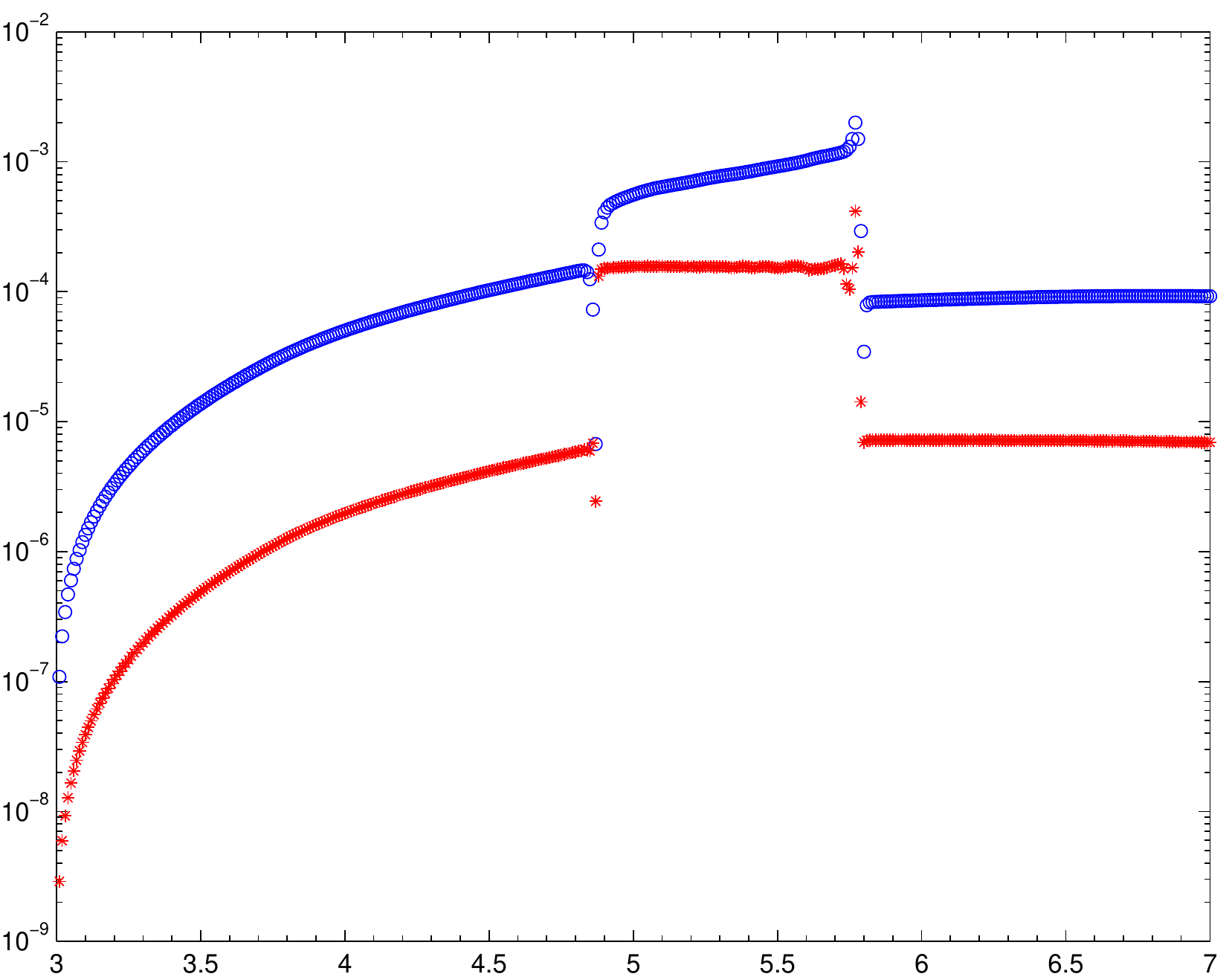}
        \includegraphics[height=0.38\textwidth]{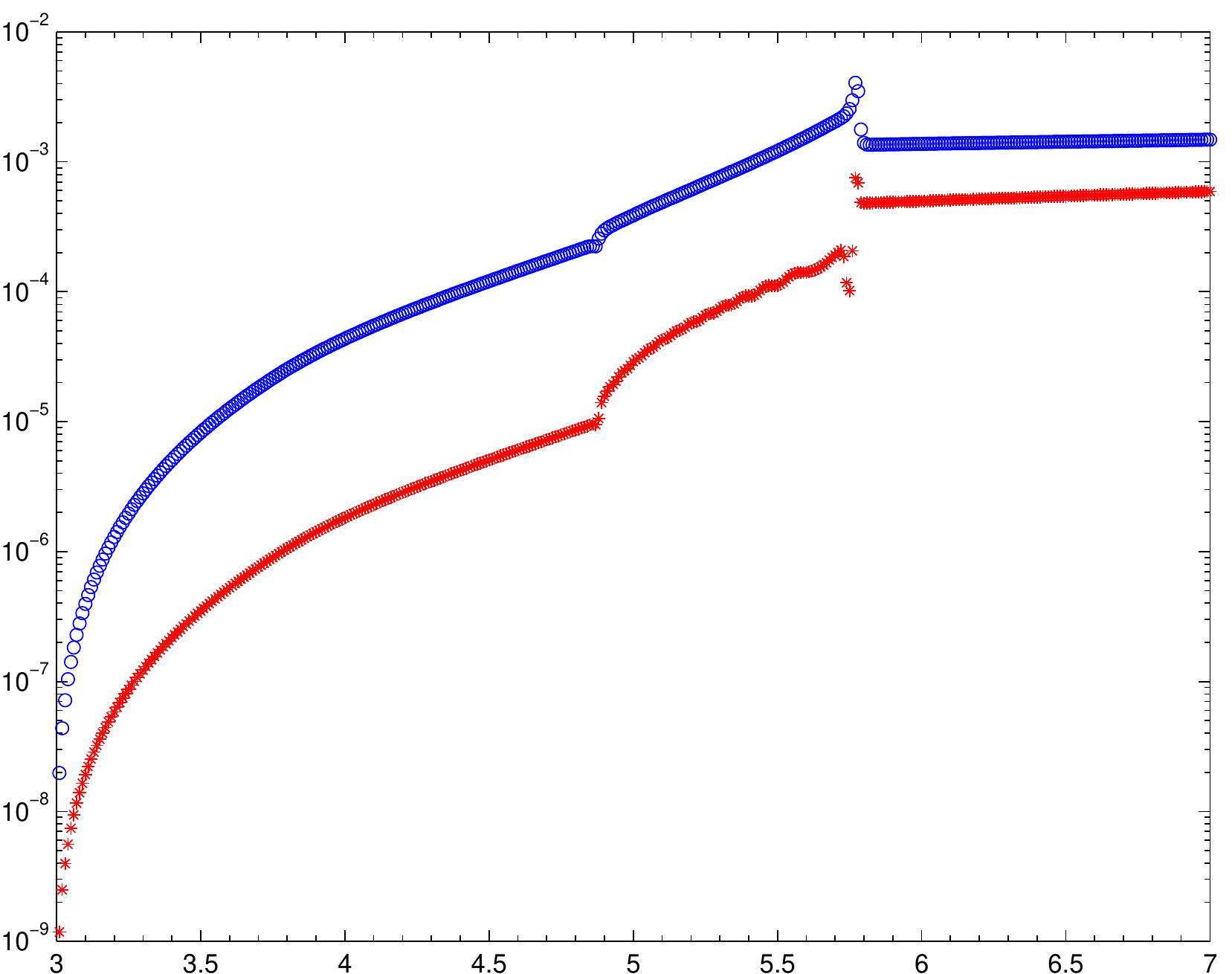}
    \caption{Example \ref{example:FRW1-TOV}:
    The numerical errors in the metric functions $A$ (left) and $B$ (right) at $t=t_0+1$.
    Those of the Godunov   and   GRP schemes
are  drawn in the symbols ``{\color{blue}$\circ$}'' and ``{\color{red}$\ast$}'', respectively.}
    \label{fig:shockwave_ErrAB}
  \end{figure}

Figs. \ref{fig:shockwave_rhov0.2}--\ref{fig:shockwave_rhov1.0}  show the numerical solutions
at $t=t_0 + 0.2,~t_0+0.6$, and $t_0+1$,  obtained
by using the first-order accurate Godunov scheme and the
    second-order accurate GRP scheme with  400 uniform cells, respectively.
It can be seen that the results agree well with the reference solutions and
the GRP scheme resolves the relativistic shock waves better than the Godunov scheme.

Figs. \ref{fig:shockwave_AB} and \ref{fig:shockwave_ErrAB} display the numerical solutions and
  corresponding errors  between the numerical   and   reference solutions in the metric functions $A$ and $B$
at $t=t_0+1$ by using the Godunov scheme and the GRP scheme with  400 uniform cells, respectively. These results show that
the GRP scheme is much more accurate than the Godunov scheme, and
the errors given by the GRP scheme is about ten percent of those computed by the Godunov scheme.

  \begin{figure}[!htbp]
    \centering
        \includegraphics[width=0.618\textwidth]{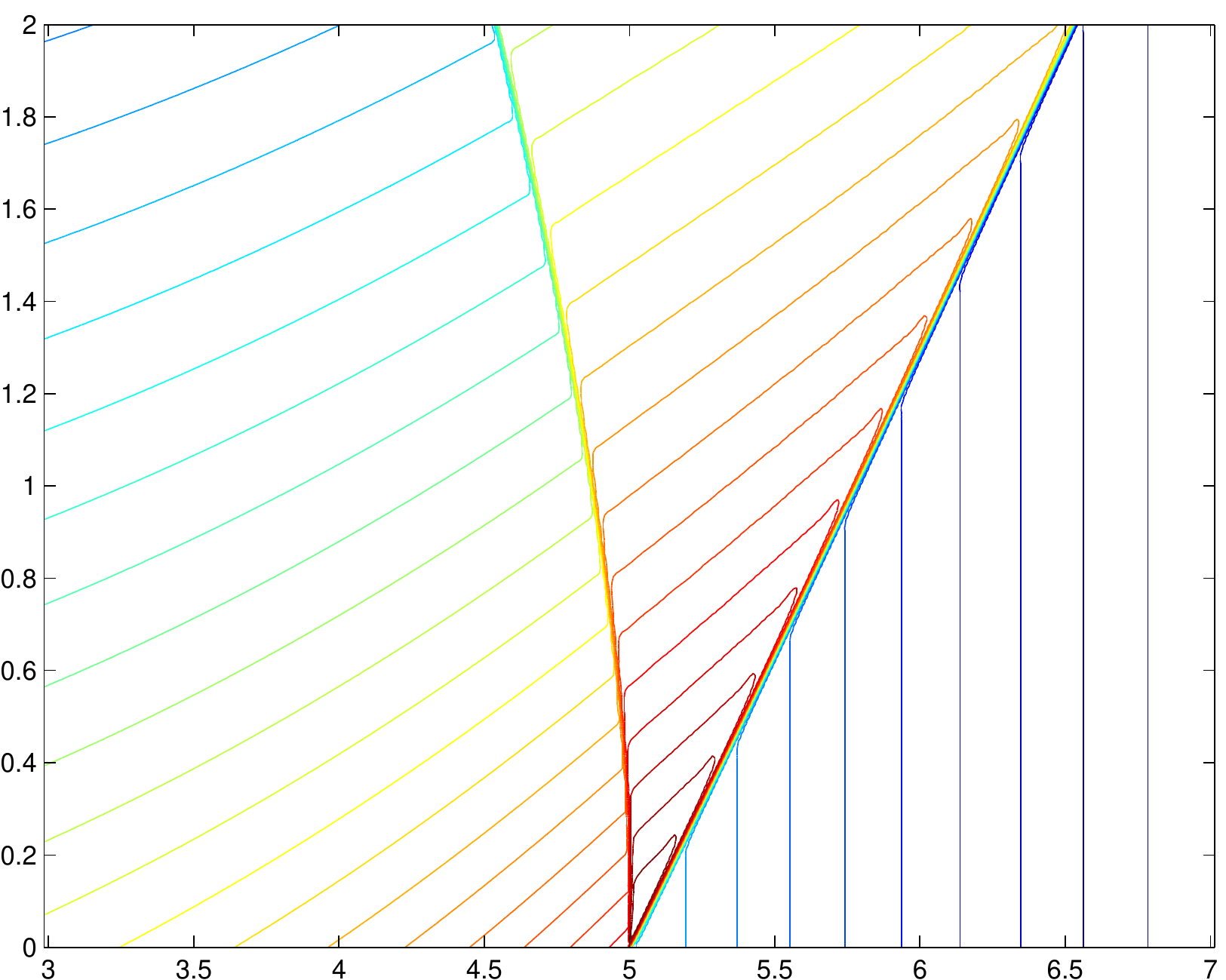}
    \caption{Example \ref{example:FRW1-TOV}:
    The contour of the rest energy density logarithm within the spacetime domain $[3,7]\times [t_0,t_0+2]$.
%    space-time diagram of the logarithmic density; 25 contour lines have
%been selected equidistantly
    }
    \label{fig:shockwave_Contour}
  \end{figure}

Fig. \ref{fig:shockwave_Contour} gives the contour of the rest energy density logarithm within the spacetime domain $[3,7]\times [t_0,t_0+2]$.
Two relativistic  shock waves are initially formed, and
the stronger and straight   moves to the right (the TOV region) while the weaker and curved propagates  in the FRW-1 region. A expanding pocket of higher density  is  produced and
then interacting with both  FRW-1 and TOV metrics.

  \begin{figure}[!htbp]
    \centering
        \includegraphics[height=0.38\textwidth]{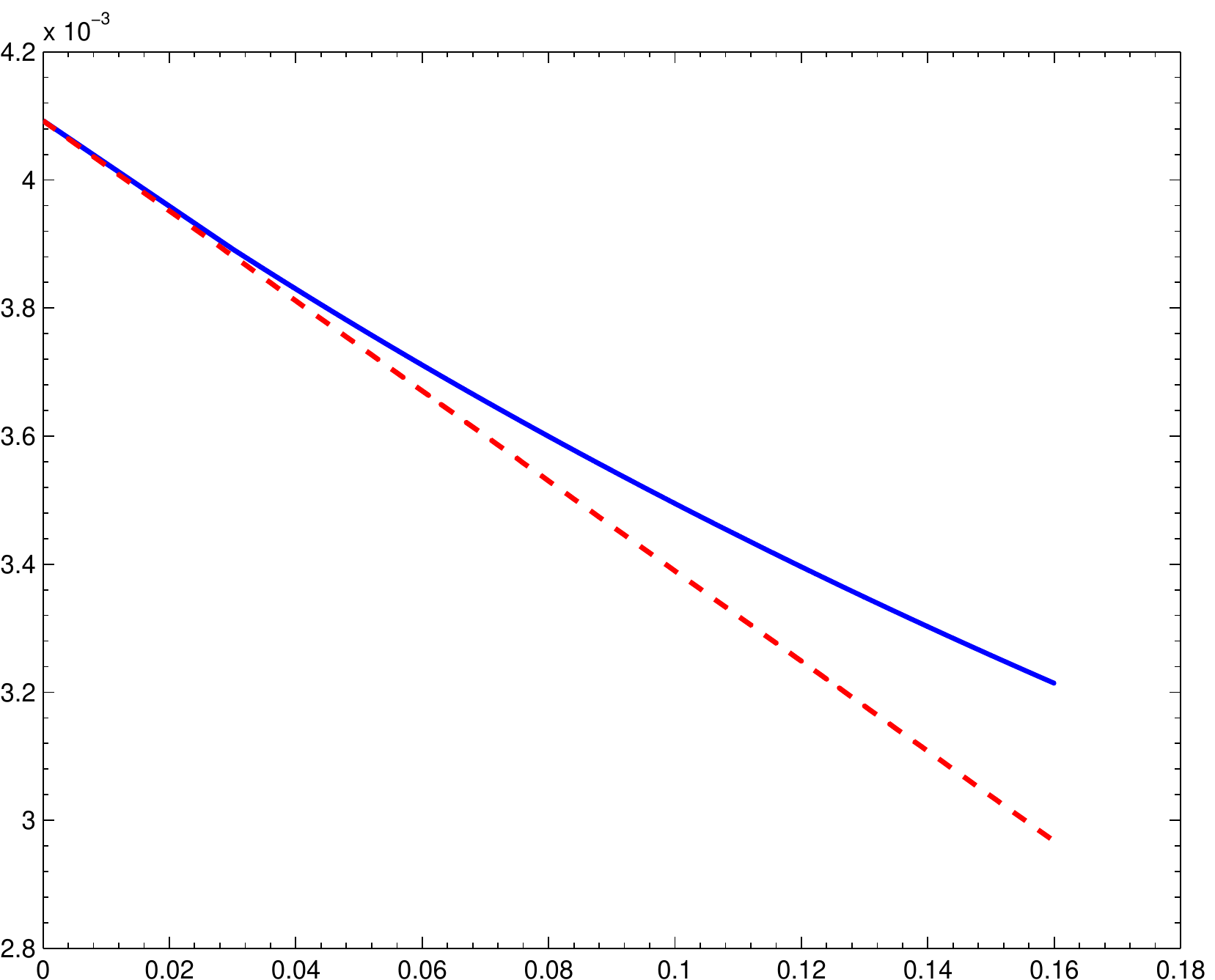}
        \includegraphics[height=0.38\textwidth]{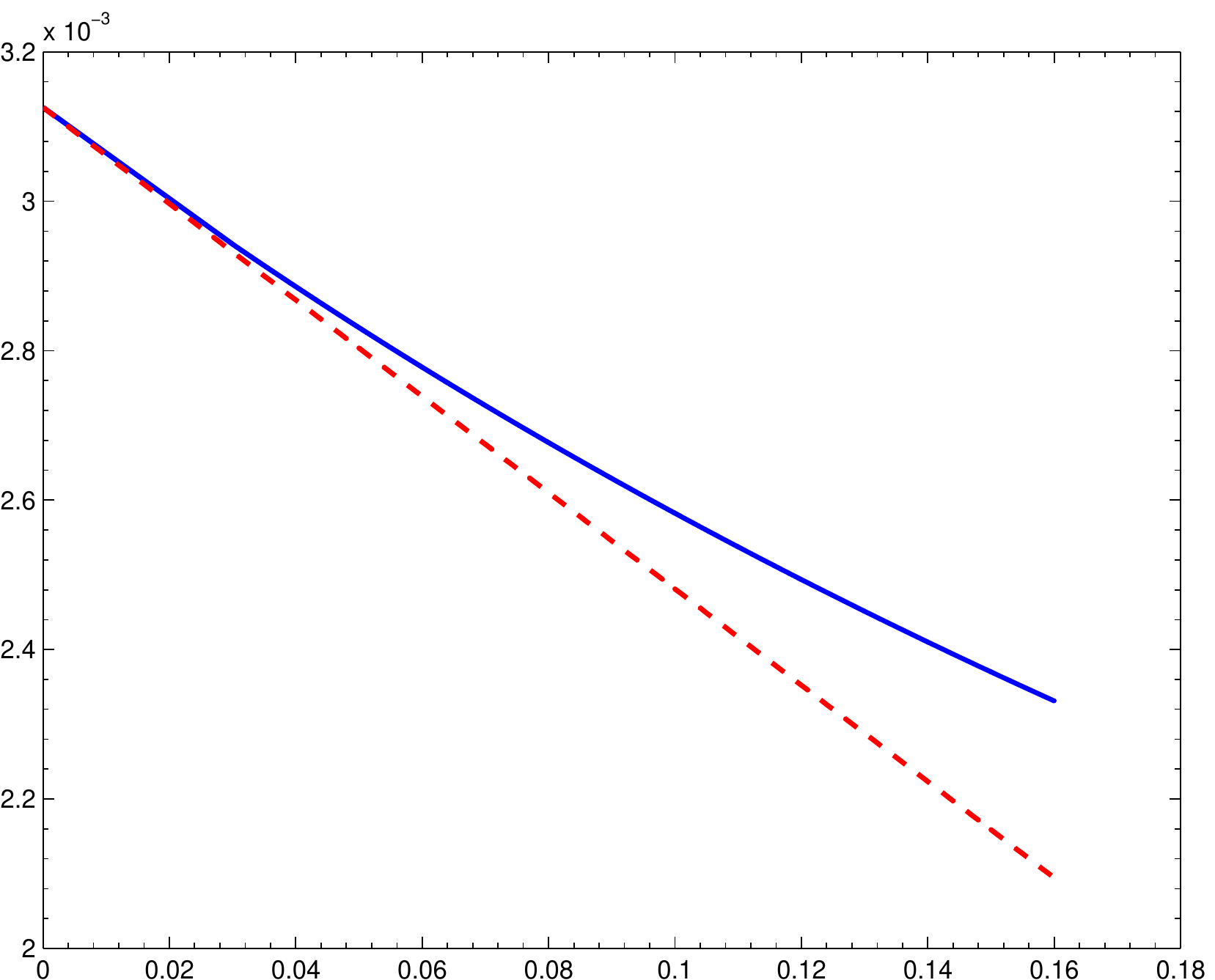}
    \caption{Example \ref{example:FRW1-TOV}: Comparison of the reference   (solid lines) and GRP solver based solutions (dash lines). The left and right are  the first
    and the second  components of $\vec U(t_0 + \tau,r_0)$, respectively.
    The horizontal axis is the simulation time $\tau$.}
    \label{fig:shockwave_GRPsolu}
  \end{figure}

Since this problem itself is  GRP, it may be used to validate the accuracy of our GRP solver presented in Section \ref{sec:PuPt}.
Fig.  \ref{fig:shockwave_GRPsolu} gives a comparison of
 the GRP solver based solutions at the interface
\[
{\vec U^{\mbox{\tiny GRP}}}({t_0} + \tau,r_0 ) = {\vec U^{\mbox{\tiny RP}}} + {\left( {\frac{{\partial \vec U}}{{\partial t}}} \right)^{\mbox{\tiny GRP}}}\tau, \]
  with the reference numerical solutions denoted by $\vec U^{\mbox{\tiny REF}}({t_0} + \tau,r_0 )$,
where ${\vec U^{\mbox{\tiny RP}}}$ is the exact solution to the (classical) RP \eqref{eq:InitialData-RP}
with $A_*=A(t_0,r_0), B_*=B(t_0,r_0)$, and the initial data $\vec U_L$ and $\vec U_R$ given by the left- and right-side limits at $r=r_0$ of the initial
data \eqref{eq:FRW1-TOV-rhov}, and ${\left( {{{\partial \vec U}}/{{\partial t}}} \right)^{\mbox{\tiny GRP}}}$ is computed by resolving the GRP at $r=r_0$.
The   reference solutions  $\vec U^{\mbox{\tiny REF}}({t_0} + \tau,r_0 )$ are computed by a second-order accurate MUSCL method with the Godunov flux on a very fine uniform mesh with the spatial
step-size of $10^{-6}$.
%The GRP solver based solutions $\vec U^{\mbox{\tiny GRP}}({t_0} + \tau,r_0 )$ are plotted in dash lines in Fig. \ref{fig:shockwave_GRPsolu}
%in comparison with
%
 Table \ref{tab:GRPsolver1} lists the errors  approximately evaluated by
$$
e_{\mbox{\tiny GRP}} (\tau) = \left\| \vec U^{\mbox{\tiny GRP}}({t_0} + \tau,r_0 ) - \vec U^{\mbox{\tiny REF}}({t_0} + \tau,r_0 ) \right\|_2,
$$
and corresponding convergence rates of the GRP solver. The results show that the GRP solver has second-order accuracy
with
respect to $\tau$ and thus
the accuracy of the limiting values ${\left( {{{\partial \vec U}}/{{\partial t}}} \right)^{\mbox{\tiny GRP}}}$ is validated.

\begin{table}[htbp]
  \centering
    \caption{\small Example \ref{example:FRW1-TOV}: Numerical errors and corresponding convergence rates of the GRP solver based solution.
  }
\begin{tabular}{|c||c|c|c|c|c|c|c|c|}
  \hline
  {$\tau$}
 & 0.16  &0.14   & 0.12  & 0.1  & 0.08  &0.06  & 0.04 &0.02 \\
 \hline
{$e_{\mbox{\tiny GRP}} (\tau)$} &
 3.42e-4&2.70e-4&2.04e-4&
 1.46e-4&9.67e-5& 5.62e-5 &2.59e-5 & 6.70e-6                 \\
order &--&  1.78 & 1.80 & 1.83& 1.86& 1.88&
1.92& 1.95                        \\
\hline
\end{tabular}\label{tab:GRPsolver1}
\end{table}

\end{example}

\begin{example}[Time reversal model]\label{example:TimeRev}\rm
The example considers reversing time and running the matched spacetime given in Example \ref{example:FRW1-TOV}. For this purpose,
the setup of this problem is the same as that for Example \ref{example:FRW1-TOV} except that the sign of the time is changed from
positive to negative. In other words, the initial time for this test is $t_0 = - \frac{ r_0 (1+v_0^2) }{2 v_0}$.

  \begin{figure}[!htbp]
    \centering
        \includegraphics[width=0.618\textwidth]{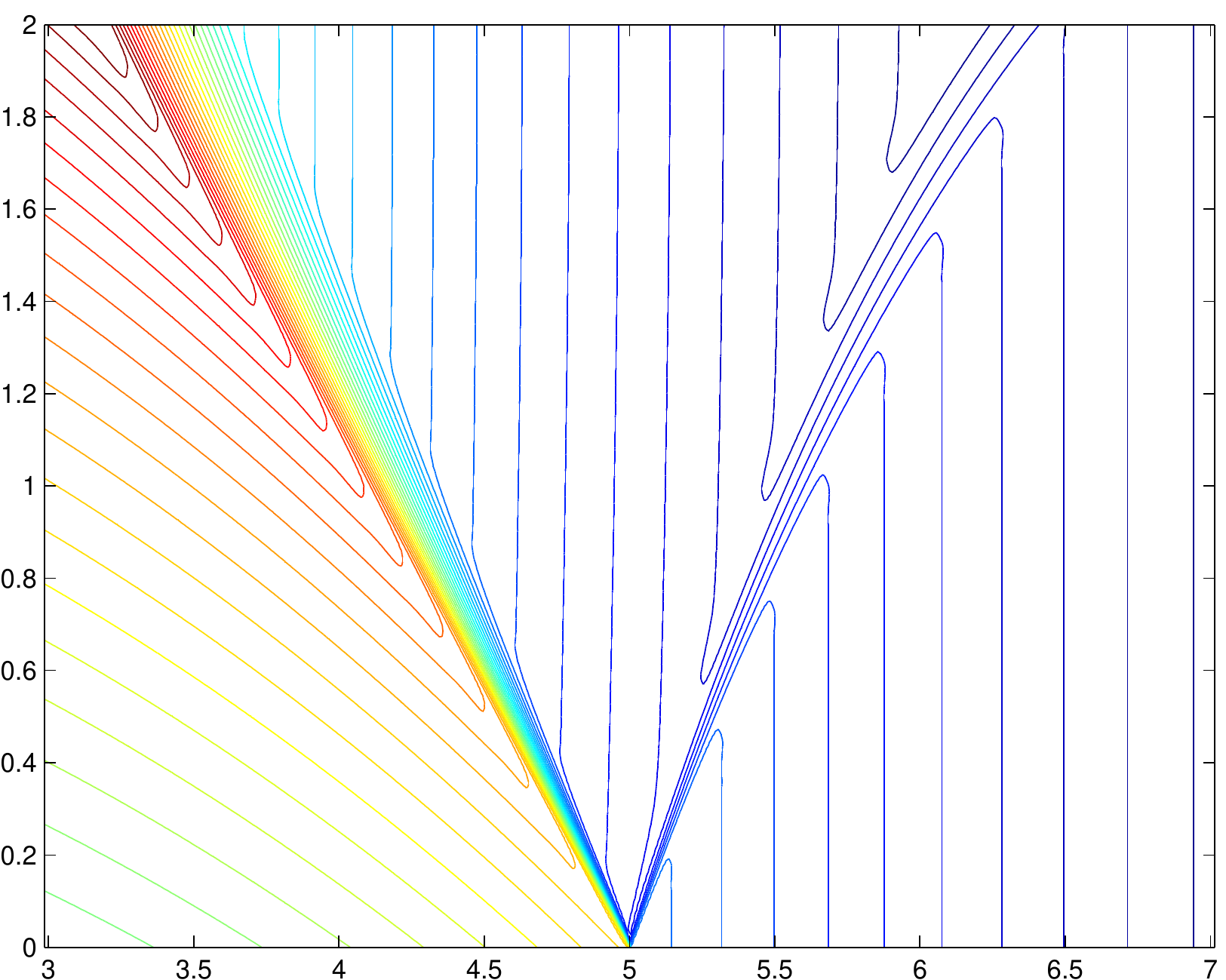}
    \caption{Same as Fig. \ref{fig:shockwave_Contour}, except for Example \ref{example:TimeRev}.}
    \label{fig:timerev_Contour}
  \end{figure}

   \begin{figure}[!htbp]
    \centering
        \includegraphics[height=0.37\textwidth]{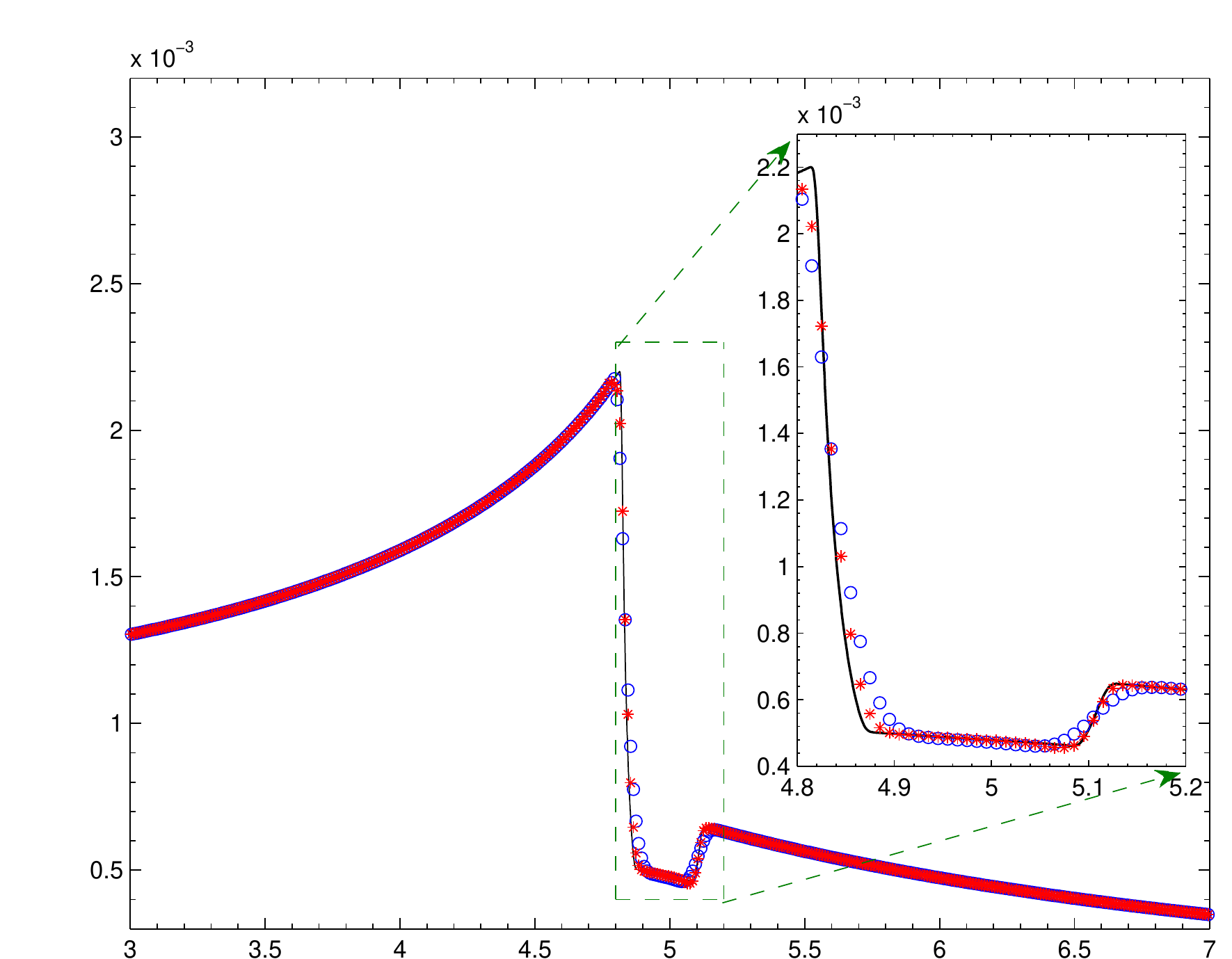}
        \includegraphics[height=0.38\textwidth]{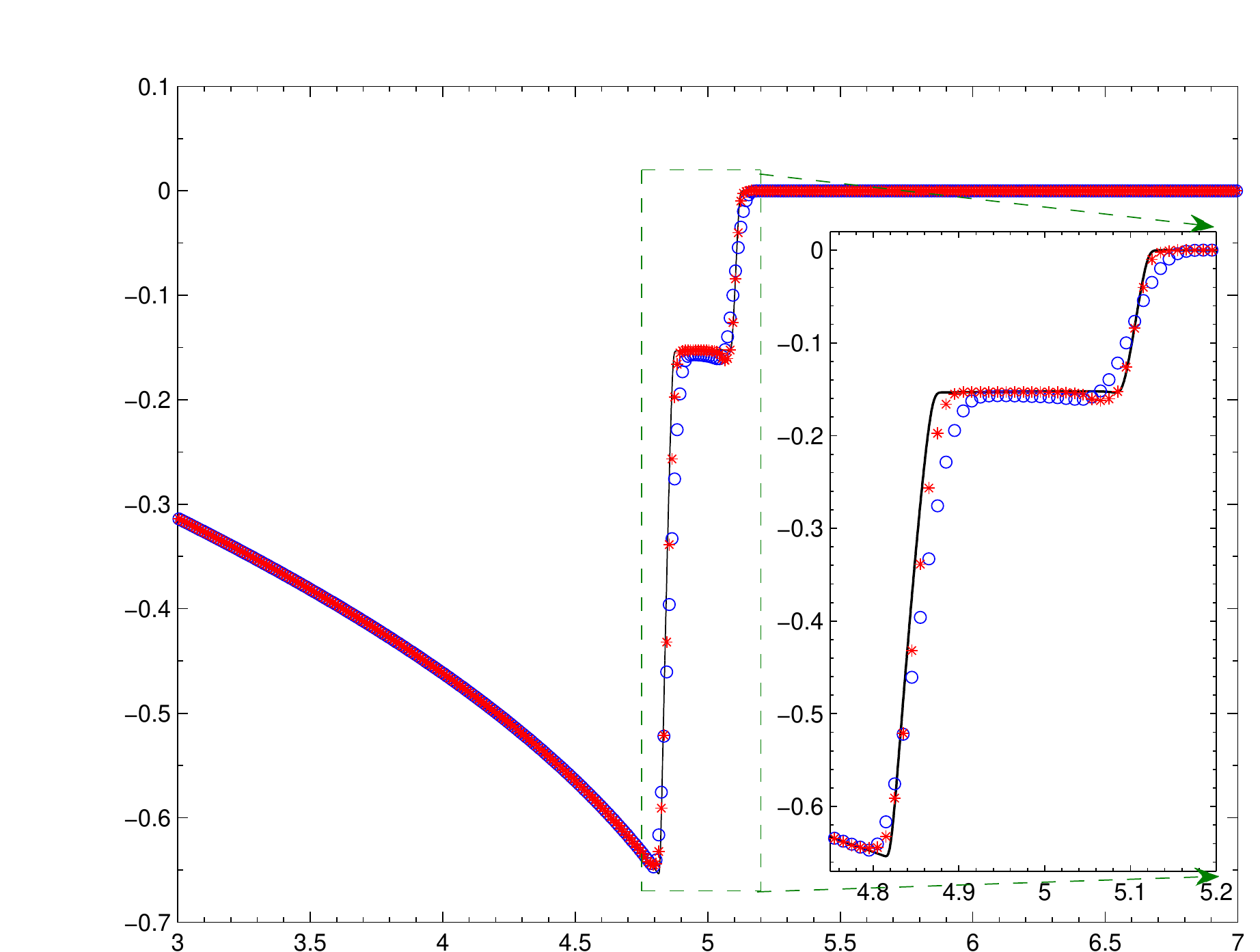}
    \caption{Same as Fig. \ref{fig:shockwave_rhov0.2}, except for Example \ref{example:TimeRev}.}
    \label{fig:timerev_rhov0.2}
  \end{figure}
  \begin{figure}[!htbp]
    \centering
        \includegraphics[height=0.38\textwidth]{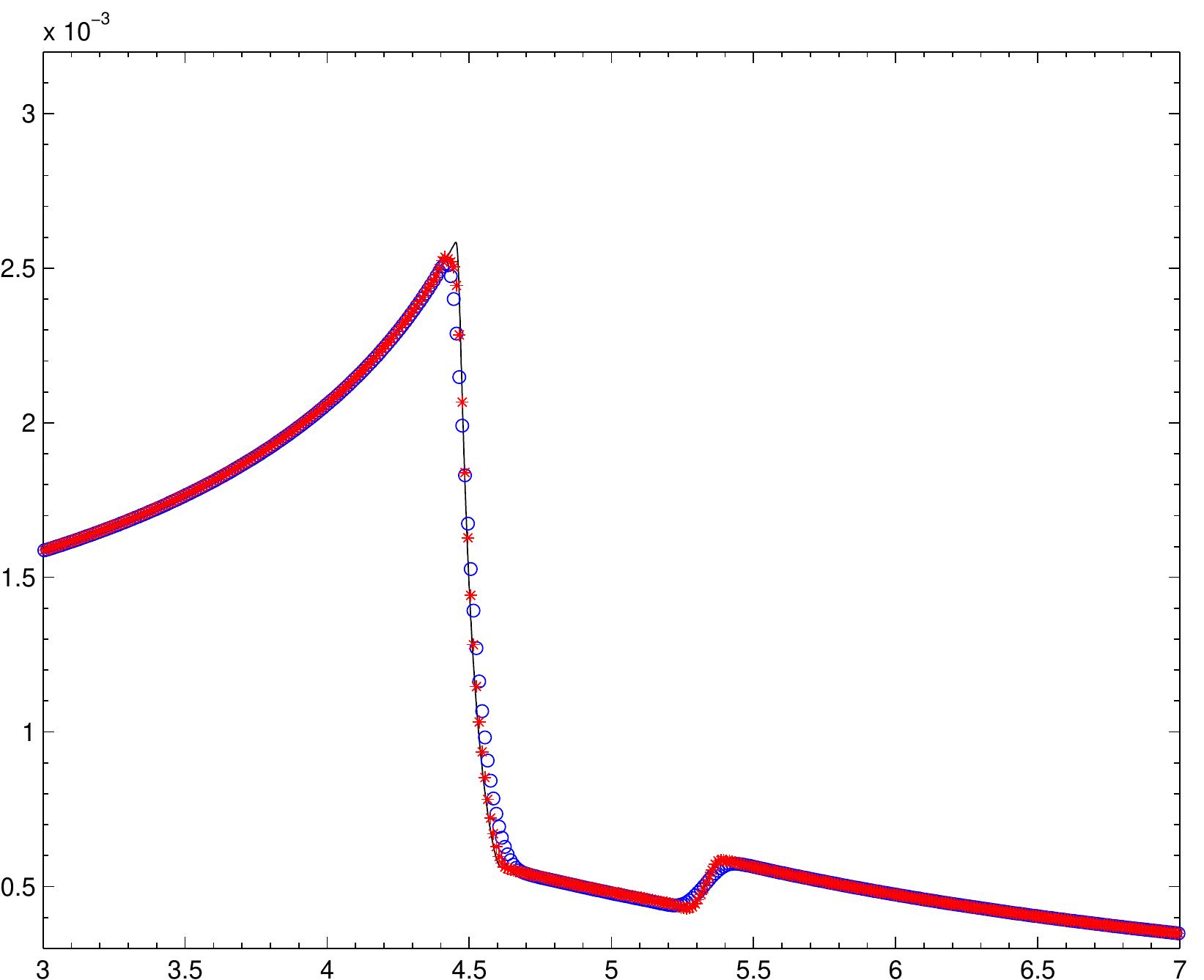}
        \includegraphics[height=0.37\textwidth]{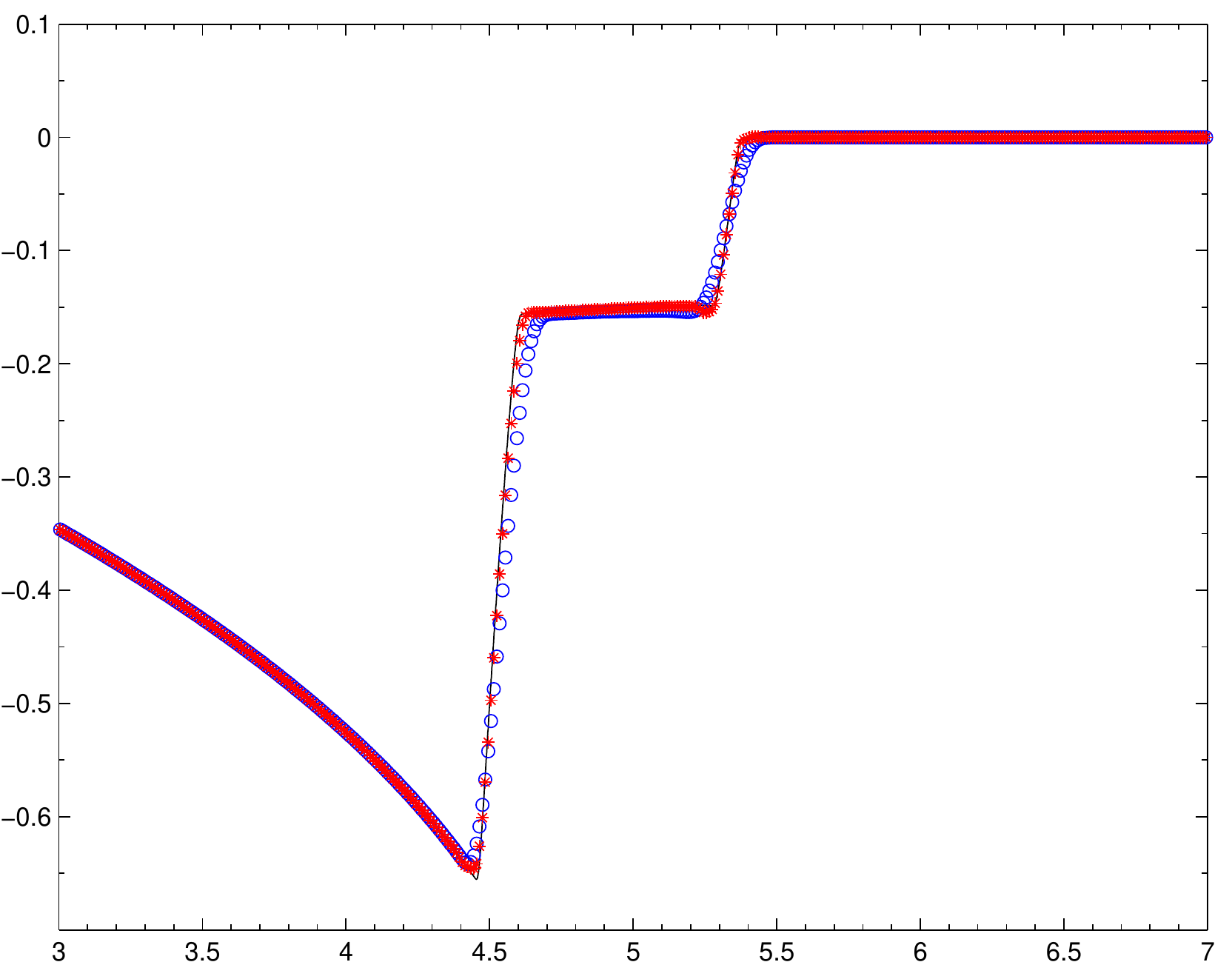}
    \caption{Same as Fig. \ref{fig:shockwave_rhov0.2}, except for Example \ref{example:TimeRev} and $t=t_0+0.6$.}
    \label{fig:timerev_rhov0.6}
  \end{figure}

  \begin{figure}[!htbp]
    \centering
        \includegraphics[height=0.38\textwidth]{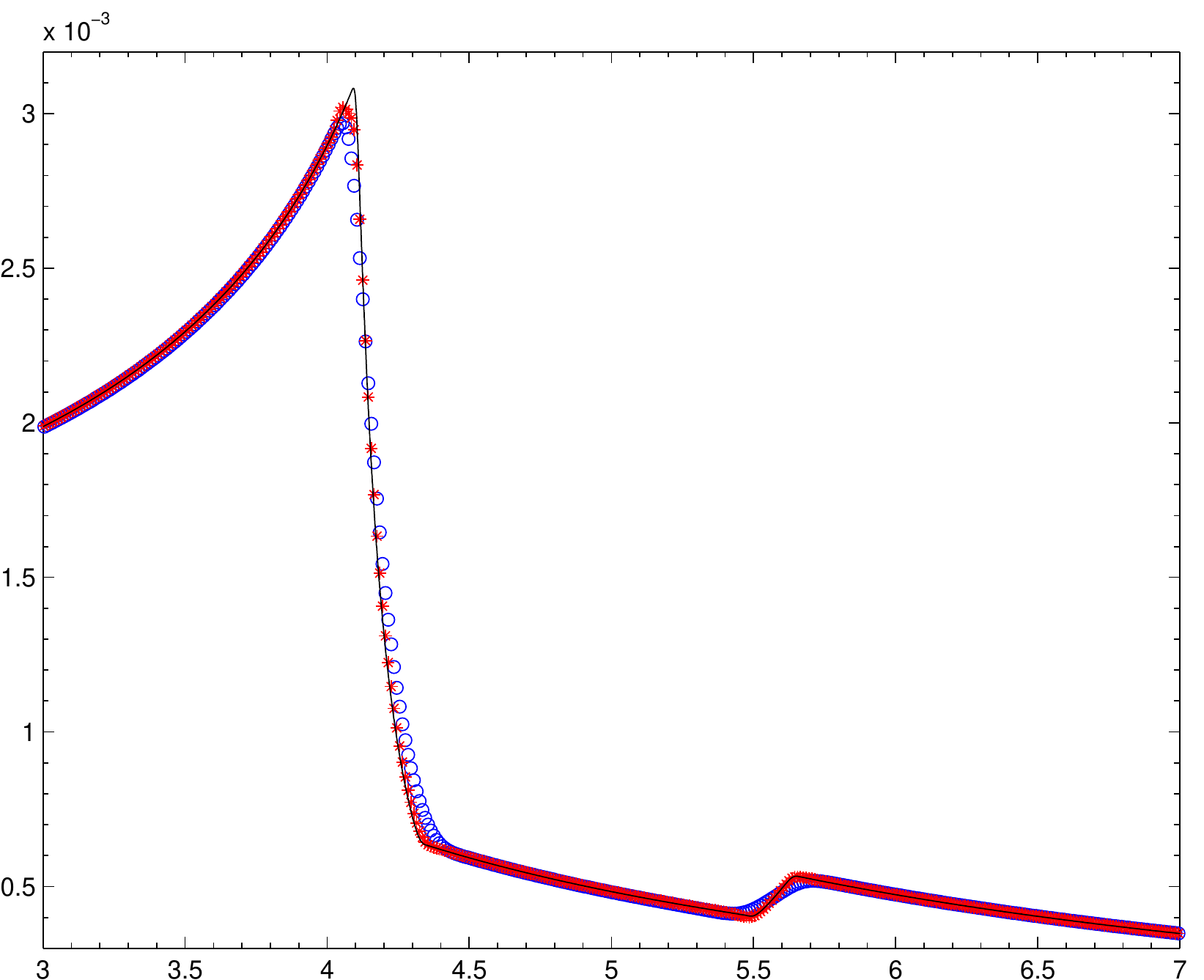}
        \includegraphics[height=0.37\textwidth]{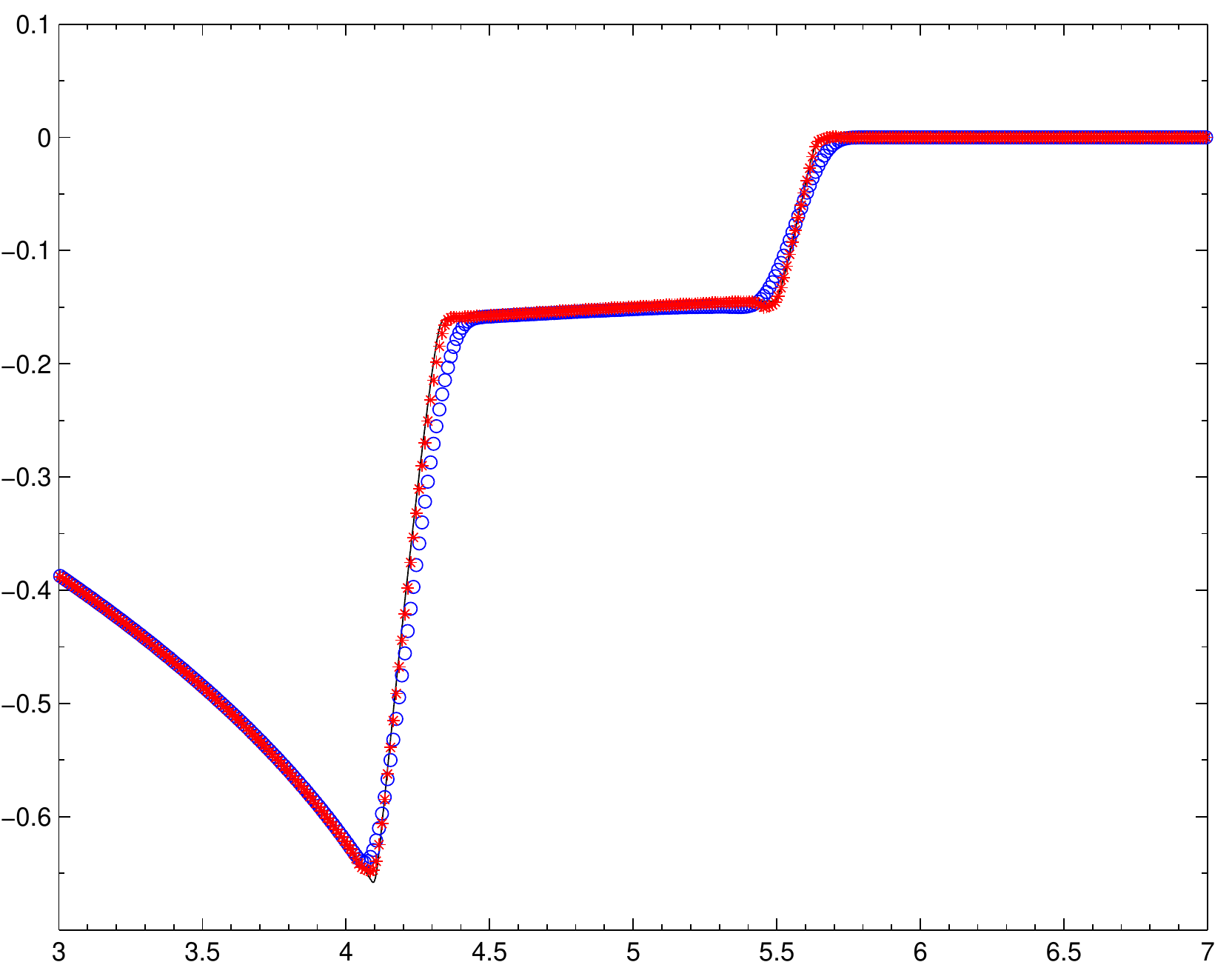}
    \caption{Same as Fig. \ref{fig:shockwave_rhov0.2}, except for Example \ref{example:TimeRev} and $t=t_0+1$.}
    \label{fig:timerev_rhov1.0}
  \end{figure}

Fig. \ref{fig:timerev_Contour} gives the contour of the rest energy density logarithm within the spacetime domain $[3,7]\times [t_0,t_0+2]$
given by the GRP scheme with 400 uniform cells.
From the initial discontinuity, two relativistic rarefaction waves are formed:
a stronger wave moving to the left and connecting  the FRW and  TOV regions while a
weaker propagating  in the TOV region.  A pocket of lower
density is produced between the two rarefaction waves.
Those are different from the forward time case in Example \ref{example:FRW1-TOV}, where
 two shock waves are surrounding a region of higher density.
Figs. \ref{fig:shockwave_rhov0.2}--\ref{fig:shockwave_rhov1.0}  show the numerical solutions of the fluid variables
at $t=t_0 + 0.2,~t_0+0.6$, and $t_0+1$
by using the Godunov  and   GRP schemes with  400 uniform cells, respectively.
We see that the results agree well with the reference solutions and
the GRP scheme exhibits better resolution than the Godunov scheme.
As the time increases, both the rarefaction waves and the pocket are
expanding. At the left of this pocket, there is a density spike,   growing and
moving in the FRW region, and more matter is falling into the center of
the universe as the time increase. If  the more long time simulation is done,
a black hole will be eventually produced  \cite{Vogler2010}.

  \begin{figure}[!htbp]
    \centering
        \includegraphics[height=0.37\textwidth]{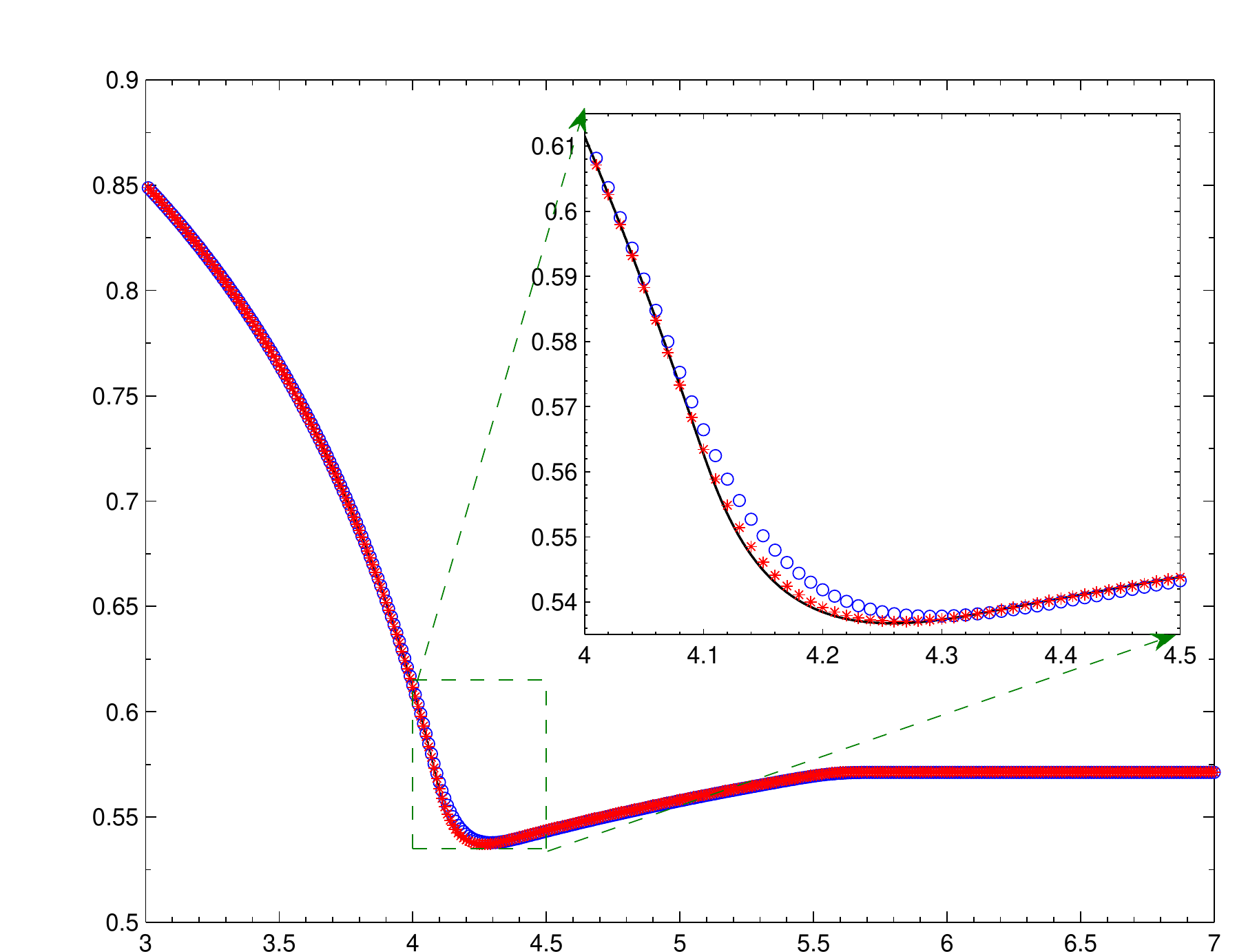}
        \includegraphics[height=0.37\textwidth]{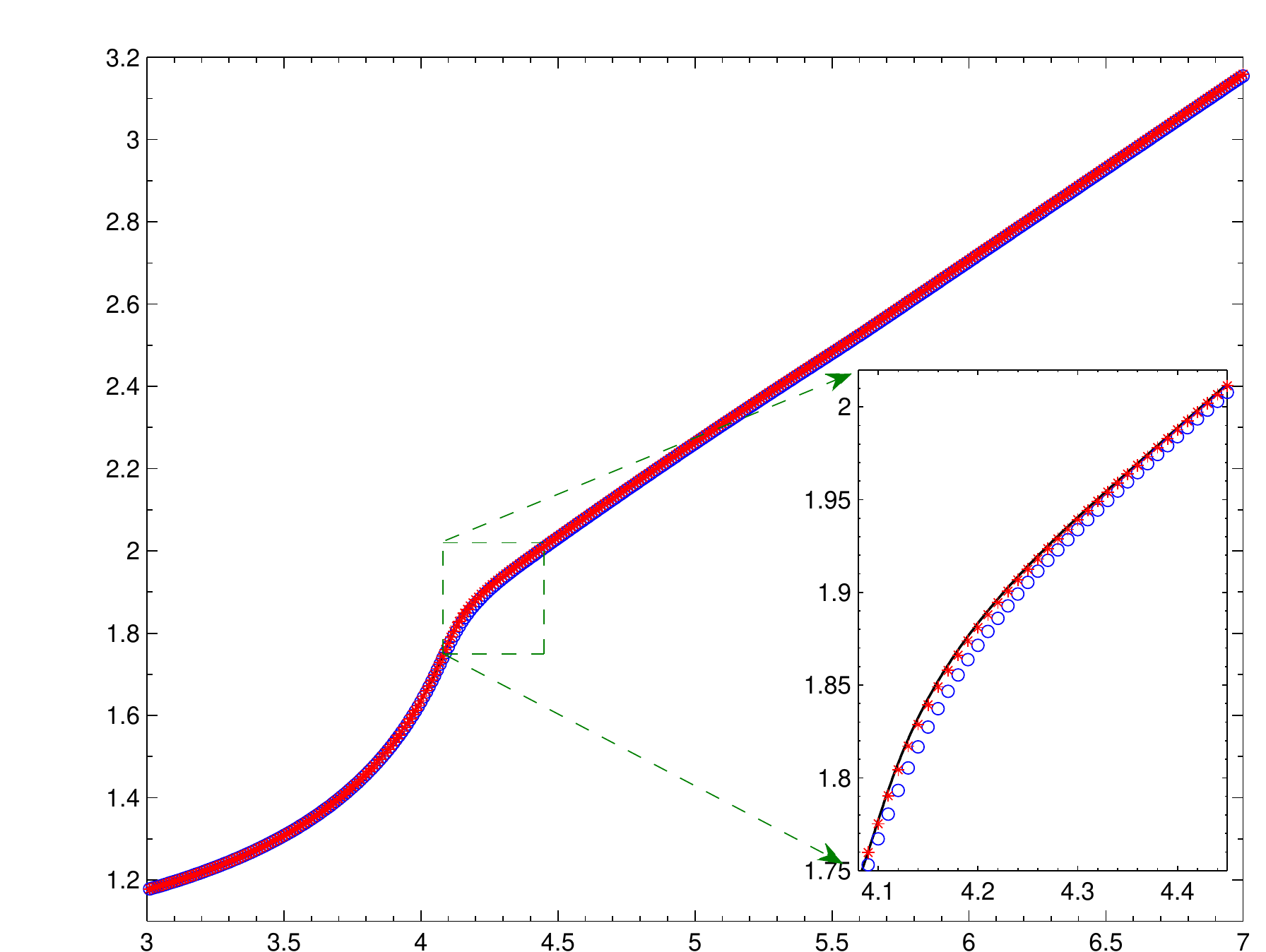}
    \caption{Same as Fig. \ref{fig:shockwave_rhov0.2}, except for Example \ref{example:TimeRev} and the metric functions $A$ (left) and $B$ (right) at $t=t_0+1$.}
    \label{fig:timerev_AB}
  \end{figure}

  \begin{figure}[!htbp]
    \centering
        \includegraphics[height=0.38\textwidth]{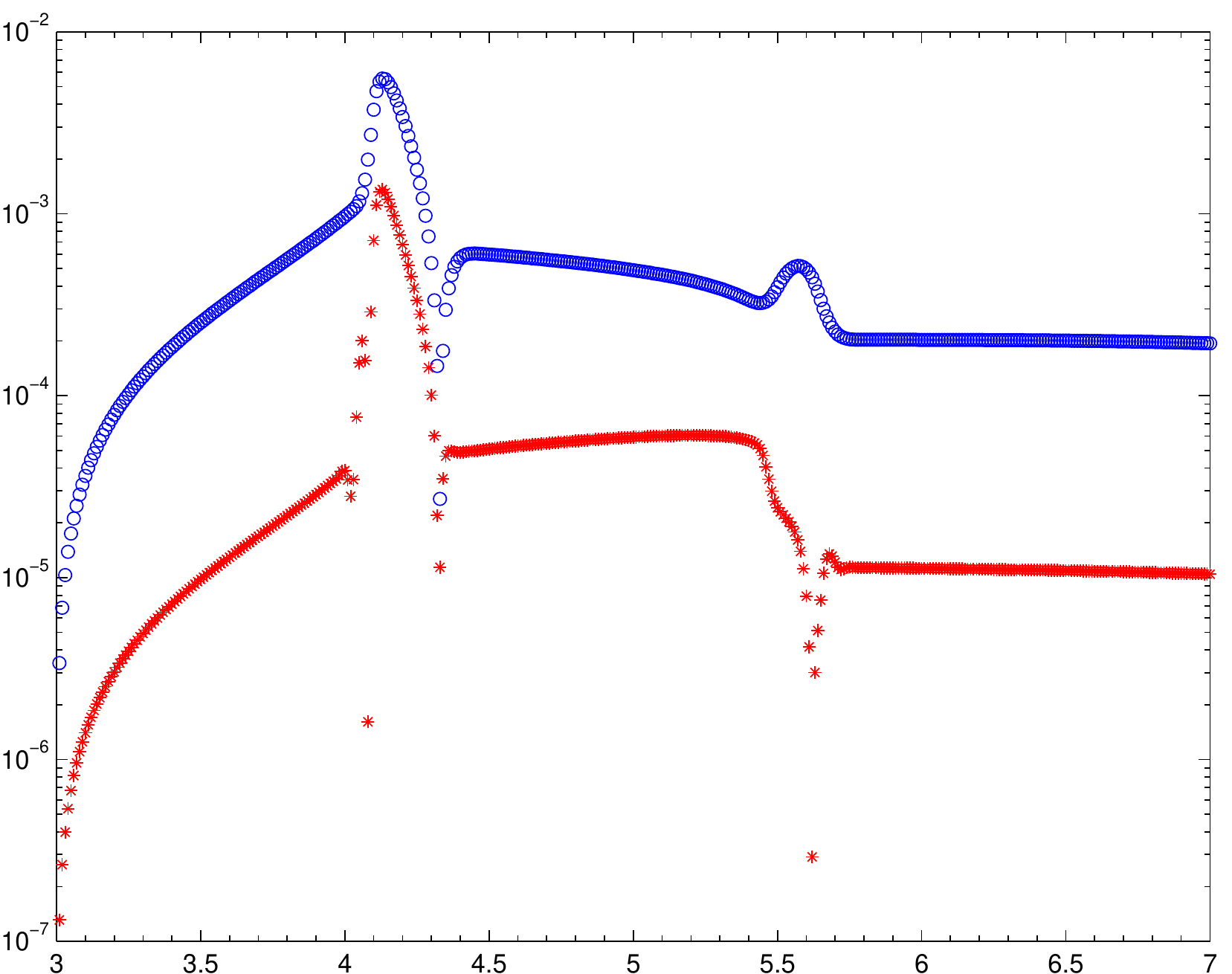}
        \includegraphics[height=0.38\textwidth]{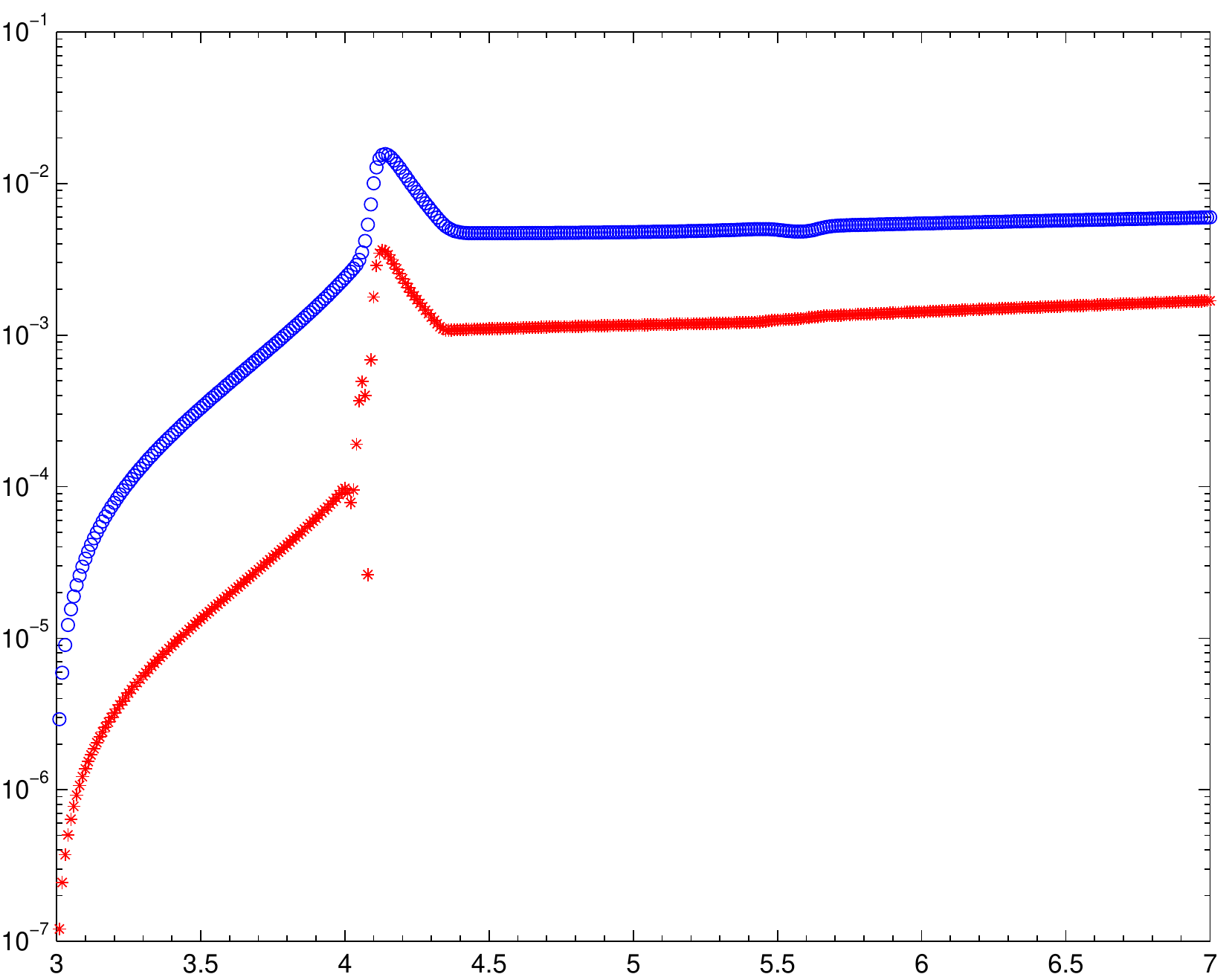}
    \caption{Same as Fig. \ref{fig:shockwave_ErrAB}, except for Example \ref{example:TimeRev}.}
    \label{fig:timerev_ErrAB}
  \end{figure}

Figs. \ref{fig:timerev_AB} and \ref{fig:timerev_ErrAB}  display the numerical solutions and
the corresponding errors in the metric functions $A$ and $B$
at $t=t_0+1$ by using the Godunov scheme and the GRP scheme with  400 uniform cells, respectively. These results show that
the GRP scheme is much more accurate than the Godunov scheme, and
the errors given by the GRP scheme is about ten percent of those computed by the Godunov scheme.

  \begin{figure}[!htbp]
    \centering
        \includegraphics[height=0.38\textwidth]{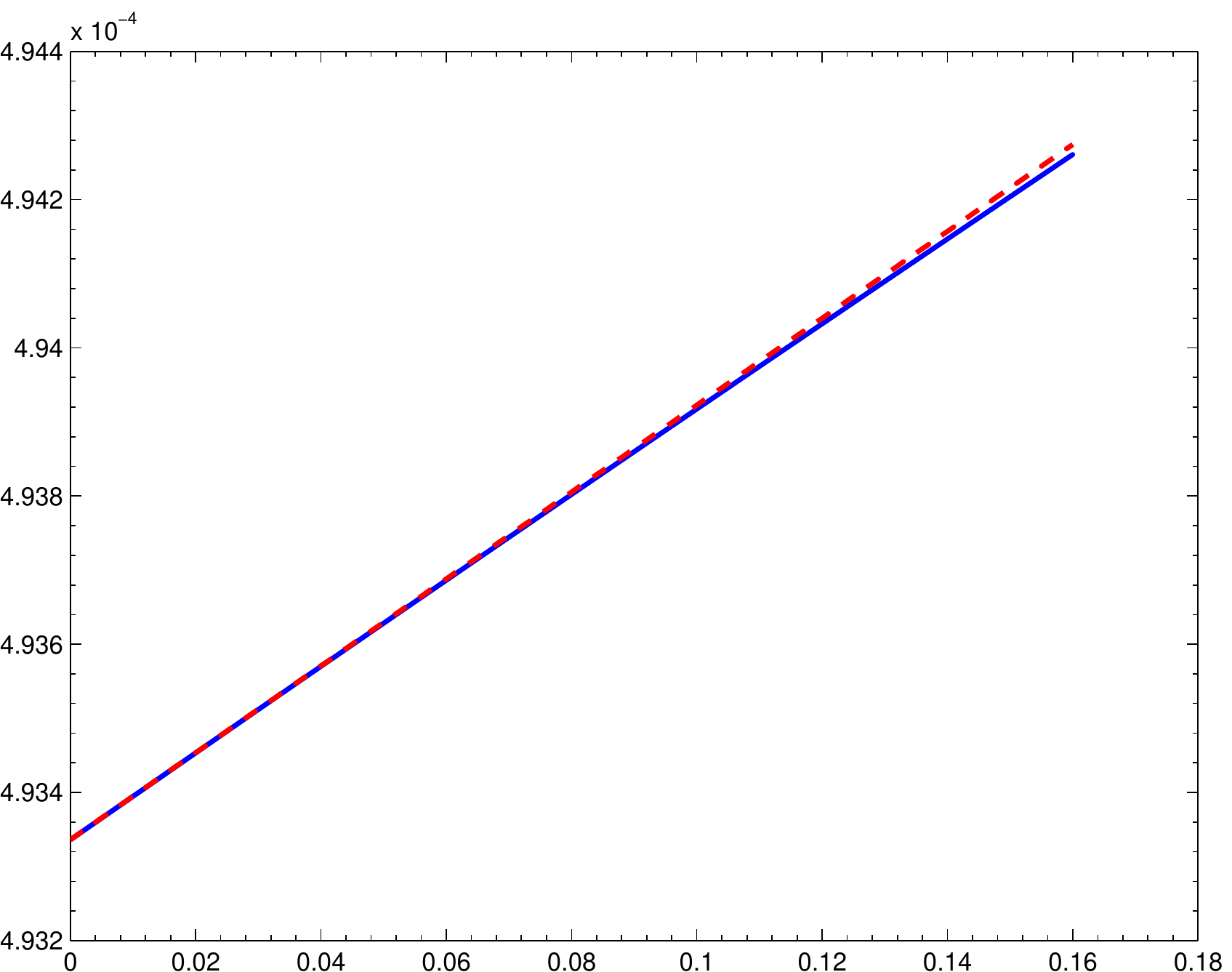}
        \includegraphics[height=0.38\textwidth]{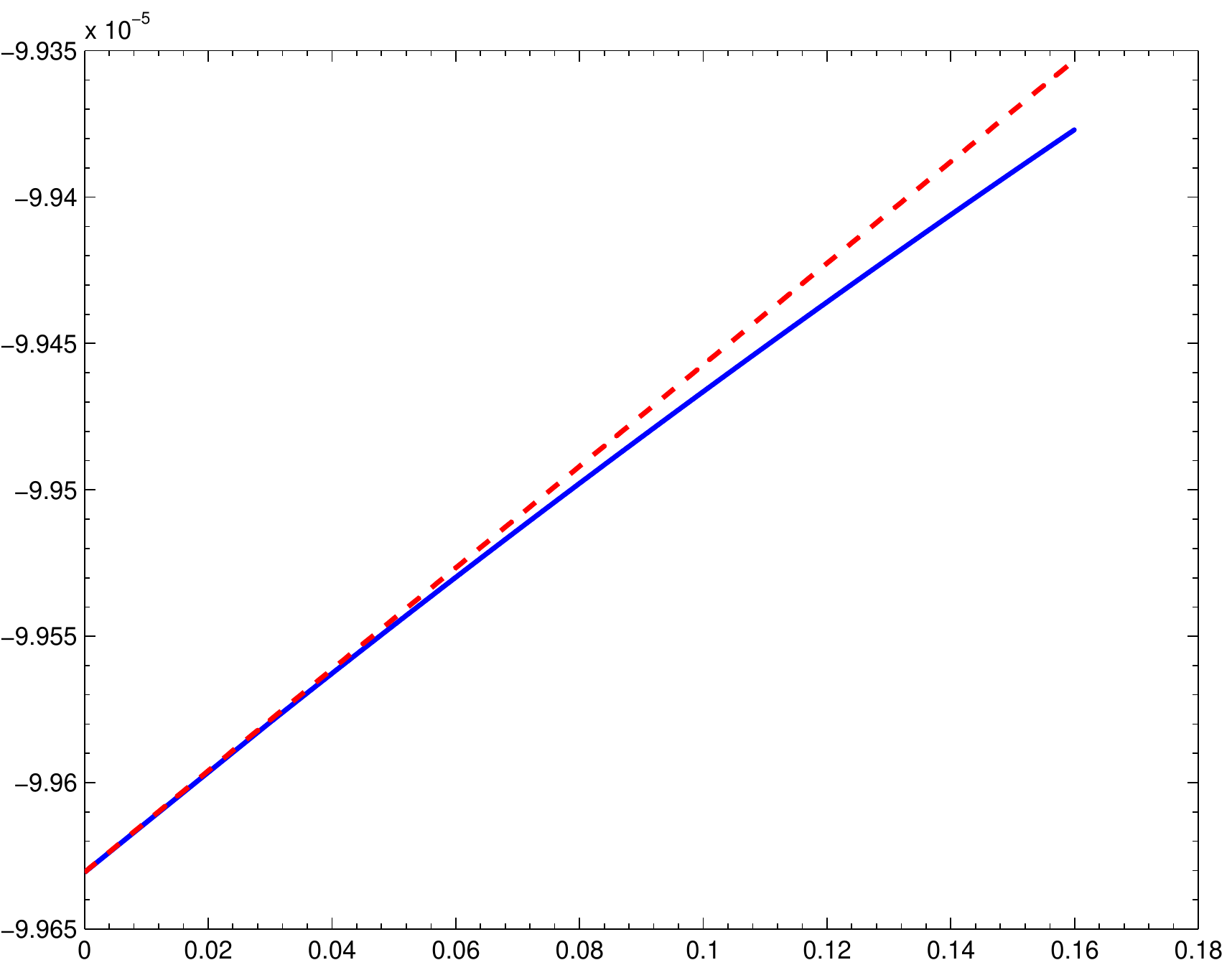}
    \caption{Same as Fig. \ref{fig:shockwave_GRPsolu}, except for Example \ref{example:TimeRev}.}
    \label{fig:timerev_GRPsolu}
  \end{figure}

\begin{table}[htbp]
  \centering
    \caption{\small Example \ref{example:TimeRev}: Numerical errors and corresponding convergence rates of the GRP solver based solution.
  }
\begin{tabular}{|c||c|c|c|c|c|c|c|}
  \hline
  {$\tau$}
 & 0.16  &0.14   & 0.12  & 0.1  & 0.08  &0.06  & 0.04  \\
 \hline
{$e_{\mbox{\tiny GRP}} (\tau)$} &
 2.74e-8 &2.09e-8&1.53e-8& 1.06e-8 &6.75e-9 &3.77e-9 & 1.66e-9               \\
order &--&  2.03 & 2.02 & 2.02&2.02&2.02&2.02                       \\
\hline
\end{tabular}\label{tab:GRPsolver2}
\end{table}

This reversing time problem is still a GRP. The GRP solver based solutions $\vec U^{\mbox{\tiny GRP}}({t_0} + \tau,r_0 )$ are plotted in dash lines in Fig. \ref{fig:timerev_GRPsolu} in comparison to
the   reference solutions given by a second-order accurate MUSCL method with the Godunov flux on a very fine uniform mesh with the spatial
step-size of $10^{-6}$.
The data in Table \ref{tab:GRPsolver2} show  that second-order rates of convergence can be obtained for the proposed GRP solver and thus  the accuracy of our GRP solver is  validated.

\end{example}

\section{Conclusions}\label{sec:conclude}

The paper developed a second-order accurate direct Eulerian generalized Riemann problem (GRP) scheme for the spherically
symmetric general relativistic hydrodynamical (RHD) equations and
a second-order accurate discretization for the  spherically
symmetric Einstein (SSE) equations.
Different from the resulting Godunov-type schemes based on exact or approximate Riemann solvers in general relativistic case,
the GRP scheme  could not be directly obtained  by a local change of coordinates from the special relativistic case.
In the GRP scheme,
the Riemann invariants and the Runkine-Hugoniot jump conditions were directly used to analytically resolve the left and right nonlinear waves of the local GRP in the Eulerian formulation together with the local change of the metrics to obtain the limiting values of  the time derivatives of the conservative variables along the cell interface and the numerical flux for the GRP scheme.
Comparing to the GRP schemes for special RHDs \cite{Yang-He-Tang_GRP-RHD1D},
the derivation of the present GRP scheme is more technique.
%
%Two main ingredients, the Riemann invariants and the
%Rankine-Hugoniot jump conditions, have been directly used to resolve the local GRP in the Eulerian formulation with the local change of the metrics taken into account.
%Based on the approximate energy-momentum tensor given by the GRP solver, the approximate metrics have been obtained by a second-order accurate scheme for the SSE equations.
The energy-momentum tensor obtained in the GRP solver was  utilized to evaluate the fluid variables in the  SSE equations and  the continuity of the metrics was constrained at the cell interfaces.
%This scheme uses
%the cell interface metrics to easily meet the continuity
%consideration of the metrics at the cell interfaces.
Several numerical examples were
 presented to demonstrate the accuracy and effectiveness of the proposed GRP scheme, in comparison with the first-order accurate Godunov scheme.

\section*{Acknowledgements}

This work was partially supported by
the National Natural Science Foundation
of China (Nos.  91330205 \& 11421101).

\begin{appendices}

%\section*{Appendixes}
\section{The Godunov scheme used   in Section \ref{sec:experiments}}
\label{sec:AppendixA}

Assume that the ``initial'' data at time $t_n$ are
\begin{equation*}%\label{eq:InitPLF:Godunov}
\left\{ \begin{array}{l}
 \vec U_h (t_n,r ) = \vec U_j^n, \\[3mm]
 \left( \begin{array}{c}
 A_h (t_n,r ) \\
 B_h (t_n,r ) \\
 \end{array} \right) = \dfr{{r_{j + \frac{1}{2}}  - r}}{{\Delta r}}\left( \begin{array}{c}
 A_{j - \frac{1}{2}}^n  \\
 B_{j - \frac{1}{2}}^n  \\
 \end{array} \right) + \dfr{r - r_{j - \frac{1}{2}} }{\Delta r}  \left( \begin{array}{c}
 A_{j + \frac{1}{2}}^n  \\
 B_{j + \frac{1}{2}}^n  \\
 \end{array} \right),
 \end{array} \right.
\end{equation*}
for $r \in I_j$, where $A_h (t_n,r )$ and $B_h (t_n,r )$ are continuous at cell interface.

\noindent
{\tt Step I}. Evolve the solution  $\vec U$ at  $t=t_{n+1}$ of \eqref{eq:RHD_1} by
\begin{align*}%\nonumber
\vec U_j^{n + 1}  = \vec U_j^n  &- \dfr{{\Delta t_n }}{{\Delta r}}
\left( {\sqrt {A_{j + \frac{1}{2}}^{n } B_{j + \frac{1}{2}}^{n } } \vec F\left( \vec U_{j + \frac{1}{2}}^{\mbox{\tiny RP},n} \right)
    - \sqrt {A_{j - \frac{1}{2}}^{n } B_{j - \frac{1}{2}}^{n } } \vec F\left( \vec U_{j - \frac{1}{2}}^{\mbox{\tiny RP},n}\right)} \right) \\[2mm]
    %\label{eq:evolve}
&
+ \Delta t_n
\vec S \left( {r_j ,A_{j }^{n} ,B_{j }^{n } ,\vec U_{j }^{n} } \right),
\end{align*}
where  $ A_{j }^{n} :=\frac{1}{2} \left( A_{j - \frac{1}{2}}^{n } + A_{j + \frac{1}{2}}^{n } \right),  B_{j }^{n} :=\frac{1}{2} \left( B_{j - \frac{1}{2}}^{n } + B_{j + \frac{1}{2}}^{n } \right)$, and
$\vec U_{j + \frac{1}{2}}^{\mbox{\tiny RP},n}$ is the value at $r=r_{j+\frac12}$ of the exact solution  to the
following RP
\begin{equation*}
\begin{cases}
\dfr{{\partial \vec U}}{{\partial t}} + \sqrt {A_{j + \frac{1}{2}}^n B_{j + \frac{1}{2}}^n } \dfr{{\partial \vec F(\vec U)}}{{\partial r}} = 0,    \qquad r>0,~t>t_n,\\[3mm]
 \vec  U(t_n,r)=
     \begin{cases}
     \vec U_{j}^n, & r<r_{j+\frac{1}{2}},\\
     \vec U_{j+1}^n, & r>r_{j+\frac{1}{2}},
     \end{cases}
\end{cases}
\end{equation*}

{\tt Step II}.  Calculate   $A_{j+\frac12}^{n+1}$ and $B_{j+\frac12}^{n+1}$,
following  {\tt Step V}  of the GRP scheme in Section \ref{sec:outline}.
%which  approximates $A(t_{n+1},r_{j+\frac12})$ and $B(t_{n+1},r_{j+\frac12})$, respectively.
\end{appendices}

%\bibliography{Ref3OrderGRP}

\begin{thebibliography}{99}




%\bibitem{Guzman2012}
%F.S. Guzm\' an, F.D. Lora-Clavijo, and M.D. Morales.  Revisiting spherically
%symmetric relativistic hydrodynamics, {\em Revista Mexicana de F\' isica E}, 58(2012), 84-98.
%
%
%
%
%\bibitem{Hawley1984}
%J.F. Hawley, L.L. Smarr and J.R. Wilson. A nuemrical study of
%nonspherical black hole accretion I: equations and test problems, {\em The Astrophysical Journal}, 277(1984), 296-311.
%
%
%
%





\bibitem{Ben-Falcovitz:1984}
M.~Ben-Artzi and J.~Falcovitz,
  A second-order {G}odunov-type scheme for compressible fluid dynamics,
  {\em J. Comput. Phys.}, 55 (1984), 1-32.

\bibitem{Ben-Falcovitz:book}
M.~Ben-Artzi and J.~Falcovitz,
 {\em Generalized Riemann Problems in Computational Fluid Dynamics},
  Cambridge University Press, 2003.

\bibitem{Ben-Li:RI:GRP}
M.~Ben-Artzi and J.Q. Li,
Hyperbolic balance laws: {R}iemann invariants and the generalized
  {R}iemann problem,
{\em Numer. Math.}, 106 (2007), 369-425.

\bibitem{Ben-Li-Warnecke}
M.~Ben-Artzi, J.Q. Li, and G.~Warnecke,
 A direct {E}ulerian {GRP} scheme for compressible fluid flows,
  {\em J. Comput. Phys.}, 218 (2006), 19-43.



\bibitem{Font2008}
J.A. Font,
  Numerical hydrodynamics and magnetohydrodynamics in general
  relativity,
  {\em Living Rev. Relativity}, 11 (2008), 7.









%\bibitem{GroahTemple2004}
%J. Groah and B. Temple, Shock wave solutions of the Einstein equations with perfect fluid sources:
%existence and consistency by a locally inertial Glimm Scheme. {\em Mem. Amer. Math. Soc.}, 172 (2004), 1-84.


\bibitem{Gourgoulhon1991}
E. Gourgoulhon, Simple equations for general relativistic hydrodynamics in spherical symmetry applied to neutron star collapse,
{\em Astron. Astrophys.}, 252 (1991), 651-663.





\bibitem{GroahTemple2007}
J. Groah, J. Smoller, and B.Temple, {\em Shock wave interactions in general relativity}, Springer Monographs in
Mathematics, Springer, New York, 2007.



\bibitem{Guzman2012}
F.S. Guzm\'an, F.D. Lora-Clavijo, and M.D. {Morales,}  Revisiting spherically
symmetric relativistic hydrodynamics, {\em Rev. Mex. Fis. E}, 58 (2012), 84-98.





\bibitem{Han-Li-TangJCP}
E.~Han, J.Q. Li, and H.Z. Tang,
\newblock An adaptive {GRP} scheme for compressible fluid flows,
\newblock {\em J. Comput. Phys.}, 229 (2010), 1448-1466.




\bibitem{Han-Li-TangCiCP}
E. Han, J.Q. Li, and H.Z. Tang,
Accuracy of the adaptive GRP scheme and the simulation of 2-D
Riemann problems for compressible Euler equations, {\em Commun. Comput. Phys.}, 10 (2011), 577-606.


\bibitem{Landau1987}
L.D. Landau and E.M. Lifschitz,
{\em Fluid Meshanics}, Pergaman Press, 1987.



\bibitem{Li-GXChen}
J.Q. Li and G.X. Chen,
\newblock The generalized {R}iemann problem method for the shallow water
  equations with bottom topography,
\newblock {\em Int. J. Numer. Meth. in Eng.}, 65 (2006), 834-862.

\bibitem{LiLiXu2011}
J.Q. Li, Q.B. Li, and K. Xu, Comparisons of the generalized Riemann solver and the gas-kinetic scheme for inviscid compressible flow simulations,
{\em J. Comput. Phys.}, 230 (2011), 5080-5099.

\bibitem{LiLiuSun2009}
J.Q. Li, T.G. Liu, and Z.F. Sun, Implementation of the GRP scheme for computing radially symmetric compressible fluid flows,
{\em J. Comput. Phys.}, 228 (2009), {5867-5887}.




\bibitem{LiYu1985}
T. Li and W.C. Yu, {\em Boundary Value Problem for Quasilinear Hyperbolic Systems},
Mathematics Department, Duke University, 1985.


\bibitem{LiZhang2013}
J.Q. Li and Y.J. Zhang,
The adaptive GRP scheme for compressible fluid flows over unstructured meshes, {\em J. Comput. Phys.}, 242 (2013), 367-386.




\bibitem{Liebendorfer2002}
M. Liebend\"oerfer, S. Rosswog, and F.-K. Thielemann,
An adaptive grid, implicit code for spherically symmetric, general relativistic hydrodynamics in comoving coordinates,
{\em Astrophys. J. Suppl.}, 141 (2002), 229-246.

\bibitem{Liebendorfer2004}
M. Liebend\"orfer, O.E.B. Messer, A. Mezzacappa, S.W. Bruenn, C.Y. Cardall, and F.K. Thielemann,
A Finite difference representation of neutrino radiation hydrodynamics for spherically symmetric general relativistic supernova simulations,
{\em Astrophys. J. Suppl.},  150 (2004), 263-316.






\bibitem{MME:2003}
J.M. Mart{\'i} and E.~M{\"u}ller,
\newblock Numerical hydrodynamics in special relativity,
\newblock {\em Living Rev. Relativity}, 6 (2003), 7.





\bibitem{May1966}
M.M. May and R.H. White, Hydrodynamic calculations of general relativistic collapse,
{\em Phys. Rev. D}, 141 (1996), 1232-1241.


\bibitem{May1967}
M.M. May and R.H. White, Stellar dynamics and gravitational collapse, {\em Methods Comput. Phys.}, 7 (1967), 219-258.



\bibitem{OConnor2010}
E. O'Connor and C.D. Ott, A new open-source code for spherically symmetric stellar collapse to neutron stars and black holes,
{\em Class. Quantum Grav.}, 27 (2010), 114103.


\bibitem{Park2013}
D.H. Park, I. Cho, G. Kang, and H.M. Lee,
A fully general relativistic numerical simulation code for spherically
symmetric matter, {\em J. Korean Phys. Soc.}, 62 (2013), 393-405.




\bibitem{Pons1998}
J.A. Pons, J.A. Font, J.M. Ib\'a\~nez, J.M. Mart\'i, and J.A. Miralles,
General relativistic hydrodynamics with special relativistic Riemann solvers,
{\em Astron. Astrophys.}, 339 (1998), 638-642.



\bibitem{QianLiWang2013}
J.Z. Qian, J.Q. Li, and S.H. Wang,
The generalized Riemann problems for compressible fluid
flows: Towards high order, {\em J. Comput. Phys.}, 259 (2014), 358-389.



\bibitem{Radice2011}
D. Radice and L. Rezzolla,
Discontinuous Galerkin methods for general-relativistic hydrodynamics: Formulation
and application to spherically symmetric spacetimes, {\em Phys. Rev. D.},  84 (2011), 024010.


\bibitem{RomeroIbanez1996}
J.V. Romero, J.M. Ib\'a\~nez, J.M. Mart\'i, andJ.A. {Miralles,}  A new
spherically symmetric general relativistic hydrodynamical code, {\em Astrophys. J}, 46 (1996), 839-854.




\bibitem{SmollerTemple1993}
J. Smoller and B. Temple, Global solutions of the relativistic Euler equations, {\em Comm. Math.
Phys.}, 157 (1993), 67-99.


\bibitem{TTang:2003}
H.Z. Tang and T.~Tang,
\newblock Adaptive mesh methods for one- and two-dimensional hyperbolic
  conservation laws,
\newblock {\em SIAM J. Numer. Anal.}, 41 (2003), 487-515.




\bibitem{TempleSmoller2009}
B. Temple and J. Smoller, Expanding wave solutions of the Einstein equations that induce
an anomalous acceleration into the Standard Model of Cosmology, {\em Proc. Natl Acad. Sci.},
106 (2009), 14213-14218.















\bibitem{Vogler2010}
Z. Vogler, {\em The numerical simulation of general relativistic shock waves by a locally
inertial Godunov method featuring dynamic time dilation}, Ph.D. thesis, University
of California, 2010.
%Doctoral Dissertation,, Davis, USA.


\bibitem{VoglerTemple2012}
Z. Vogler and B. Temple, Simulation of general relativistic shock wave interactions by a
locally inertial Godunov method featuring dynamical time dilation, {\em Proc. R. Soc. A}, 468 (2012), 1865-1883.





\bibitem{Wilson:1972}
J.R. Wilson,
\newblock Numerical study of fluid flow in a {Kerr} space,
\newblock {\em Astrophys. J.}, 173 (1972), 431-438.

\bibitem{Wilson2003}
J.R. Wilson and G.J. Mathews,
\newblock {\em Relativistic Numerical Hydrodynamics},
\newblock Cambridge University Press, 2003.









\bibitem{WuTang2014}
K.L. Wu and H.Z. Tang,
\newblock Finite volume local evolution galerkin method for two-dimensional
  relativistic hydrodynamics,
\newblock {\em J. Comput. Phys.}, 256 (2014), 277-307.



\bibitem{WuTang2015}
K.L. Wu and H.Z. Tang,
\newblock High-order accurate physical-constraints-preserving finite difference WENO schemes for special relativistic hydrodynamics,
\newblock {\em J. Comput. Phys.}, 298 (2015), {539-564}.


\bibitem{WuYangTang2013}
K.L. Wu, Z.C. Yang, and H.Z. Tang,
\newblock A third-order accurate direct {Eulerian GRP} scheme for the {Euler}
  equations in gas dynamics,
{\em J. Comput. Phys.}, 264 (2014), 177-208.


\bibitem{WuYangTang2014}
K.L. Wu, Z.C. Yang, and H.Z. Tang,
A third-order accurate direct Eulerian GRP scheme for one-dimensional relativistic hydrodynamics,
{\em East Asian J. Appl. Math.}, 4 (2014), 95-131.


\bibitem{Yamada1997}
S. Yamada, An implicit Lagrangian code for spherically symmetric general relativistic hydrodynamics with an approximate Riemann solver,
{\em Astrophys. J.}, 475 (1997), 720-739.


\bibitem{Yang-He-Tang_GRP-RHD1D}
Z.C. Yang, P.~He, and H.Z. Tang,
\newblock A direct {E}ulerian {GRP} scheme for relativistic hydrodynamics:
  One-dimensional case,
\newblock {\em J. Comput. Phys.}, 230 (2011),  7964-7987.

\bibitem{Yang-Tang_GRP-RHD2D}
Z.C. Yang and H.Z. Tang,
\newblock A direct {E}ulerian {GRP} scheme for relativistic hydrodynamics:
  Two-dimensional case,
\newblock {\em J. Comput. Phys.}, 231 (2012), 2116-2139.


\end{thebibliography}
%\bibliographystyle{plain}

\end{document}